\documentclass{amsart}
\usepackage{nz_style}
\usepackage{amsaddr}

\newcount\cnti
\newcount\cntii
\title[Analogs of Dirichlet $L$-functions in chromatic homotopy theory]{Analogs of Dirichlet $L$-functions \protect\\ in chromatic homotopy theory}
\author{Ningchuan Zhang}
\subjclass[2020]{Primary 55N22, 55P42; Secondary 11S40, 55Q50.}
\address{Department of Mathematics, University of Pennsylvania, Philadelphia, PA 19104, USA}
\hemail{nczhang@sas.upenn.edu}
\date{}

\begin{document}
	\begin{abstract}
		The relation between Eisenstein series and the $J$-homomorphism is an important topic in chromatic homotopy theory at height $1$. Both sides are related to the special values of the Riemann $\zeta$-function. Number theorists have studied the twistings of the Riemann $\zeta$-function and Eisenstein series by Dirichlet characters. 
		
		Motivated by the Dirichlet equivariance of these Eisenstein series, we introduce the Dirichlet $J$-spectra in this paper. The homotopy groups of the Dirichlet $J$-spectra are related to the special values of the Dirichlet $L$-functions. Moreover, we find Brown-Comenetz duals of the Dirichlet $J$-spectra, whose formulas resemble functional equations of the corresponding Dirichlet $L$-functions. In this sense, the Dirichlet $J$-spectra we constructed are analogs of Dirichlet $L$-functions in chromatic homotopy theory.
		
		\smallskip
		\noindent \textbf{Keywords.} Dirichlet $L$-functions, chromatic homotopy theory, $J$-spectra.
	\end{abstract}
	\maketitle
	\tableofcontents

	Bernoulli numbers show up in many seemingly unrelated areas of mathematics, as observed in  \cite{Mazur_Bernoulli}. They are the special values of the Riemann $\zeta$-function at negative integers:
	\begin{equation*}\zeta(1-k)=-\frac{B_{k}}{k}.\end{equation*}
	Another two such occasions are $q$-expansions of normalized Eisenstein series in number theory:
		\begin{equation*}E_{2k}(q)=1-\frac{4k}{B_{2k}}\sum_{n\geq 1}\sigma_{2k-1}(n)q^n,\end{equation*} 
	and the images of the $J$-homomorphisms in the stable homotopy groups of spheres in algebraic topology:
		\begin{equation*}\imag(J_{4k-1})\simeq \Z/D_{2k},~D_{2k}=\text{the denominator of }B_{2k}/4k. \end{equation*}
	\noindent The connections between the congruences of the normalized Eisenstein series $E_{2k}$ and images of the $J_{4k-1}$ have been explained in \cite{Baker_Hecke_operations,topqexp,atmf,congbeta} in different ways since the invention of elliptic cohomology and topological modular forms (TMF). 
	
	Number theorists have studied the twistings of the Riemann $\zeta$-functions and Eisenstein series by Dirichlet characters. Let $\chi\colon \znx\to\Cx$ be a primitive Dirichlet character of conductor $N$. Leopoldt defined generalized Bernoulli numbers $B_{k,\chi}$ associated to $\chi$ in \cite{Leopoldt_GBN}. These numbers are algebraic numbers in $\Q(\imag \chi)$. Moreover, they are related to the special values of the Dirichlet $L$-function $L(s,\chi)$ at negative integers:
	\begin{equation*}
		L(1-k;\chi)=-\frac{B_{k,\chi}}{k}.
	\end{equation*}
	As in the classical case, $B_{k,\chi}$ appears in the $q$-expansion of $E_{k,\chi}$, the normalized Eisenstein series associated to $\chi$ when $(-1)^k=\chi(-1)$:
	\begin{equation*}
		E_k(q;\chi)=1-\frac{2k}{B_{k,\chi}}\sum_{n=1}^{\infty}\sigma_{k-1,\chi}(n)q^n.
	\end{equation*}
	
	Denote the ideal of $\Z[\chi]=\Z[\imag \chi]$ generated by the denominator of $\frac{B_{k,\chi}}{2k}$ by $\mathcal{D}_{k,\chi}$ when $(-1)^k=\chi(-1)$.\footnote{A priori, the denominator of $\frac{B_{k,\chi}}{2k}$ is not well-defined since the ring $\Z[\chi]$ is in general not a unique factorization domain and has a non-trivial unit group. But since $\Z[\chi]$ is a Dedekind domain, its fractional ideals have unique factorizations. As a result, the principal fractional ideal generated by $\frac{B_{k,\chi}}{2k}$ can be uniquely written as the difference of two actual ideals of $\Z[\chi]$. Thus the term "denominator ideal" makes sense in this context.} One may now wonder what is the object in homotopy theory that completes the analogy below:
	\begin{table}[ht]
		\centering
		\tabulinesep=1.2mm
		\begin{tabu}{ccc} \everyrow{\tabucline-} 
			\textbf{$L$-functions} & \textbf{Modular forms} & \textbf{Homotopy theory} \\
			$\zeta(1-2k)=-\frac{B_{2k}}{2k}$ & $E_{2k}\equiv 1\mod D_{2k}$&$\imag J_{4k-1}\simeq\Z/D_{2k}$\\
			$L(1-k;\chi)=-\frac{B_{k,\chi}}{k} $& $E_{k,\chi}\equiv 1\mod \mathcal{D}_{k,\chi}$& ?
		\end{tabu}
		\caption{Analogy of $L$-functions, modular forms and homotopy theory}\label{table:analogy}
	\end{table}

	In this paper, we construct analogs of Dirichlet $L$-functions in homotopy theory, called the \textbf{Dirichlet $J$-spectra}, that fit in the table above. We further compute their homotopy groups and study their properties. The relations between homotopy groups of the Dirichlet $J$-spectra and congruences of $E_{k,\chi}$ will be explained in a subsequent paper in preparation.

	The motivation for our construction of the Dirichlet $J$-spectra is the Dirichlet equivariance of the Eisenstein series $E_{k,\chi}$. This Eisenstein series is a modular form of weight $k$ and level $\Gamma_1(N)$. Moreover, it satisfies an automorphic equation \eqref{aut_eqn} for a larger congruence subgroup $\Gamma_0(N)$ that translates into a  Dirichlet equivariance with respect to the action of the quotient group $\Gamma_0(N)/\Gamma_1(N)\simeq \znx$: 
	\begin{equation*}
	E_{k,\chi}\in\hom_{\znx\textup{-rep}}(\Cbb_{\chi^{-1}},H^0(\Mell(\Gamma_1(N)),\bfo{k})).
	\end{equation*}	
	Imitating this formula, we define the Dirichlet $J$-spectrum in \Cref{con:twisted_j} by
	\begin{equation*}
	J(N)^{h\chi}=\map\left(M(\Z[\chi]),J(N)\right)^{h\znx}.
	\end{equation*}
	In this formula,
	\begin{itemize}
		\item The notation $(-)^{h\chi}$ stands for the "homotopy $\chi$-eigen-spectrum".
		\item $\Z[\chi]$ is the $\Z$-subalgebra of $\Cbb$ generated by the image of $\chi$. The character $\chi$ induces a $\znx$-action on $\Z[\chi]$ where $a\in\znx$ acts by multiplication by $\chi(a)$. 
		\item $M(\Z[\chi])$ is the Moore spectrum of $\Z[\chi]$ with a $\znx$-action such that the induced $\znx$-action on $\pi_0$ is equivalent to that on $\Z[\chi]$. The existence of such actions on the Moore spectra is non-trivial since the formation of Moore spectra is NOT functorial. In \Cref{Subsec:Moore}, we give an explicit construction of $M(\Z[\chi])$ with $\znx$-action suggested by Charles Rezk. 
		\item $J(N)$ is the "$J$-spectrum with $\mu_N$-level structure". It is defined as the homotopy pullback of the arithmetic fracture square \eqref{eqn:jn}:
		\begin{equation*}
		\begin{tikzcd}
		J(N)\rar\dar\arrow[dr, phantom, "\lrcorner", very near start]&\prod_{p}S^0_{KU/p}\left(p^{v_p(N)}\right)\dar["\text{Rationalization}"]\\
		S^0_\Q\rar["\text{Hurewicz}"]&\left(\prod_{p}S^0_{KU/p}\left(p^{v_p(N)}\right)\right)_\Q
		\end{tikzcd}
		\end{equation*}
		Here, $S^0_{KU/p}\left(p^{v}\right)=\left(\Kp\right)^{h(1+p^v\Zp)}$ is a $\zx{p^v}$-Galois extension of the $K(1)$-local sphere $S^0_{KU/p}$. The spectrum $J(N)$ is endowed with a $\znx$-action by assembling the Galois actions of $\zx{p^{v_p(N)}}$ for each prime $p\mid N$. 
		
		In particular, $J=J(1)$ is equivalent to $S^0_{KU}$, the Bousfield localization of the sphere spectrum $S^0$ at $KU$, as discussed in \cite{Bousfield_localization}. We call it the $J$-spectrum, because its Hurewicz map detects the image of the stable $J$-homomorphism.
		The details of this construction are explained in \Cref{Subsec:jn}.
	\end{itemize}
	\begin{prop*}\textup{(\ref{prop:hess})}
		There is a variant of the homotopy fixed point spectral sequence to compute $\pi_*(J(N)^{h\chi})$:
		\begin{equation*}E_2^{s,t}\simeq \ext^s_{\Z[\znx]}\left(\Z[\chi],\pi_t(J(N))\right)\Longrightarrow \pi_{t-s}\left(J(N)^{h\chi}\right). \end{equation*}
		As the $E_2$-page consists of derived $\chi$-eigenspaces of $\pi_*(J(N))$, it is appropriate to call this spectral sequence the "homotopy eigen(-spectrum) spectral sequence".
	\end{prop*}	 
	
	This computation is carried out $p$-adically. For a $p$-adic Dirichlet character $\chi\colon \znx\to \Cpx$, we construct the Dirichlet $K(1)$-local sphere $S^0_{K(1)}\left(p^{v}\right)^{h\chi}$ in a similar fashion. We show in \Cref{prop:jnchi_p_decomposition} that the $p$-completion of $J(N)^{h\chi}$ decomposes into a wedge sum of Dirichlet $K(1)$-local spheres. When $N=p>2$ or $4$, the summands in this decomposition represent elements of finite order in the $K(1)$-local Picard group, first defined in \cite{HMS_picard}. Moreover, we notice the definitions of the Dirichlet $J$-spectra and $K(1)$-local spheres depend on the group actions on the Moore spectra. In the case when $N=4$ and $p=2$, we observe in \Cref{rem:exotic_pic} that the Dirichlet $K(1)$-local spheres constructed using different group actions on the Moore spectra differ by the  the exotic element in the $K(1)$-local Picard group at $p=2$.
	
	The homotopy groups of these Dirichlet $K(1)$-local spheres are computed by a homotopy fixed point spectral sequence (HFPSS), whose $E_2$-page consists of continuous group cohomology.
	\begin{cor*}\textup{(\ref{cor:hfpss_d_k1_gp_coh})}
		Write $N=p^v\cdot N'$, where $p\nmid N'$. Then $\chi$ factors as $\chi=\chi_p\cdot \chi'$ where $\chi_p$ and $\chi'$ have conductors $p^v$ and $N'$, respectively. Then there is a HFPSS:
		\begin{equation*}
		E_2^{s,2t}=\ext_{\Zp\llb \Zpx\x \zx{N'}\rrb}^s(\Zp[\chi],\Zp^{\otimes t})
		\simeq H_c^s(\Zpx\x \zx{N'};\Zp^{\otimes t}[\chi^{-1}])\Longrightarrow \pi_{2t-s}\left(S^0_{K(1)}(p^v)^{h\chi}\right),
		\end{equation*}
		where $\Zp^{\otimes t}[\chi]$ is the representation associated to the character $\Zpx\x\zx{N'}\xrightarrow{(a,b)\mapsto \chi_p(a)\chi'(b)a^t}(\Zp[\chi])^\x$.
	\end{cor*}
	In a subsequent paper, we will relate the group cohomology $H_c^1(\Zpx\x \zx{N'};\Zp^{\otimes k}[\chi^{-1}])$ in \Cref{cor:hfpss_d_k1_gp_coh} to congruences of the $p$-adic Eisenstein series $E_{k,\chi^{-1}}$, using \Dieudonne theory of height $1$ formal groups and formal $A$-modules.
	
	Assembling the computations of homotopy groups of the Dirichlet $K(1)$-local spheres in the first three subsections of \Cref{Sec:pi_Dirichlet_j}, we record the homotopy groups of the Dirichlet $J$-spectra in \Cref{thm:pi_dirichlet_j}.  These homotopy groups are related to the special values of the corresponding Dirichlet $L$-functions:
	\begin{thm*}\textup{(\ref{thm:dirichlet_j_gbn})}
		Assume $N=p^v>1$. For all integers $k$ satisfying $(-1)^k=\chi(-1)$, we have
		\begin{equation*}
		\pi_{2k-1}\left(J(p^v)^{h\chi}\left[\frac{1}{\ell(\chi)}\right]\right)\simeq \left.\Z\left[\chi\right]\right/\mathcal{I}_{|k|,\chi^{-1}},\quad \textup{where }\ell(\chi)=\left\{\begin{array}{cl}
		\ell,&\textup{if }|\imag(\chi)|\textup{ is a power of a prime }\ell\neq p;\\
		1,&\textup{otherwise},
		\end{array}\right.
		\end{equation*}
		and the possible multiplicative difference of the ideals $\mathcal{I}_{k,\chi}$ and $\mathcal{D}_{k,\chi}$ of $\Z[\chi]$ contains the principal ideal $(2)$ in $\Z[\chi]$.
	\end{thm*}
	This computation of Dirichlet $J$-spectra allows us to compare the spectrum $J(N)$ with the Dedekind $\zeta$-function attached to $\Q(\zeta_N)$. The comparison does not work directly, as the latter has only zero special values. Instead, we focus on \emph{totally real} abelian extensions of $\Q$. 
	\begin{thm*}\textup{(\ref{thm:J_K_Dedekind})}
			Let $\Kbb/\Q$ be a totally real finite abelian extension and $p^v$ be the smallest integer such that $\Kbb\subseteq \Q(\zeta_{p^v})$. Denote the Galois group $\gal(\Q(\zeta_{p^v})/\Kbb)$ by $G$. Then 
			\begin{equation*}
			\pi_{4t-1}\left(J(p^v)^{hG}\left[\frac{1}{|G|}\right]\right)=\left.\Z\left[\frac{1}{|G|}\right]\right/D_{\Kbb,2t},
			\end{equation*}
			where $D_{\Kbb,2t}\in \Z_{>0}$ is the denominator of $\zeta_{\Kbb}(1-2t)$. 
	\end{thm*}
	Special values of $\zeta_\Kbb$ are closely related to the algebraic $K$-theory of $\mathcal{O}_\Kbb$, the ring of integers of $\Kbb$. The precise formula of this connection is given by the Lichtenbaum-Quillen Conjecture, which is proved by Voevodsky-Rost.  From Dwyer-Mitchell's computation of the of $\Kp$-homology groups of algebraic $K$-theory spectrum of $\Ocal_\Kbb$ in \cite{Dwyer-Mitchell_alg_int}, we define a $J$-spectrum for $\Kbb$ as the homotopy pullback:
	\begin{equation*}
		\begin{tikzcd}
			J(\Kbb)\rar\dar\arrow[dr, phantom, "\lrcorner", very near start]&\prod_{p}\left(\Kp\right)^{hG_p(\Kbb)}\dar["L_\Q"]\\
			S^0_\Q\rar["h_\Q"']&\left(\prod_{p}\left(\Kp\right)^{hG_p(\Kbb)}\right)_\Q
		\end{tikzcd},
	\end{equation*}
	where $G_p(\Kbb)$ is the Galois group of $\Kbb(\mu_{p^\infty})$ over $\Kbb$. Recent work of Bhatt-Clausen-Mathew in \cite{BCM_rmk_k1_alg_k} shows $J$-spectra and algebraic $K$-theory spectra are related by:
	\begin{thm*}\textup{(\ref{thm:h_K})} $J(\Kbb)$ is a $KU$-local $\einf$-ring spectrum with a $\gal(\Kbb/\Q)$-action by $\einf$-ring maps. There is a $\gal(\Kbb/\Q)$-equivariant map $h(\Kbb)\colon  J(\Kbb)\to L_{KU}K(\Ocal_\Kbb)$ of $KU$-local $\einf$-ring spectra such that
		\begin{enumerate}
			\item $h(\Q)$ is the $KU$-local Hurewicz map $h_{KU}:S^0_{KU}\to L_{KU}K(\Z)$.
			\item The map $h(\Kbb)$ is natural with respect to field extensions.
			\item The induced maps of $h(\Kbb)$ on even degree $\Kp$-homology groups are isomorphisms. (The induced maps on odd-degree $\Kp$-homology groups are zero since $(\Kp)_*(J(\Kbb))$ is zero in odd degrees.)
			\item When $t>1$, the image of $\pi_{2t-1}(h(\Kbb))$ is the Harris-Segal summand in $\pi_{2t-1}L_{KU}K(\Ocal_\Kbb)\simeq K_{2t-1}(\Ocal_\Kbb)$.
		\end{enumerate}
	\end{thm*}
	As a result, our construction of the $J(\Kbb)$ is a global version of Kahn's spectral lifting of the Harris-Segal summands in \cite{Kahn_Bott_elements}. As a $K(1)$-local spectrum is determined by the Adams operations on its $\Kp$-homology groups \cite{Bousfield_K_local}, the map $h(\Kbb)$ above exhibits the $J$-spectrum $J(\Kbb)$ as the "even half" of the $KU$-local algebraic $K$-theory spectrum $L_{KU}K(\Ocal_{\Kbb})$.
	 
	Moreover, we find the Brown-Comenetz duals of the Dirichlet $J$-spectra and $K(1)$-local spheres in \Cref{subsec:BC_dual}. 
	\begin{defn*}
		Let $I$ be the spectrum that represents the cohomology theory $X\mapsto \hom_{\Z}(\pi_0(X),\Q/\Z)$. The mapping spectrum $\map(X,I)$ is called the Brown-Comenetz dual of a spectrum$X$, denoted by $IX$. The $KU$-local Brown-Comenetz dual of a spectrum $X$ is defined to be $I_{KU}X=\map(X,I_{KU})$, where $I_{KU}$ is the Brown-Comenetz dual of the $KU$-local sphere $L_{KU}$.
	\end{defn*}
	\begin{thm*}\textup{(\ref{thm:BC_dual_D_J})}
		Let $p$ be an odd prime. Then we have an equivalence of spectra:
		\begin{align*}
			I_{KU}\left(J(p^v)^{h\chi}\left[\frac{1}{\ell(\chi)}\right]\right)\simeq \Sigma^2J(p^v)^{h\chi^{-1}}\left[\frac{1}{\ell(\chi)}\right],\end{align*}
		which implies an isomorphism of homotopy groups:
		 \begin{equation*}
		 	\pi_t\left(J(p^v)^{h\chi}\left[\frac{1}{\ell(\chi)}\right]\right)\simeq \pi_{-2-t}\left(J(p^v)^{h\chi^{-1}}\left[\frac{1}{\ell(\chi)}\right]\right).
		 \end{equation*}
	\end{thm*}
	This duality phenomenon resembles functional equations of the corresponding Dirichlet $L$-functions that relates the values of $L(k;\chi)$ and $L(1-k,\chi^{-1})$:
	\begin{equation*}
	L(k;\chi)=\frac{\tau(\chi)}{2(k-1)!}\cdot\left(\frac{2\pi i}{N}\right)^k\cdot L(1-k;\chi^{-1}), \text{ where }\tau(\chi)=\sum_{a=1}^N\chi(a)e^{\frac{2\pi i a}{N}}.
	\end{equation*}

	We also find the Brown-Comenetz dual of $J(N)$. The formula resembles the functional equation of the corresponding Dedekind $\zeta$-function.
	\begin{prop*}\textup{(\ref{prop:jn_dual}, \ref{cor:J_4N_BC_dual})} The $KU$-local Brown-Comenetz dual of $J(N)$ is equivalent to $\Sigma^2\Ecal_{KU}\wedge J(N)\wedge M(\wh{\Z})$, where $\Ecal$ is the finite CW-spectrum
		\begin{equation*}
			\Ecal=\Sigma^{-2}(S^{-1}\cup_2 e^0\cup_{\eta}e^2),
		\end{equation*}
		and $M(\wh{\Z})$ is the Moore spectrum of the group of profinite integers $\wh{\Z}$. In particular when $N=4n$ is divisible by $4$, we have an equivalence of spectra:
		\begin{equation*}
			{I_{KU} J(4n)\simeq \Sigma^2J(4n)\wedge M(\wh{\Z}),}
		\end{equation*} 
		{which implies an isomorphism of homotopy groups:}
		\begin{equation*}
			{\pi_t(J(4n))^\wedge \simeq\hom_{\Z}(\pi_{-2-t}(J(4n)),\Q/\Z).}
		\end{equation*}
		Here, $(-)^\wedge$ means the profinite completion of an abelian group.
	\end{prop*}
	
	It is because of these observations that the Dirichlet $J$-spectra constructed in this paper are analogs of Dirichlet $L$-functions in chromatic homotopy theory. We end the introduction with a table of analogy between homotopy theory and $L$-functions.
	\begin{table}[ht]
		\centering
		\tabulinesep=1.2mm
		\begin{tabu}{ccc}\everyrow{\tabucline-} 
			\textbf{Chromatic Homotopy Theory}&& \textbf{$L$-functions} \\
			$J=S^0_{KU}$&& $\zeta(s)$\\
			$J(\Kbb)$&& $\zeta_{\Kbb}(s)$\\
			$J(N)^{h\chi^{-1}}$ && $L(s;\chi)$  \\
			Homotopy Groups && Denominators of Special Values\\
			Brown-Comenetz Duality &&Functional Equation
		\end{tabu}
		\caption{Comparisons of chromatic homotopy theory and $L$-functions}
	\end{table}
	\subsection*{Notations and conventions}
	\begin{itemize}
		\item Denote the \Teichmuller  character by the Greek letter $\omega$  and denote the sheaf of invariant differentials on various stacks by the boldface version of the same Greek letter $\bfomega$. 
		\item $C_n$ is the cyclic group of order $n$ and $\sigma$ is the sign representation of $C_2$.
		\item Denote the suspension spectrum $\Sigma^\infty X_+$ of a based space $X_+$ also by $X_+$. 
		\item $X_E$ and $L_E X$ are the Bousfield localization of a spectrum $X$ at a homology theory $E$. In particular, we denote $L_{E(1)}$ by $L_1$. $\Sp_E$ is the category of $E$-local spectra. 
		\item $KU$ is the topological complex $K$-theory, $KO$ is the topological real $K$-theory, and $\KR$ is Atiyah's genuine $C_2$-equivariant Real $K$-theory in \cite{Atiyah_KR}. $K(R)$ is the algebraic $K$-theory spectrum for a ring $R$.
		\item $K(1)$ is the Morava $K$-theory of height $1$ at a prime $p$. When the prime needs to be specified, we write $X_{KU/p}$ for the $K(1)$-localization of $X$.\footnote{When $p=2$, $K(1)\simeq KU/2$. When $p$ is odd, $K(1)$ and $KU/p$ are related by the Adams splitting : $KU/p\simeq \bigvee_{i=0}^{p-2} \Sigma^{2i}K(1)$. As a result, $K(1)$ and $KU/p$ are Bousfield equivalent.} 
		\item We write $S^0_p$ for the $p$-complete sphere spectrum. When $p>2$, we write $S^0_{\omega^a}$ for the $p$-complete sphere spectrum with an action of $\zx{p}$ induced by the character $\omega^a$, where $0\le a\le p-2$.
		\item $\Cp$ is the analytic completion of $\overline{\Qp}$, the algebraic closure of the rational $p$-adics.
	\end{itemize}	
			
	\subsection*{Acknowledgments} 
	I would like to thank Matt Ando for his constant guidance and support throughout this project and my graduate studies; Mark Behrens for many helpful conversations and answers to my questions, and for pointing out a very important issue that I overlooked in an earlier version of this paper; Charles Rezk for suggesting \Cref{con:integral_moore} that resolves the issue Mark Behrens has pointed out, and for explaining to me how to think about group actions on spectra; Mike Hopkins for advising me to think about the connections between $L$-functions and homotopy theory. In addition, I would also like to thank Patrick Allen, William Balderrama, Agn\`{e}s Beaudry, Eva Belmont, Sanath Devalapurkar, Elden Elmanto, David Gepner, Paul Goerss, Peter May, Lennart Meier, Mona Merling, Catherine Ray, Jay Shah, XiaoLin Danny Shi, Vesna Stojanoska, and Foling Zou for many helpful discussions and comments on this project. Finally, I would like to thank the anonymous referees for many helpful comments and suggestions on revisions. 
	\section{Dirichlet characters and modular forms}\label{Sec:number_theory}
	\subsection{Dirichlet $L$-functions and Dedekind $\zeta$-functions} Definitions and statements in this subsection are from \cite[\S1, \S2]{p_adic_L}, unless otherwise specified.
	\begin{defn}
		A multiplicative map $\chi\colon  \Z\to \Cbb$ is called a \textbf{Dirichlet character} of modulus $N$ if it is nonzero only at integers coprime to $N$ and it only depends on the residue class modulo $N$. Alternatively, a Dirichlet character is equivalent to a group homomorphism $\chi\colon  \znx\to \Cx$. A Dirichlet character $\chi\colon  \Z\to \Cbb$ of modulus $N$ is said to be \textbf{primitive} if it is not of modulus $M$ for any $M<N$. This $N$ is called the \textbf{conductor} of $\chi$. Denote the trivial Dirichlet character that maps every nonzero integer to $1$ by $\chi^0$. 
		
		The \textbf{Dirichlet $L$-function} associated to $\chi$ is defined to be the series:
		\begin{equation*}
		L(s;\chi)=\sum_{n=1}^{\infty}\frac{\chi(n)}{n^s}.
		\end{equation*}
	\end{defn}
	By definition, $L(s;\chi^0)=\zeta(s)$. Like the Riemann $\zeta$-function, $L(s;\chi)$ has a Euler factorization:
	\begin{equation*}
		L(s;\chi)=\prod_{p}(1-\chi(p)p^{-s})^{-1}.
	\end{equation*}
	As a function of $s$, $L(s,\chi)$ converges absolutely for all $s$ with $\Re(s)>0$ and non-absolutely for $\Re (s)>0$ when $\chi\neq \chi^0$.  Thus $L(s;\chi)$ defines a holomorphic function on the half plane $\Re(s)>0$ ($\Re (s)>1$ if $\chi=\chi^0$) and it admits an analytic continuation to the whole complex plane (minus $s=1$ if $\chi=\chi^0$). Just as the Riemann $\zeta$ function, $L(s;\chi)$ takes special values at negative integers. These values are related to the \textbf{generalized Bernoulli numbers}.
	\begin{defn}
		The ordinary Bernoulli numbers are defined by
		\begin{equation*}
			F(t)=\frac{te^t}{e^t-1}=\sum_{k=0}^{\infty}B_k\frac{t^k}{k!}.
		\end{equation*}
		Let $\chi$ be a Dirichlet character with conductor $N$. We define the generalized Bernoulli numbers associated to $\chi$ by setting
		\begin{equation}\label{eqn:GBN}
			F_\chi(t)=\sum_{a=1}^{N}\frac{\chi(a)te^{at}}{e^{Nt}-1}=\sum_{n=0}^{\infty}B_{k,\chi}\frac{t^k}{k!}.
		\end{equation}
	\end{defn}
	\begin{rem}
		Notice that the conductor of the trivial character $\chi^0$ is $1$. So we have $F_{\chi^0}(t)=F(t)$ and $B_{k,\chi^0}=B_k$.
	\end{rem}
	\begin{prop}
		The generalized Bernoulli number $B_{k,\chi}$ is zero unless $(-1)^k=\chi(-1)$. In particular, $B_k=0$ when $k$ is odd.
	\end{prop}
	\begin{prop}
		Let $k$ be a positive integer. For any Dirichlet character $\chi\colon \znx\to \Cx$, we have
		\begin{equation*}
			L(1-k;\chi)=-\frac{B_{k,\chi}}{k}.
		\end{equation*}
		It now follows from \eqref{eqn:GBN} that $L(1-k;\chi)\in \Q(\chi)$, where $\Q(\chi)$ is the field extension of $\Q$ by the image of $\chi$. In particular, $\zeta(1-k)\in\Q$.
	\end{prop}
	Arithmetic properties of $B_k$ and $B_{k,\chi}$ are summarized below:
	\begin{thm}[Clausen-von Staudt, von-Staudt, {\cite[Theorem B.3, B.4]{milnor_char_class}}]\label{thm:von_staudt}~
		\begin{enumerate}
			\item The denominator of $B_k$, expressed as a fraction in the lowest term is equal to the product of all primes $p$ with $(p-1)\mid 2k$.
			\item A prime divides the denominator of $\frac{B_{k}}{2k}$ if and only if it divides the denominator of $B_k$.
		\end{enumerate}
	\end{thm}	
	\begin{thm}[{\cite[Theorem 1 and 3]{Carlitz_GBN}}]\label{thm:GBN}
		Let $\chi\colon \znx\to \Cx$ be a primitive Dirichlet character of conductor $N$.
		\begin{enumerate}
			\item If $N$ is divisible by at least two distinct prime numbers, then $\frac{B_{k,\chi}}{k}$ is an algebraic integer. When $N=p^v$, the ideal of  $\Z[\chi]$ generated by the denominator of $\frac{B_{k,\chi}}{k}$ contains only prime ideal factors of $(p)$.
			\item If $N=p^v, p>2$, let $g$ be a primitive $\phi(N)$-th root of unity mod $p$. The algebraic number $\frac{B_{k,\chi}}{k}$ is integral unless $\mathfrak{p}=(p,1-\chi(g)g^k)\neq (1)$. In this case, when $v=1$, 
			\begin{equation}\label{eqn:bkchi_p}
				pB_{k,\chi}\equiv p-1 \mod \mathfrak{p}^{v_p(k)+1};
			\end{equation}
			when $v>1$,
			\begin{equation}\label{eqn:bkchi_pv}
				(1-\chi(1+p))\frac{B_{k,\chi}}{k}\equiv 1\mod \mathfrak{p}.
			\end{equation}
			\item If $N=4$, then
			\begin{equation}\label{eqn:bkchi_4}
				\frac{B_{k,\chi}}{k}\equiv \frac{k}{2}\mod 1\text{ i.e. }	\frac{B_{k,\chi}}{k}-\frac{k}{2}\in\Z[\chi].
			\end{equation}
			If $N=2^v, v>2$, then $\frac{B_{k,\chi}}{k}$ is an algebraic integer.
		\end{enumerate}
	\end{thm}
	We also define Dedekind $\zeta$-functions attached to number fields.
	\begin{defn}[{\cite[160]{Lang_ANT}}]
		Let $\Kbb$ be a number field and $\mathcal{O}_\Kbb$ be its ring of integers. We define
		\begin{equation*}
			\zeta_\Kbb(s)=\prod_{\mathfrak{p}}\frac{1}{1-|\mathcal{O}_\Kbb/\mathfrak{p}|^{-s}},
		\end{equation*}
		where $\mathfrak{p}$ ranges over all nonzero prime ideals of $\mathcal{O}_\Kbb$.
	\end{defn}
	When $\Kbb/\Q$ is a finite abelian extension, $\Kbb\subseteq \Q(\zeta_N)$ for some $N$ by the Kronecker-Weber Theorem. The Galois group  $\gal(\Q(\zeta_N)/\Kbb)$ is a subgroup of $\gal(\Q(\zeta_N)/\Q)$. The latter is isomorphic to $\znx$ by \Cref{lem:gal_Q_zeta}.
	\begin{thm}[{\cite[Theorem 4.3]{Washington_cyclotomic}}]\label{thm:Dedekind_zeta_Dirichlet_L} Let $\Kbb/\Q$ be a finite abelian extension and $N$ be the smallest integer such that $\Kbb\subseteq \Q(\zeta_N)$. Then $\zeta_\Kbb$ and Dirichlet $L$-functions are related by:
		\begin{equation*}
		\zeta_{\Kbb}(s)=\prod_{\substack{\chi\colon \znx\to\Cx\\ \gal(\Q(\zeta_N)/\Kbb)\subseteq\ker \chi}} L(s,\chi).
		\end{equation*}
	\end{thm} 
	\subsection{Eisenstein series}
	One way to study the Dirichlet $L$-functions is through modular forms, more precisely the Eisenstein series. Here, we give a brief review of the basic theory of modular forms from \cite{ataec}. 
	\begin{defn}
		A subgroup $\Gamma$ of $\SL_2(\Z)$ is called a \textbf{congruence subgroup} if it contains all matrices congruent to $NI_2$ in $\SL_2(\Z)$ for some integer $N>0$. Examples of congruence subgroups are
		\begin{itemize}
			\item  $\Gamma(N)=\left\{\begin{pmatrix}
			a&b\\c&d
			\end{pmatrix}\in\SL_2(\Z)\mid a\equiv d\equiv 1, b\equiv c\equiv 0\mod N\right\}$,
			\item  $\Gamma_0(N)=\left\{\begin{pmatrix}
			a&b\\c&d
			\end{pmatrix}\in\SL_2(\Z)\mid c\equiv 0\mod N\right\}$,
			\item  $\Gamma_1(N)=\left\{\begin{pmatrix}
			a&b\\c&d
			\end{pmatrix}\in\SL_2(\Z)\mid a\equiv d\equiv 1,c\equiv 0\mod N\right\}$.
		\end{itemize}
	Let $\Gamma\le \SL_2(\Z)$ be a congruence subgroup. The subgroup $\Gamma=\SL_2(\Z)$ when $N=1$. A modular form of level $\Gamma$ and weight $k$ is a holomorphic function over the complex upper half plane $\mathfrak{h}$ satisfying the functional equation: 
	\begin{equation}\label{mf}
	f(\gamma z)=(cz+d)^{k}f(z)\quad \text{for all }\gamma=\begin{pmatrix}
	a&b\\c&d
	\end{pmatrix}\in \Gamma\text{, } \imag z>0
	\end{equation}  
	and is holomorphic at all cusps. The space of such modular forms is denoted by $M_k(\Gamma)$, where $\Gamma$ is omitted if it is $\SL_2(\Z)$.
	\end{defn}
	Recall that the classical Eisenstein series of weight $k$ attached to a lattice $\Lambda\subseteq\Cbb$ is defined by 
	\begin{equation*}
		G_k(\Lambda)=\sum_{w\in \Lambda\backslash\{0\}}\frac{1}{w^k}.
	\end{equation*}
	This summation is absolutely convergent when $k>2$. Let $z\in \mathfrak{h}$ be a complex number in the upper half plane and denote the lattice $(z\Z\oplus\Z)\subseteq\Cbb$ by $\Lambda(z)$. Define
	 \begin{equation*}
	 	G_k(z)=G_k(\Lambda(z))=\sum_{(m,n)\neq(0,0)}\frac{1}{(mz+n)^k}.
	 \end{equation*}
	 This is a modular function of weight $k$ and level $\SL_2(\Z)$. It is easy to see $G_k(z)=0$ when $k$ is odd. As $G_{2k}(z+1)=G_{2k}(z)$ by \eqref{mf}, $G_{2k}$ is a function of $q=e^{2\pi iz}$: 
	\begin{equation*}
		G_{2k}(q)=2\zeta(2k)+\frac{(2\pi i)^{2k}}{(2k-1)!}\sum_{n=1}^{\infty}\sigma_{2k-1}(n)q^n,\text{ where }\sigma_{m}(n)=\sum_{0<d\mid n}d^m.
	\end{equation*}
	This is the \textbf{$q$-expansion} of $G_{2k}$. As $G_{2k}(q)$ is a power series of $q$, it is holomorphic at the only cusp $q=0$ and thus a modular form. Dividing $G_{2k}$ by the constant term in its $q$-expansion, we get the \textbf{normalized Eisenstein series} $E_{2k}$ of weight $2k$:
	\begin{equation*}
		E_{2k}(q)=\frac{G_{2k}(q)}{2\zeta(2k)} =1-\frac{4k}{B_{2k}}\sum_{n=1}^{\infty}\sigma_{2k-1}(n)q^n.
	\end{equation*}
	Let $\chi\colon  \znx\to \Cx$ be a primitive Dirichlet character of conductor $N$. We are now going to introduce the twisting of $G_{k}$ by $\chi$ following \cite[\S 5.1]{Hida_Eisenstein}.  
	\begin{defn}
		The Eisenstein series associated $\chi$ of weight $k$ is defined to be
		\begin{equation*}
		G_k(z;\chi)=\sum_{(m,n)\neq(0,0)}\frac{\chi^{-1}(n)}{(mNz+n)^k}.
		\end{equation*}
	\end{defn} 
	This series is nonzero only when $\chi(-1)=(-1)^k$. It is not hard to see $G_k(z;\chi)\in M_k(\Gamma_1(N))$. Moreover, it also satisfies an automorphic equation for $\gamma\in \Gamma_0(N)$:
	\begin{equation}\label{aut_eqn}
		G_k(\gamma\cdot z;\chi)=\chi(d)(cz+d)^kG_k(z;\chi),\quad\text{for }\gamma=\begin{pmatrix}
		a & b \\c & d
		\end{pmatrix}\in \Gamma_0(N).
	\end{equation}
	\begin{defn}\label{mk_chi}
		$M_k(\Gamma_0(N),\chi)=\{f\in M_k(\Gamma_1(N))\mid f\text{ satisfies \eqref{aut_eqn}}\}$. When $\chi=\chi^0$ is the trivial character, we recover $M_k(\Gamma_0(N))$.
	\end{defn}
	\begin{prop}
		Set $q=e^{2\pi i z}$ and assume $(-1)^k=\chi(-1)$. The $q$-expansion of $G_{k,\chi}$ is
		\begin{equation*}
			G_{k,\chi}(q)=2L(k;\chi^{-1})+2N^{-k}\left(\sum_{l=1}^N\chi^{-1}(l)e^{\frac{2\pi il}{N}}\right)\frac{(-2\pi i)^k}{(k-1)!}\left(\sum_{\substack{m\geq 0,n\geq 0\\(n,N)=1}}\chi(n)n^{k-1}q^{nm}\right).
		\end{equation*}
	\end{prop}
	When $\chi$ is primitive or $\chi=\chi^0$, one can use the functional equation of $L(s;\chi^{-1})$ to normalize the constant term of $G_{k,\chi}(z)$. We define
	\begin{equation}\label{eqn:qexp_twisted_eisenstein}
		E_{k,\chi}(q)=\frac{G_{k,\chi}(z;q)}{2L(k;\chi^{-1})}=1-\frac{2k}{B_{k,\chi}}\sum_{n=1}^{\infty}\sigma_{k-1,\chi}(n)q^n, \text{ where }\sigma_{m,\chi}(n)=\sum_{0<d\mid n}\chi(d)d^m.
	\end{equation}
	\begin{rem}
		$E_{2k}$ and $E_{k,\chi}$ can be expressed in terms of $z$ as:
		\begin{equation*}
		E_{2k}(z)=1+\frac{1}{z^{2k}}+\sum_{ \substack{m>0,n\neq 0,\\ (m,n)=1}}\frac{1}{(mz+n)^{2k}},\quad E_{k}(z;\chi)=1+\sum_{ \substack{m>0,n\neq 0,\\ (m,n)=1}}\frac{\chi^{-1}(n)}{(mNz+n)^{k}}.
		\end{equation*}
		It is straight forward to check from these formulas that \begin{equation*}G_{2k}(z)=2\zeta(2k)E_{2k}(z),\quad G_{k}(z;\chi)=2L(k;\chi^{-1})E_{k}(z;\chi).\end{equation*}
	\end{rem}
	\subsection{Moduli interpretations of modular forms}
	Modular forms are closely related to moduli stacks of elliptic curves with level structures over $\Cbb$.
	\begin{defns}\label{def_moduli}
		Let $\Mell$ be the moduli stack of \textbf{generalized elliptic curves} over $\Cbb$. That is, cubic curves with possible nodal singularities. Let $N$ be a positive integer. Define the following moduli stacks:
		\begin{itemize}
			\item $\Mell(\Gamma_0(N))$ is the moduli stack for the pairs $(C,H)$, where $C$ is a generalized elliptic curve and $H\subseteq C$ is a subgroup of order $N$.
			\item $\Mell(\Gamma_1(N))$ is the moduli stack for the triples $(C,H,\eta)$, where $C$ is a generalized elliptic curve, $H\subseteq C$ is a subgroup of order $N$, and $\eta:\Z/N\simto H$ is an isomorphism.
		\end{itemize}
	\end{defns}
	\begin{rem}
		When $N=1$, the moduli stack $\Mell(\Gamma)$ is the same as $\Mell$.
	\end{rem}
	\begin{prop}
		For the stacks above, denote the sheaves of invariant differentials by $\bfomega$. Then we have \begin{equation*}M_k(\Gamma)\simeq H^0(\Mell(\Gamma),\bfo{k}).\end{equation*}
	\end{prop}
	It is not hard to see the forgetful map $\Mell(\Gamma_1(N))\to \Mell(\Gamma_0(N))$ is a $\znx$-torsor: $g\in\znx\simeq \aut(\Z/N)$ acts by $(C,H,\eta)\mapsto(C,H,\eta\circ g)$. As a result, there is a natural action of $\znx$ on \begin{equation*}H^0(\Mell(\Gamma_1(N)),\bfo{k})\simeq M_k(\Gamma_1(N)).\end{equation*}
	\begin{prop}\label{Prop:Dirichlet_char_moduli}
		Let $\chi\colon\znx\to \Cx$ be a Dirichlet character.The vector space $M_k(\Gamma_0(N),\chi)$ defined in \Cref{mk_chi} is isomorphic to $\hom_{\znx\textup{-rep}}(\Cbb_{\chi^{-1}},M_k(\Gamma_1(N)))$.
	\end{prop}
	\begin{proof}
		By \Cref{mk_chi}, $M_k(\Gamma_0(N),\chi)$ is a subspace of $M_k(\Gamma_1(N))$. On the other hand, the vector space $\hom_{\znx\textup{-rep}}(\Cbb_{\chi^{-1}},M_k(\Gamma_1(N)))$ is isomorphic to the $\chi^{-1}$-eigensubspace of $M_k(\Gamma_1(N))$ under the action of $\Gamma_0(N)/\Gamma_1(N)\simeq \znx$. We need to identify the two subspaces. For that, it suffices to rephrase the automorphic equation \eqref{aut_eqn} in terms of the $\znx$-action on the moduli stack $\Mell(\Gamma_1(N))$. Consider the lattice $\Lambda(z)=z\Z\oplus\Z$. There is a triple $(C,H,\eta)$ associated to $\Lambda(z)$:
		 \begin{equation*}C=\Cbb/\Lambda(z),H=\Lambda(z/N)/\Lambda(z)\subseteq C, \eta:(\Z/N)\simto H, 1\mapsto z/N. \end{equation*}
		For $\gamma=\begin{pmatrix}
		a&b\\c&d
		\end{pmatrix}\in\Gamma_0(N)$, its actions on the lattices are: 
		\begin{align*}
			\Lambda(z)&\mapsto\Z(az+b)\oplus \Z(cz+b)=\Lambda(z),&\\
			\Lambda(z/N)&\mapsto\Z(az/N+b)\oplus \Z(cz/N+b)\equiv\Lambda(az/N)\equiv\Lambda(z/N)&\mod \Lambda(z),\\
			z/N&\mapsto az/N+b\equiv az/N&\mod	\Lambda(z).		
		\end{align*}
		Here the second line uses the facts $c\equiv 0\mod N$ and $a$ is invertible mod $N$. From this formula, the action of $\gamma$ is trivial when  $a\equiv 1\mod N$, i.e. $\gamma\in\Gamma_1(N)$. For $[\gamma]\in \Gamma_0(N)/\Gamma_1(N)\simeq \znx$, its action on the triple $(C,H,\eta)$ is: 
		\begin{equation*}
			(C,H,\eta:1\mapsto z/N)\longmapsto (C,H,\eta\circ [\gamma]: 1\mapsto a\mapsto az/N).
		\end{equation*}
		Thus for $f(z)\in M_k(\Gamma_0(N),\chi)\simeq \hom_{\znx\textup{-rep}}(\Cbb_{\chi^{-1}},M_k(\Gamma_1(N)))$, we have
		\begin{equation*}f(\gamma\cdot z)=\chi^{-1}(a)(cz+d)^kf(z)=\chi(d)(cz+d)^kf(z).\end{equation*}
	\end{proof} 

	\section{From the $J$-homomorphism to the $K(1)$-local sphere}\label{Sec:k1}
	\subsection{The $J$-homomorphism and the $e$-invariant}
	The $J$-homomorphism was originally defined by G.W. Whitehead as follows:
	\begin{defn}[{\cite[page 636]{Whitehead_j_defn}}, see also {\cite[Definition 1.1.12]{green}}]
		 Let $f\colon  S^k\to \SO(n)$ be a based map. A linear isometry of $\Rbb^n$ restricts to a boundary preserving isometry of the unit ball $D^n$. It induces a boundary-preserving map $\hat{f}\colon  S^k\x D^n\to D^n, (x,y)\mapsto (f(x))(y)$. We now extend $\hat{f}$ to a map $J(f): S^{k+n}\to S^n$ by sending the complement of $S^k\x D^n$ in $S^{k+n}$ to the base point of $S^n$. The homotopy class of $J(f)$ depends only on the homotopy class of $f$. As a result, it induces a map on homotopy groups $J_{k,n}\colon \pi_k(\SO(n))\to \pi_{n+k}(S^n)$. 
	\end{defn}
	\begin{prop}[\mbox{\cite[Theorem 1]{Whitehead_j_defn}}]
		The map $J_{k,n}$ is a group homomorphism.
	\end{prop}
 	We now give three other constructions of the $J$-homomorphisms that lead to different directions in homotopy theory.
	\begin{prop}\label{prop:j_equiv_defns}
		Consider the following constructions of the $J$-homomorphisms:
		\begin{enumerate}
			\item Loop spaces.  A linear isometry of $\Rbb^n$ extends to a based self-homeomorphism of its one-point compactification $S^n$. From this, we get a continuous based map $g_n\colon \SO(n)\to \Omega^nS^n=\map_*(S^n,S^n)$. Here, both $\SO(n)$ and $\map_*(S^n,S^n)$ are based at identity. One can check $\pi_k(\Omega^nS^n,id)\simeq \pi_k(\Omega^nS^n,*)$, where $*\in \Omega^nS^n$ is the constant map. We define  
			\begin{equation*}
				J'_{k,n}=\pi_k(g_n)\colon  \pi_k(\SO(n),id)\longrightarrow \pi_k(\Omega^nS^n,id)\simeq \pi_k(\Omega^nS^n,*)\simeq \pi_{k+n}(S^n).
			\end{equation*}
			\item Framed cobordism. Geometrically, the image of the $J$-homomorphism identifies the framed $k$-dimensional submanifolds of $S^{n+k}$ whose underlying submanifolds are $S^{k}$.  As the normal bundle of $S^k\hookrightarrow S^{n+k}$ is trivial, a framing of this embedding is equivalent a map $f\colon  S^k\to O(n)$. One can further show two framings of the embedding $S^{k}\hookrightarrow S^{n+k}$ are equivalent iff the associated maps are homotopic. Thus we get a map $J''_{k,n}\colon \pi_k(O(n))\to \pi_{n+k}(S^n)$. 
			\item Thom space. A map $f\in\pi_k(\SO(n))\simeq \pi_{k+1}(B\SO(n))$ induces a $n$-dimensional oriented vector bundle $\xi_f$ over $S^{k+1}$. The Thom space of $\xi_f$ is a two-cell complex $\Th(\xi_f)=S^{n}\cup e^{n+k+1}$. Define $J'''_{k,n}(f)$ to be the gluing map of $\Th(\xi_f)$, i.e.
			\begin{equation*}
			S^{n+k}=\partial e^{n+k+1}\xrightarrow{J'''_{k,n}(f)}S^n\xrightarrow{\qquad}\Th(\xi_f).
			\end{equation*}
		\end{enumerate}	
		The three constructions above are equivalent to the original definition up to a sign.
	\end{prop}	
	\begin{proof}
		\begin{enumerate}
			\item The isomorphism of homotopy groups upon changing base points: $\pi_k(\Omega^nS^n,id)\simeq \pi_k(\Omega^nS^n,*)$ is described in \cite[Remark 15.6.2]{Husemoller_fibre_bundle}. The identification of the two maps $J_{k,n}$ and $J'_{k,n}$ then follows from the factorization of $J_{k,n}$ in \cite[Proposition 15.6.4]{Husemoller_fibre_bundle}.  
			\item By \cite[Claim on page 349]{Kervaire_interpretation}, the two maps $J_{k,n}$ and $J''_{k,n}$ agree up to a sign.			
			\item By \cite[Lemma 10.1]{j4}, for any $f\in \pi_k(\SO(n))$, the homotopy cofiber of the maps $J_{k,n}(f)$ and $J'''_{k,n}(f)$ are homotopic. So $J_{k,n}(f)$ and $J'''_{k,n}(f)$ are equivalent up to a sign. 
		\end{enumerate} 
	\end{proof}
	It follows that the $J$-homomorphisms defined above have the same images. So we will not distinguish them in the computation. This map passes to a stable $J$-homomorphism $J_k\colon \pi_k(\SO)\to \pi_k(S^0)$ by the following:
	\begin{prop}[{\cite[Theorem 2]{Whitehead_j_defn}}]
		The $J$-homomorphisms $J_{k,n}$ are compatible under stabilization. More precisely, let $i_n\colon \SO(n)\hookrightarrow \SO(n+1)$ be the map that sends an $n\x n$ orthogonal matrix $A$ to $\begin{pmatrix}
		A&\\&1
		\end{pmatrix}$. The following diagram commutes:
		\begin{equation*}
			\begin{tikzcd}
			\pi_k(\SO(n))\rar["J_{k,n}"]\dar["\pi_k(i_n)"]&\pi_{n+k}(S^{n})\dar["\Sigma"]\\
			\pi_k(\SO(n+1))\rar["J_{k,n+1}"]&\pi_{n+k+1}(S^{n+1})
			\end{tikzcd}
		\end{equation*} 
	\end{prop}\color{black}
	\begin{defn}
		We define the stable $J$-homomorphism to be the colimit:
		\begin{equation*}J_k=\colim_{n}J_{k,n}\colon \pi_k(\SO)\longrightarrow \pi_k(S^0)\end{equation*}
	\end{defn}
	\begin{rem}
		$J_{k,n}$ stabilizes when $n>k+1$. For the stabilization of the left column, one can see it from the long exact sequence of homotopy groups attached to the fiber sequence:
		\[ \SO(n)\longrightarrow \SO(n+1)\longrightarrow S^n. \]
		The stabilization of the right column follows from the Freudenthal suspension theorem. 
	\end{rem}
	\begin{rem}
		The definitions of the $J$-homomorphism in \Cref{prop:j_equiv_defns} can be phrased stably:
		\begin{enumerate}
			\item The colimit of the maps $g_n$ in the first definition is a map $g\colon \SO\longrightarrow \Omega^\infty S^\infty$. The induced map \begin{equation*}\pi_k(g)\colon  \pi_k(\SO)\longrightarrow\pi_k\left(\Omega^\infty S^\infty\right)\simeq \pi_k(S^0)\end{equation*}
			is then the $k$-th stable $J$-homomorphism.
			\item In terms of framed cobordism, the stable homotopy group $\pi_k(S^0)$ classifies the framed-cobordism classes of $k$-dimensional manifolds with a framing on its stable normal bundle, when embedded in $\mathbb{R}^{\infty}$. A framing on the stable normal bundle of $S^k$ is then a map $f\colon S^k\to \SO$. Again if $f_1,f_2\colon S^k\to \SO$ are homotopic, then the corresponding stably framed $k$-dimensional manifolds are framed cobordant. From this point view we get the stable $J$-homomorphism $J_k\colon \pi_k(\SO)\to \pi_k(S^0)$. 
			\item $f\in \pi_k(\SO)\simeq \pi_{k+1}(B\SO)$ induces a virtual vector bundle $\xi_f$ of dimension $0$ on $S^{k+1}$. The Thom space of $\xi_f$ is a two-cell complex $\Th(\xi_f)=e^0\cup e^{k+1}$. Again, $J(f)$ is defined to be the gluing map of the stable two-cell complex $\Th(\xi_f)$.
		\end{enumerate}
	\end{rem}
	\begin{rem}
		The three definitions of the $J$-homomorphisms above lead to different directions in homotopy theory. (1) leads to the units of ring spectra, studied in \cite{units_ring_spec}. (2) is related to the work of Kervaire and Milnor in \cite{Kervaire_Milnor}. (3) leads to the computation of the image of the $J$-homomorphism by Adams in \cite{j4}, which we explain below. 
	\end{rem}
	
	Define the $e$-invariant of a stable map $f\colon S^{2k-1}\to S^0$ as below. Consider the cofiber sequence:
	\begin{equation*}
		\begin{tikzcd}
			S^0\rar&S^0\cup_{f}e^{2k}\rar &S^{2k}.
		\end{tikzcd}
	\end{equation*}
	Apply complex $K$-theory homology $KU$ to this sequence. As $KU_*$ is concentrated in even degrees, we get a short exact sequence:
	\begin{equation*}
		\begin{tikzcd}
		0\rar &KU_0(S^{0})\rar&KU_0(S^0\cup_{f}e^{2k})\rar&KU_0(S^{2k})\rar&0.
		\end{tikzcd}
	\end{equation*}
	This is not only an extension of abelian groups, but also of $KU_0 KU$-comodules. As such, this short exact sequence corresponds to an element
	\begin{equation*}
		e(f)\in \ext^1_{KU_0KU}(KU_{0}(S^{0}),KU_{0}(S^{2k})).
	\end{equation*}
	This is the \textbf{$e$-invariant} of $f\colon S^{2k-1}\to S^0$.
	\begin{rem}
		$KU_*KU$ is computed in \cite[Theorem 2.3]{Adams_Harris_Switzer}:
		\begin{equation*}
		KU_*KU\simeq \left\{f(u,v)\in \Q(\!(u,v)\!)\left\mid~ f(ht,kt)\in\Z\left[t,t^{-1},\frac{1}{hk}\right],\forall h,k\in\Z\right. \right\},
		\end{equation*}
		where $t\in KU_2(KU)$. In particular, 
		\begin{equation*}
			KU_0 KU\simeq \left\{f(w)\in\Q(\!(w)\!)\mid f(\Z)\subseteq \Z \right\}.
		\end{equation*}
	\end{rem}
	\begin{thm}[{\cite[Theorem 1.1--1.6]{j4}}]
		The image of the stable $J$-homomorphism $J_k:\pi_k(\SO)\to \pi_k(S^0)$ is described below:
		\begin{enumerate}
			\item $J_k$ is injective when $k\equiv 0,1\mod 8$.
			\item  The image of $J_{8k+3}$ is a cyclic group of order $D_{4k+2}$, the denominator of $\frac{B_{4k+2}}{8k+4}$. The image of $J_{8k-1}$ is a cyclic group of order $D_{4k}$ or $2D_{4k}$.
			\item The image of $J_{4k-1}$ in $\pi_{4k-1}(S^0)$ is a direct summand. The direct sum splitting is accomplished by  $e'\circ J_{4k-1}:\pi_{4k-1}(\SO)\twoheadrightarrow \Z/D_{2k}$. Here $e'$ is the post-composition of the $e$-invariant:	\begin{equation*}
				e'\colon \pi_{4k-1}(S^0)\xrightarrow{e}\ext^1_{KU_0KU}(KU_0(S^{0}),KU_0(S^{2k}))\twoheadrightarrow \Z/D_{2k},
			\end{equation*}  
			with the second map constructed using Chern characters.
		\end{enumerate}
	\end{thm}
	\subsection{$K$-theory and formal groups of height $1$}
	In this subsection, we will discuss the relation between complex $K$-theory and formal groups of height $1$. In the end, we identify $\ext^1_{KU_0 KU}(KU_{0}(S^{0}),KU_{0}(S^{2k}))$ with the continuous cohomology of a profinite group. References on formal groups and chromatic homotopy theory can be found in \cite{blue,coctalos,CHT}.	
	\begin{defn}
		A cohomology theory $E$ is called \textbf{complex oriented} if it is multiplicative and it satisfies the Thom isomorphism theorem for complex vector bundles. It is \textbf{even periodic} if $E_*$ is concentrated in even degrees and there is a $\beta\in E^{-2}(\mathrm{pt})$ such that $\beta$ is invertible in $E_*$.  
	\end{defn}
	\begin{prop}
		Let $E$ be a complex oriented evenly periodic cohomology theory, then
		\begin{enumerate}
			\item $E^{0}(\CPi)\simeq E^{0}\llb t\rrb$, where $t$ is the image of first Chern class of the tautological line bundle $\xi$ over $\CPi$ under the isomorphism $E^2(\CPi)\simeq E^0(\CPi)$. 
			\item Let $p_i:\CPi\x\CPi\to\CPi$ be the projection map onto the $i$-th component for $i=1,2$. Then $E^{0}(\CPi\x \CPi)\simeq E^{0}\llb t_1,t_2\rrb$, where $t_i=p_i^*c_1(\xi)$. 
			\item The tensor product of line bundles over $\CPi$ induces a \textbf{formal group} structure on $\spf E^0(\CPi)$ over $E_0=E^0$. Denote this formal group associated to a complex-oriented cohomology theory $E$ by $\Gh_E$. 
			\item $E^{0}(S^{2k})$ is identified with $\bfo{k}$, the $k$-th tensor power of the sheaf of invariant differentials on $\Gh_E$. 
		\end{enumerate}
	\end{prop}
	\begin{exmps} We provide two examples of complex oriented cohomology theories and their associated formal groups:
		\begin{enumerate}
			\item For ordinary cohomology $H$, $\Gh_H\simeq \Gah$ is the additive formal group.
			\item For complex $K$-theory, $\Gh_{KU}\simeq \Gmh$ is the multiplicative formal group.
		\end{enumerate}
	\end{exmps}
	\begin{thm}[Quillen]
		The formal group law associated to the periodic complex cobordism $MP =\displaystyle\bigvee_{i\in\Z}\Sigma^{2i}MU$ is the universal one. More precisely, the pair $(MP_0,MP_0(MP))$ classifies formal group laws and isomorphisms between them.
	\end{thm}
	As $\Gh_{MP}$ is the universal formal group, one might wonder given a formal group over a ring $R$ classified by a map $MP_0\to R$, is $MP_*(-)\otimes_{MP_0} R$ a cohomology theory? The answer is yes when the map $MP_0\to R$ satisfies certain flatness conditions. In particular, we have
	\begin{thm}[Conner-Floyd]
		Let $\theta\colon MP_0\to KU_0$ be the map that classifies $\Gmh$. Then $KU_*(X)\simeq MP_0(X)\otimes_{MP_0}KU_*$ and 
		\begin{equation*}KU_0 KU\simeq KU_0\otimes_{MP_0} MP_0(MP)\otimes_{MP_0}KU_0.\end{equation*}
	\end{thm}
	The map of Hopf algebroids $\theta\colon (MP_0,MP_0(MP))\to (KU_0,KU_0 KU)$ induces a map of comodule ext-groups:
	\begin{equation*}
		\theta_*\colon\ext^1_{MP_0 MP}(MP_{0}(S^{0}),MP_{0}(S^{2k}))\to\ext^1_{KU_0KU}(KU_{0}(S^{0}),KU_{0}(S^{2k}))
	\end{equation*}
	The $e$-invariant lives in the target and the source is on the $E_2$-page of the \textbf{Adams-Novikov spectral sequence} (ANSS):
	\begin{equation}\label{eqn:ANSS}
		E_2^{s,t}=\ext^{s}_{MP_0 MP}(MP_{0}(S^{0}),MP_{0}(S^{t}))\Longrightarrow \pi_{t-s}(S^0).
	\end{equation} 
	\begin{thm}
		The $e$-invariant map $e\colon\pi_{2k-1}(S^0)\to \ext^1_{KU_0 KU}(KU_{0}(S^{0}),KU_{0}(S^{2k}))$ factors through $\theta_*$. Moreover, $\theta_*$ is an isomorphism.
	\end{thm}
	\begin{proof}
		Replacing $KU$ by $MP$ in the definition of the $e$-invariant for a homotopy class $f\in \pi_{2k-1}(S^0)$, we get an element $e_{MP}(f)\in \ext^{1}_{MP_0 MP}(MP_0(S^{0}),MP_0(S^{2k}))$. We can see $\theta_*(e_{MP}(f))=e(f)$ is in $\ext^{1}_{KU_0 KU}(KU_0(S^{0}),KU_0(S^{2k}))$ from the following diagram:
		\begin{equation*}
			\begin{tikzcd}[column sep=small]
				{[}0\rar &MP_0(S^{0})\rar\dar["\theta"]&MP_0(S^0\cup_{f}e^{2k})\rar\dar["\theta"]&MP_0(S^{2k})\rar\dar["\theta"]&0{]}=e_{MP}(f) \dar[mapsto,"\theta_*",shift left=2 ex]\\
				{[}0\rar &KU_0(S^{0})\rar&KU_0(S^0\cup_{f}e^{2k})\rar&KU_0(S^{2k})\rar&0{]}=e(f)
			\end{tikzcd}
		\end{equation*}
		This proves the first half of the claim. The remaining follows from \cite[Theorem 5.3.7]{green}.
	\end{proof}
	\begin{rem}
		As $MP_0=MU_*$ is concentrated in even degrees, so is the $0$-line in the ANSS \eqref{eqn:ANSS}.  This means there is no contribution of the 0-line in the ANSS to $\pi_{2k-1}(S^0)$ for any $k$, yielding an edge homomorphism 
			\begin{equation*}
				\begin{tikzcd}
					\pi_{2k-1}(S^0)\rar[->>]& E_\infty^{1,2k}\rar[hook]& E_2^{1,2k}=\ext^{1}_{MP_0 MP}(MP_0(S^{0}),MP_0(S^{2k})).
				\end{tikzcd}				
			\end{equation*}	
			One can check that this edge homomorphism coincides with the map $e_{MP}$ in the proof above.
			
			Moreover by \cite[Theorem 4.3.2]{green}, $E_2^{0,t}=0$ in \eqref{eqn:ANSS} unless $t=0$. The evenness of $MP_0=MU_*$ implies there is no contribution of the $1$-line in the ANSS to even-degree stable stems. Consequently, this yields an edge homomormophism when $k>1$:
			\begin{equation*}
				\begin{tikzcd}
					\pi_{2k-2}(S^0)\rar[->>]& E_\infty^{2,2k}\rar[hook]& E_2^{2,2k}=\ext^{2}_{MP_0 MP}(MP_0(S^{0}),MP_0(S^{2k})).
				\end{tikzcd}				
			\end{equation*}	
			This edge homomorphism is the $f$-invariant introduced in \cite[page 411]{topqexp}.
	\end{rem}
	Thus, the image of the $J$-homomorphism is computed by its image under the $e$-invariant map in the $KU_0KU$-Ext groups. Completed at a prime $p$, these $\ext$-groups are identified with group cohomology.
	\begin{cor}
		As $MP_0(MP)$ classifies isomorphisms between formal group, $\spec KU_0 KU$ is isomorphic to the group scheme $\aut(\Gmh)$ over $\Z$.
	\end{cor} 
	\begin{thm}[{\cite{Hovey_Morita}}]
		Let $(A,\Gamma)$ be a Hopf algebroid.
		\begin{enumerate}
			\item $(\spec A, \spec \Gamma)$ is a groupoid scheme.
			\item There is an equivalence of abelian categories between $(A,\Gamma)$-comodules and quasicoherent sheaves over the quotient stack $\spec A//\spec\Gamma$.
		\end{enumerate}
	\end{thm}
	\begin{cor}\label{cor:BAutGm}
		The stack associated to the pair $(KU_0,KU_0KU)$ is the classifying stack \begin{equation*}B\aut (\Gmh)=\spec\Z//\aut (\Gmh).\end{equation*} As a result, the $e$-invariant lives in
		\begin{align*}
			\ext^1_{KU_0 KU}(KU_{0}(S^0),KU_{0}(S^{2k}))\simeq&R^1\hom_{\qcoh\left(B\aut (\Gmh)\right)}(\mathcal{O},\bfo{k})\\
			\simeq &H^1(B\aut (\Gmh),\bfo{k}).
		\end{align*}
	\end{cor}
	The group scheme $\aut (\Gmh)$ is not a constant group scheme over $\Z$. However, it becomes one when restricted to the closed points $\spec \Fp\in\spec \Z$. This is even true over $\spf \Zp$, the formal neighborhood of the closed point $\spec \Fp$ in $\spec \Z$.
	\begin{lem}
		Over $\Fp$ or $\Zp$, $\aut (\Gmh)\simeq \underline{\Zpx}$ as a constant pro-group scheme.
	\end{lem}
	As a result, the sheaf cohomology $H^1(B\aut (\Gmh),\bfo{k})$ reduces to group cohomology when completed at a prime $p$:
	\begin{equation}\label{eqn:e_inv}
		H^1(B\aut (\Gmh)^\wedge_{p},\bfo{k})\simeq H^1\left(B\Zpx,\bfo{k}\right)\simeq H^1_c\left(\Zpx;\left(\Kp\right)_{2k}\right),
	\end{equation}
	where $\Kp$ is the $p$-completion of the complex $K$-theory and $a\in \Zpx$ acts on $\left(\Kp\right)_{2k}$ by multiplication by $a^k$.
	
	\subsection{The homotopy fixed point spectral sequence}\label{Subsec:HFPSS}
	Let $G$ be a finite group. Recall that the group cohomology of $G$ is the derived functor of $G$-fixed points. If $G$ acts on a spectrum $E$, then there is a spectral sequence to compute homotopy groups of $E^{hG}$, the \textbf{homotopy fixed point spectrum} of $E$ under the $G$-action. The $E_2$-page of this spectral sequence is given by the group cohomology of $G$ with coefficients in $\pi_*(E)$.
	 
	\begin{defn} Let $G_+^{\wedge\bullet}\wedge E$ be the group action cosimpicial spectrum. The homotopy fixed points of this action is defined to be the totalization of this cosimplicial spectrum:
		\begin{equation*}
		E^{hG}=\map(\Sigma^\infty EG_+,E)^G\simeq \left(\tot\left[\map(G_+^\bullet, E)\right]\right)^G.
		\end{equation*}
	\end{defn}
	\noindent The Bousfield-Kan spectral sequence associated to this cosimpicial spectrum is called the \textbf{homotopy fixed point spectral sequence} (HFPSS), whose $E_2$-page is identified with
	\begin{equation}\label{eqn:HFPSS}
	E_2^{s,t}=H^s(G;\pi_t(E))\Longrightarrow \pi_{t-s}(E^{hG}).
	\end{equation}
	
	In \eqref{eqn:e_inv}, we showed that the $p$-adic $e$-invariant is in $H^1\left(\Zpx;\left(\Kp\right)_{2k}\right)$, where $\Zpx$ acts on the $p$-adic $K$-theory spectrum by the Adams operations. In \cite{fixedpt}, Devinatz and Hopkins defined $E^{hG}$ for \emph{pro-finite} groups and showed that the $E_2$-page of the associated HFPSS consists of \emph{continuous} group cohomology of $G$. Moreover, they proved  	
	\begin{thm}\label{thm:S_k1}
		Let $\Zpx$ acts on the $p$-adic $K$-theory spectrum by Adams operation. Then the homotopy fixed points $\left(\Kp\right)^{h\Zpx}$ is equivalent to $S^0_{K(1)}$, the \textbf{$K(1)$-local sphere}. Here, $S^0_{K(1)}$ is the Bousfield localization of the sphere spectrum $S^0$ at $K(1)$, the Morava $K$-theory of height $1$ and prime $p$.
	\end{thm}
	In this paper, we will study finite Galois extensions of $S^0_{K(1)}$ in the sense of \cite{galext}. 
	\begin{defn}\label{defn:s_k1_pv}
		Define $S^0_{K(1)}(p^v)$ to be the homotopy fixed point spectrum $\left(\Kp\right)^{h(1+p^v\Zp)}$ under the Adams operations. This notation was used in \cite[Definition 5.10]{Lawson_Naumann_comm}.
	\end{defn}
	The spectrum $S^0_{K(1)}(p^v)$ is a $\zx{p^v}$-Galois extension of $S_{K(1)}^0$. This means that there is a Galois correspondence between open subgroups of $\Zpx$ and finite Galois extensions of $S^0_{K(1)}$. We consider the following family of open subgroups of $\Zpx$ nested in a descending chain for $p>2$:
	\begin{equation*}
	\Zpx\supsetneq 1+p\Zp\supsetneq 1+p^2\Zp\supsetneq 1+p^3\Zp\supsetneq\cdots,
	\end{equation*} 
	and for $p=2$:
	\begin{equation*}
	\Z_2^\x=1+2\Zp\supsetneq 1+2^2\Zp\supsetneq 1+2^3\Zp\supsetneq\cdots.
	\end{equation*} 
	
	Now we are going to compute $\pi_*\left(S^0_{K(1)}(p^v)\right)$ using the HFPSS, whose $E_2$-page is
	\begin{equation}\label{HFPSS}
	E_2^{s,t}=H_c^s\left(1+p^v\Zp;\left(\Kp\right)_{t}\right)\Longrightarrow \pi_{t-s}\left(S^0_{K(1)}(p^v)\right).
	\end{equation}
	\noindent One reference for this computation (and also the HFPSS at general height $n$) is \cite{mini-course}. There are two cases in this computation:
	
	\noindent\textbf{Case I:} $p>2$ or $p=2$ and $v\ge2$. In this case, $\Zpx$ and $1+4\Z_2$ are pro-cyclic. Let $g$ be a topological generator in $\Zpx$ for $p>2$ and in $1+4\Z_2$ for $p=2$. Then for $p>2$, $1+p^v\Zp=\left\langle g^{(p-1)p^{v-1}}\right\rangle$ and for $p=2$, $1+2^v\Z_2=\left\langle g^{2^{v-2}}\right\rangle$. Let $n=1$ if $G=\Zpx$ and $n=(p-1)p^{v-1}$ if $G=1+p^v\Zp$ for $p>2$, and $n=2^{v-2}$ if $G=1+2^v\Z_2$. The minimal continuous projective resolution for $\Zp$ in $\Zp\llb G\rrb$ is 
	\begin{equation}\label{eqn:c_res}
	\begin{tikzcd}
	0\rar&\Zp\llb G\rrb\rar["1-g^{n}"]&\Zp\llb G\rrb\rar["g^{n}\mapsto 1"]&\Zp\rar&0.
	\end{tikzcd}
	\end{equation}
	Since the length of the resolution is $1$, the HFPSS collapses on $E_2$-page. The $p$-adic Adams operations on $\Kp$ realize $\left(\Kp\right)_{2t}$ as the $t$-th power representation of $G$. From this we get when $G=\Zpx$ for $p>2$: 
	\begin{align}\label{k1se2}
		H_c^s(\Zpx;\left(\Kp\right)_t)&=\left\{\begin{array}{cl}
		\Zp,&s=0,1\textup{ and }t=0;\\
		\Z/p^{v_p(k)+1},&s=1\textup{ and } t=2(p-1)k;\\
		0,&\textup{otherwise},
		\end{array}\right.\\
		\label{k1spi} \Longrightarrow\pi_{i}\left(S_{K(1)}^0\right)&=\left\{\begin{array}{cl}
		\Zp, & i=0,-1;\\
		\Z/p^{v_p(k)+1}, & i=2(p-1)k-1;\\
		0,&\textup{otherwise.}
		\end{array}\right.
	\end{align}
	When $G=1+p^v\Zp$ ($v>1$ if $p=2$), we have
	\begin{align}
		H_c^s(1+p^v\Zp;\left(\Kp\right)_t)&=\left\{\begin{array}{cl}
		\Zp, &s=0,1\textup{ and }t=0;\\
		\Z/p^{v_p(k)+v}, &s=1\textup{ and } t=2k\neq 0;\\
		0,&\textup{otherwise},
		\end{array}\right.\nonumber\\
		\label{eqn:pi_Sk1_pv}\Longrightarrow\pi_{i}\left(S_{K(1)}^0(p^v)\right)&=\left\{\begin{array}{cl}
		\Zp, & i=0,-1;\\
		\Z/p^{v_p(k)+v}, & i=2k-1\neq -1;\\
		0,&\textup{otherwise.}
		\end{array}\right.
	\end{align}
	\noindent\textbf{Case II:} $p=2$ and $G=\Z_2^\x$. In this case, $\Z_2^\x$ is not pro-cyclic. Rather, we have 
	\begin{equation*}
		\Z_2^\x\simeq \{\pm 1\}\x (1+4\Z_2).
	\end{equation*}		
	Notice $(KU_2^\wedge)^{h\Z/2}\simeq KO_2^\wedge$, where $\Z/2$ acts by complex conjugation on $KU^\wedge_2$. The homotopy groups of $KO_2^\wedge$ are given by:
	\begin{equation}\label{eqn:pi_ko}
		\begin{array}{c|c|c|c|c|c|c|c|c}
		i\mod 8& 0&1&2&3&4&5&6&7\\\hline
		\pi_i(KO_2^\wedge)&\Z_2&\Z/2&\Z/2&0&\Z_2&0&0&0
		\end{array}
	\end{equation}
	A topological generator $g\in 1+4\Z_2$ acts on $\pi_{4l}$ by multiplication by $g^{2l}$, and on $\pi_{8l+1}$ and $\pi_{8l+2}$ by identity, respectively. The $E_2$-page of the HFPSS is 
	\begin{equation}\label{eqn:hfpss_k1_sp2}
		E_2^{s,t}=H_c^s(1+4\Z_2;\pi_t(KO_2^\wedge))=\left\{\begin{array}{cl}
		\Z_2, & s=0,1 \textup{ and } t=0;\\
		\Z/2, & s=0,1\textup{ and } t\equiv 1, 2 \mod 8;\\
		\Z/2^{v_2(k)+3}, & s=1\textup{ and }t=4k\neq 0;\\
		0,& \textup{otherwise.}
		\end{array}\right.
	\end{equation}
	\begin{prop}\label{prop:extn_prob}
		The extension problems of this spectral sequence are trivial.
	\end{prop}
	\begin{proof}
		We need to solve the extension problems when $t-s=0$ or $t-s\equiv 1 \mod 8$. The following explanation is from Mark Behrens.
		
		The extension when $t-s=0$ is trivial, because there is no non-trivial extension of $\Z/2$ by $\Z_2$. 
		
		When $t-s\equiv 1\mod 8$, we recall that the Hopf element $\eta\in \pi_1\left(S^0\right)$ has order $2$. The element $\eta$ is represented in \eqref{eqn:hfpss_k1_sp2} by the non-zero element of $H^0(1+4\Z_2;\pi_1(KO^{\wedge}_2))=\Z/2$. If the extension at $t-s=1$ were nontrivial, then $\pi_1\left(S^0_{K(1)}\right)\simeq \Z/4$. From the short exact sequence
		\begin{equation*}
		\begin{tikzcd}[cramped, sep=small]
		0\rar&H_c^1(1+4\Z_2;\pi_0(KO^{\wedge}_2))\rar&\pi_1\left(S^0_{K(1)}\right)\rar&H_c^0(1+4\Z_2;\pi_1(KO^{\wedge}_2))\rar&0,
		\end{tikzcd}
		\end{equation*}
		$\eta$ would then have order $4$ in $\pi_1\left(S^0_{K(1)}\right)$. This contradicts the fact that the order of $\eta\in \pi_1(S^0)$ is $2$. 
		
		For the general $t-s=8k+1$ case, replace $\eta$ by $\beta^k\cdot\eta\in \pi_{8k+1}(KO)$ in the argument above, where $\beta\in\pi_8(KO)$ is the Bott element.
	\end{proof}
	In conclusion, we get when $p=2$, 
	\begin{equation}\label{pis0_2}
		\pi_i\left(S^0_{K(1)}\right)=\left\{\begin{array}{cl}
		\Z_2\oplus \Z/2, & i=0;\\
		\Z_2,&i=-1;\\
		\Z/2\oplus \Z/2, & i\equiv 1 \mod 8;\\
		\Z/2, & i\equiv 0, 2 \mod 8 \textup{ and } i\neq 0;\\
		\Z/2^{v_2(k)+3}, & i=4k-1\neq -1;\\
		0,& \textup{otherwise.}
		\end{array}\right.
	\end{equation}
	Alternatively, we can apply the HFPSS on $G=\Z_2^\x$ directly. The $E_2$-page is computed using the \textbf{Hochschild-Serre spectral sequence} (HSSS) whose $E_2$-page is 
	\begin{equation}\label{HSSS}
		E_2^{p,q}=H_c^p(1+4\Z_2;H^q(\Z/2;\left(KU^\wedge_2\right)_t))\Longrightarrow H_c^{p+q}(\Z_2^\x; \left(KU^\wedge_2\right)_t).
	\end{equation}
	This spectral sequence collapses on the $E_2$-page and we have
	\begin{equation*}
		H_c^s(\Z_2^\x; \left(KU^\wedge_2\right)_t)=\left\{\begin{array}{cl}
		\Z_2, & s=0, 1\textup{ and } t=0;\\
		\Z/2^{v_2(k)+3}, & s=1 \textup{ and } t=4k\neq 0;\\
		\Z/2, & s=1 \textup{ and } t=4k+2;\\
		\Z/2,& s\geq 2\textup{ and } t \text{ even};\\
		0,& \textup{otherwise.}
		\end{array}\right.
	\end{equation*}
	\section{Constructions of the Dirichlet $J$-spectra and $K(1)$-local spheres}\label{Sec:Dirichlet_J_con}
	Let $\chi\colon \znx\to\Cx$ be a primitive Dirichlet character of conductor $N$. In this section, we construct $J(N)^{h\chi}$, the Dirichlet $J$-spectrum in three steps:
	\begin{enumerate}
		\item Identify an integral model of the $J$-spectrum, a ring spectrum whose Hurewicz map detects the image of the $J$-homomorphism in $\pi_*(S^0)$.
		\item Define $J(N)$, "the $J$-spectrum with $\mu_N$-level structure" using local structures of the finite group scheme $\mu_N$ and the Hopkins-Miller theorem. The spectrum $J(N)$ is endowed with a natural $\znx$-action by assembling the $\zx{p^v}$-Galois action at each prime.
		\item Construct a Moore spectrum $M(\Z[\chi])$ with a $\znx$-action that lifts the $\znx$-action on $\Z[\chi]$ induced by $\chi$. Here $\Z[\chi]$ is the subalgebra of $\Cbb$ generated by the image of $\chi$. This construction is non-trivial, because the Moore spectrum construction is not functorial. We give an explicit construction of Moore spectra with group actions suggested by Charles Rezk.
	\end{enumerate}
	From these data, we define the Dirichlet $J$-spectrum associated to $\chi$ by
	\begin{equation*}
		J(N)^{h\chi}=\map\left(M(\Z[\chi]),J(N)\right)^{h\znx}.
	\end{equation*}
	This definition leads to a spectral sequence whose $E_2$-page consists of derived $\chi$-eigenspaces of $\pi_*(J(N))$: 
	\begin{equation*}
		E_2^{s,t}\simeq \ext^s_{\Z[\znx]}\left(\Z[\chi],\pi_t(J(N))\right)\Longrightarrow \pi_{t-s}\left(J(N)^{h\chi}\right).
	\end{equation*}
	The actual computation of $J(N)^{h\chi}$ is carried out by studying its local structures. Rationally, the Dirichlet $J$-spectra are contractible unless $\chi$ is trivial. Completed at each prime, the $J(N)^{h\chi}$ splits into a wedge sum of Dirichlet $K(1)$-local spheres. The Dirichlet $K(1)$-local spheres are constructed analogously to the Dirichlet $J$-spectra, but the $p$-adic Moore spectra with a prescribed $\znx$-action induced $\chi$ is constructed by Cooke's obstruction theory in \cite{Cooke_obstruction}. This splitting of $p$-completion of integral Moore spectra uses the uniqueness part of Cooke's obstruction theory. 
	\subsection{An integral model of the $J$-spectrum}
	In the previous section, we have explained the relations between the images of the stable $J$-homomorphisms and the $K(1)$-local spheres:
	\begin{equation*}
		\imag(J_{4k-1})^\wedge_{p}\simeq \pi_{4k-1}\left(S^0_{K(1)}\right), k>0.
	\end{equation*}
	We are now going to define an integral $J$-spectrum by assembling the $K(1)$-local spheres at each prime. As the height $1$ Morava $K$-theory $K(1)$ and mod-$p$ complex $K$-theory $KU/p$ are Bousfield equivalent, we will use $(-)_{KU/p}$ or $L_{KU/p}(-)$ to denote $K(1)$-localization at the prime $p$.
	\begin{thm}[{Adams-Baird,} {\cite[Corollary 4.5, 4.6]{Bousfield_localization}}] \label{thm:Bousfield}
		Let $J=S^0_{KU}$, the Bousfield localization of the sphere spectrum $S^0$ at complex $K$-theory.
		\begin{enumerate}
			\item The $J$-spectrum and the $KU/p$-local spheres are related by the arithmetic fracture square:
			\begin{equation}\label{eqn:J_frac_sq}
			\begin{tikzcd}
			J\rar\dar\arrow[dr, phantom, "\lrcorner", very near start]&\prod_{p}S^0_{KU/p}\dar["L_\Q"]\\
			S^0_\Q\rar["h_\Q"']&\left(\prod_{p}S^0_{KU/p}\right)_\Q
			\end{tikzcd}
			\end{equation}
			Here $h_\Q$ is the rational Hurewicz map and $L_\Q$ is the rationalization map.
 			\item Denote the denominator of $B_{2k}/4k$ by $D_{2k}$. We have:
			\begin{equation}\label{eqn:pi_j}
			\pi_i(J)=\left\{\begin{array}{cl}
			\Z\oplus \Z/2,& i=0;\\
			\Q/\Z, &i=-2;\\
			\Z/D_{|2k|}, &i=4k-1\neq -1;\\
			\Z/2\oplus \Z/2, &i\equiv 1\mod 8;\\
			\Z/2, &i\equiv 0,2\mod 8\textup{ and }i\neq 0;\\
			0, & \textup{otherwise.}
			\end{array}\right.
			\end{equation}
		\end{enumerate}
	\end{thm}
	\begin{cor}
		$J^\wedge_{p}\simeq S^0_{KU/p}$ and $J_{(p)}\simeq S^0_{E(1)}$ is the Bousfield localization of $S^0$ at $E(1)=BP\langle 1\rangle$.
	\end{cor}
	\begin{rem}
		$J=S^0_{KU}$ is an $\einf$-ring spectrum since it is a localization of an $\einf$-ring spectrum by \cite[page 155]{EKMM}.
	\end{rem}
	\begin{proof}
		\eqref{eqn:J_frac_sq} is the almost same homotopy pullback diagram for $S^0_{KU}$ as in the proof of \cite[Corollary 4.7]{Bousfield_localization}, except for the lower left corner -- the rationalization of $S^0_{KU}$ is a priori $S^0_{KU\Q}$, where $KU\Q=KU\wedge M\Q$ is the rational $K$-spectrum. Now it remains to show $KU\Q$ and $H\Q$ are Bousfield equivalent. This follows from the facts that $KU\Q$ and the periodic $HP\Q=\bigvee_i\Sigma^{2i}H\Q$ are equivalent cohomology theories via the Chern character map and that $HP\Q$ is Bousfield equivalent to $H\Q$.
		
		The computation of $\pi_*(J)$ is the integral version of that of the $\pi_*\left(S^0_{E(1)}\right)$ in \cite[Theorem 6, Lecture 35]{CHT}. The arithmetic fracture square \eqref{eqn:J_frac_sq} induces a long exact sequence of homotopy groups:
		\begin{equation*}			
		\cdots\to\pi_i(J)\longrightarrow \pi_i\left(S^0_\Q\right)\oplus\prod_{p}\pi_i\left(S^0_{KU/p}\right)\longrightarrow \left(\prod_{p}\pi_i\left(S^0_{KU/p}\right)\right)\otimes\Q\longrightarrow\pi_{i-1}(J)\to\cdots			
		\end{equation*}
		Notice that $\left(\prod_{p}\pi_i\left(S^0_{KU/p}\right)\right)\otimes\Q=0$ unless $i=0$ or $-1$ and $\pi_i\left(S^0_\Q\right)=0$ unless $i=0$, we have $\pi_i(J)\simeq \prod_{p}\pi_i\left(S^0_{KU/p}\right)$ unless $i\in\{-2,-1,0\}$. In those three cases, there is an exact sequence: 
		\begin{equation*}
		0\to\pi_0(J)\to\Q\oplus\prod_{p}\Zp\oplus \Z/2\xrightarrow{h_0} \prod_{p}\Qp\to \pi_{-1}(J)\to \prod_{p}\Zp\xrightarrow{h_{-1}}\prod_{p}\Qp\to \pi_{-2}(J)\to 0.
		\end{equation*}
		As $h_0$ is surjective and $h_{-1}$ is injective, we have 
		\begin{equation*}
		\pi_0(J)\simeq \Z\oplus \Z/2, \quad \pi_{-1}(J)=0,\quad \pi_{-2}(J)\simeq \Q/\Z.
		\end{equation*}
		For $i\neq 0,-1,-2$, we recover $\pi_i(J)$ from \Cref{Subsec:HFPSS} and \Cref{thm:von_staudt}.
	\end{proof}
	\begin{rem}
		We call $S^0_{KU}$ the $J$-spectrum because the Hurewicz map (also the $KU$-localization map) $S^0\longrightarrow S^0_{KU}$ detects the image of $J_{4k-1}$. But $\pi_k(J)$ is not the same as the image of the stable $J$-homomorphism in general. The spectrum $J$ is non-connective and has an extra $\Z/2$-summand in $\pi_{0}(J)$ and $\pi_{8k+1}(J)$ when $k>0$. For details, see \cite{j4}. 
	\end{rem}
	
	\subsection{$J$-spectra with level structures}\label{Subsec:jn}
	We will now add level structures to the $J$-spectrum. Let $\mu_N$ be the $N$-torsion sub-group scheme of $\Gmh$. Define $\M_{mult}(N)$ to be the moduli stack of globally height $1$ formal groups with $\mu_N$-level structures.  More precisely, let $\M_{mult}$ be the stack associated to the Hopf algebroid $(KU_0,KU_0)$. By \Cref{cor:BAutGm} and descent theory for formal groups, $\M_{mult}$ classifies formal groups that are height $1$ at all finite primes of the coefficient ring. Let $\M_{fin~gp}$ be the moduli stack of finite flat group schemes. Then $\M_{mult}(N)$ is the 2-categorical pullback:
	\begin{equation*}
		\begin{tikzcd}
			\M_{mult}(N)\ar[dr,phantom,"\lrcorner",very near start]\rar\dar&*\dar["\mu_N"]\\
			\M_{mult}\rar["{\Gh\mapsto \Gh{[N]}}"']&\M_{fin~gp},
		\end{tikzcd}
	\end{equation*}
	Here $*$ means the stack representing the constant of sheaf of groupoids with one object and the identity morphism. By the definition of 2-categorical pullbacks of stacks, objects in the  $R$-points of $\M_{mult}(N)$ are given by:
	\begin{equation*}
	\M_{mult}(N)(R)=\left\{\left(\Gh,\eta\colon \mu_N\simto\Gh[N]\right)\left\mid \begin{array}{cc}
	\text{$\Gh$ is a formal group over $R$}\\ \text{that has height $1$ at all primes}
	\end{array}\right.\right\}.
	\end{equation*}
	The local structures of $\M_{mult}(N)$ are determined by the local behaviors of $\mu_N$.	
	\begin{lem} 
		The multiplicative formal group $\Gmh$ has no non-trivial finite subgroup over $\Q$. Over $\Zp$, finite subgroups of $\Gmh$ are of the form $\mu_{p^v}$ for some $v\ge 0$. As a result, $(\mu_N)_\Q\simeq 0$ for all $N$ and $\left(\mu_N\right)^\wedge_{p}\simeq \mu_{p^{v}}$, where $v=v_p(N)$.
	\end{lem}
	\begin{proof}
		This follows from the facts that the endomorphism rings of $\Gmh$ over $\Q$ and over $\Zp$ are $\mathrm{End}_{\Q}(\Gmh)\simeq \Q$ and $\mathrm{End}_{\Zp}(\Gmh)\simeq\Zp$, respectively.
	\end{proof}
	\begin{prop}\label{Prop:M_mult_N_structure}
		 Rationally, we have $(\M_{mult}(N))_\Q\simeq (\M_{mult})_\Q$. Fix a prime $p$ and let $v=v_p(N)$, we have
		 \begin{equation*}
		 	\M_{mult}(N)^\wedge_{p}\simeq\M_{mult}(p^v)^\wedge_{p}\simeq B(1+p^v\Zp).
		 \end{equation*}
	\end{prop}
	\begin{cor}
		For any odd number $N$, we have an equivalence of stacks $\M_{mult}(N)\simeq \M_{mult}(2N)$.
	\end{cor}
	\begin{proof}
		This follows from the fact $\zx{2N}$ is canonically isomorphic $\zx{N}$ if $N$ is odd.
	\end{proof}
	
	\begin{thm}[Hopkins-Miller, Goerss-Hopkins, {\cite[Theorem 2.1]{Rezk_Hopkins-Miller}}, \mbox{\cite{Goerss-Hopkins_moduli_spaces}}] 
		Let $\mathcal{FG}$ denote the category whose objects are pairs $(\kappa,\Gamma)$, where $\Gamma$ is a finite height formal group over a finite field $k$ of characteristic $p$ and whose morphisms are pairs of maps $(i,f)\colon(\kappa_1,\Gamma_1)\to(\kappa_2,\Gamma_2)$, where $i\colon\kappa_1\to \kappa_2$ is a ring homomorphism and $f\colon  \Gamma_1\simto i^*\Gamma_2$ is an isomorphism of formal groups. 
		
		Then there exists a functor $(\kappa,\Gamma)\to E_{\kappa,\Gamma}$ from $\mathcal{FG}^\mathrm{op}$ to the category of $\einf$-ring spectra, such that
		\begin{enumerate}
			\item $E_{\kappa,\Gamma}$ is a commutative ring spectra.
			\item There is a unit in $\pi_2(E_{\kappa,\Gamma})$.
			\item $\pi_\mathrm{odd}E_{\kappa,\Gamma}=0$, which implies $E_{\kappa,\Gamma}$ is complex-oriented.
			\item The formal group associated to $E_{\kappa,\Gamma}$ is the universal deformation of $(\kappa,\Gamma)$.
		\end{enumerate}
	\end{thm}
	\begin{prop}[{\mbox{\cite[Theorem 2.7]{Goerss_O_top}}. See also \mbox{\cite[Lemma 3.8]{BCM_rmk_k1_alg_k}}}]\label{prop:otop_section}
		There is a sheaf $\otop_{K(1)}$ of $K(1)$-local $\einf$-ring spectra over the stack $\Hh{1}\simeq B\Zpx= \spf \Zp//\Zpx$ such that
		\begin{equation*}
		\Gamma\left(\otop_{K(1)},B\Zpx\right)\simeq S^0_{K(1)},\quad \Gamma\left(\otop_{K(1)},B(1+p^v\Zp)\right)\simeq S^0_{K(1)}(p^v)=\left(\Kp\right)^{h(1+p^v\Zp)}.
		\end{equation*}
	\end{prop}
	\begin{rem}[{\mbox{\cite[Theorem 2.7]{Goerss_O_top}}}]
		Let $\Hh{h}$ be the moduli stack of formal groups over $p$-complete local rings with height $h$ reductions modulo the maximal ideal. The Hopkins-Miller theorem and the Goerss-Hopkins theorem imply there is a sheaf of $K(h)$-local $\einf$-ring spectra $\otop_{K(h)}$ over $\Hh{h}$ whose global section is the $K(h)$-local sphere $S^0_{K(h)}$. For the algebro-geometric properties of the stack $\Hh{h}$, see \cite[Chapter 7]{modfg}.
	\end{rem}
	\Cref{cor:jn_structure} implies $\M_{mult}(N)^\wedge_{p}\simeq\M_{mult}(p^v)^\wedge_{p}\to (\M_{mult})^\wedge_{p}$ is a $\zx{p^v}$-torsor for each prime $p$. Thus by \Cref{prop:otop_section} we can define $J(N)$, the \textbf{$J$-spectrum with $\mu_N$-level structure} by setting $J(N)^\wedge_p=\otop_{K(1)}(\M_{mult}(p^v))\simeq S^0_{KU/p}(p^v)$ and $J(N)_\Q=S^0_\Q$ as follows:
	\begin{con}\label{con:jn}
		The spectrum $J(N)$ is the homotopy pullback of the following arithmetic fracture square as in \eqref{eqn:J_frac_sq}: 
		\begin{equation}\label{eqn:jn}
		\begin{tikzcd}
		J(N)\rar\dar\arrow[dr, phantom, "\lrcorner", very near start]&\prod_{p}S^0_{KU/p}\left(p^{v_p(N)}\right)\dar["L_\Q"]\\
		S^0_\Q\rar["h_\Q"']&\left(\prod_{p}S^0_{KU/p}\left(p^{v_p(N)}\right)\right)_\Q
		\end{tikzcd}
		\end{equation}
		Here $h_\Q$ is the rational Hurewicz map and $L_\Q$ is the rationalization map. The map $h_\Q$ exists because the lower right corner in the diagram is a rational ring spectrum.
	\end{con}
 	The spectrum $J(N)$ defined above satisfies the prescribed local properties:
 	\begin{cor}\label{cor:jn_structure} 
 		$J(N)_\Q\simeq S^0_\Q$ for all $N$ and $J(N)^\wedge_{p}\simeq S^0_{K(1)}(p^v)$, where $v=v_p(N)$. Moreover, $J(N)\simeq J(2N)$ for any odd number $N$.
 	\end{cor}
	\begin{prop}\label{prop:jn_znx_action}
		$J(N)$ admits a natural $\znx$-action such that
		\begin{itemize}
			\item $\znx$ acts on $J(N)_\Q$ trivially.
			\item $\znx$ acts on $J(N)^\wedge_{p}\simeq S^0_{K(1)}(p^v)$ by the Galois action of its quotient group $\zx{p^v}$.
		\end{itemize}
	\end{prop}
	\begin{proof}
		Since the spectrum $S^0_{K(1)}(p^v)$ is a $\zx{p^v}$-Galois extension of $S^0_{K(1)}$, it admits a natural $\zx{p^v}$-action. As a result the product $\prod_{p}S^0_{KU/p}\left(p^{v_p(N)}\right)$ admits a natural $\znx\simeq \prod_{p\mid N}\zx{p^v}$-action. (When $p\nmid N$, $\znx$ acts on $S^0_{KU/p}$ trivially). The spectrum $\left(\prod_{p}S^0_{KU/p}\left(p^{v_p(N)}\right)\right)_\Q$ in the lower right corner of \eqref{eqn:jn} then inherits a $\znx$-action from that on $\prod_{p}S^0_{KU/p}\left(p^{v_p(N)}\right)$.
		
		We now need to check the rational Hurewicz map $h_\Q$ in \eqref{eqn:jn} is $\znx$-equivariant. As both spectra are rational, it suffices to check the induced maps on homotopy groups are equivariant by Cooke's obstruction theory (see \Cref{Subsec:Moore}). Since $\pi_*(S^0_\Q)$ is concentrated in $\pi_0$ and $\znx$ acts on it trivially, it suffices to check $\znx$ acts $\pi_0\left(S^0_{KU/p}\left(p^{v_p(N)}\right)_\Q\right)$ trivially. Recall from \Cref{defn:s_k1_pv}, $S^0_{KU/p}(p^v)= \left(\Kp\right)^{h(1+p^v\Zp)}$. The HFPSS in \Cref{Subsec:HFPSS} shows
		\begin{equation*}
			\pi_0\left(S^0_{KU/p}\left(p^{v}\right)_\Q\right)\simeq H_c^0\left(1+p^v\Zp;\pi_0\left(\Kp\right)\right)\otimes \Q.
		\end{equation*}
		As the Adams operation $\psi^a$ acts on $\pi_0\left(\Kp\right)$ trivially for all $a\in\Zpx$, the residual $\zx{p^v}$-action on the group cohomology $H_{{c}}^*\left(1+p^v\Zp;\pi_0\left(\Kp\right)\right)$ is also trivial. Hence $\zx{p^v}$ acts trivially on $\pi_0\left(S^0_{KU/p}\left(p^{v}\right)_\Q\right)$. 
		
		We have shown the rational Hurewicz map $h_\Q$ is $\znx$-equivariant. Then $J(N)$ as the homotopy pullback in \eqref{eqn:jn} of a diagram of $\znx$-equivariant maps of spectra has a natural $\znx$-action with the prescribed local properties.
	\end{proof}
	\begin{prop}\label{prop:JN_k_loc}
		$J(N)$ is a $KU$-local $\einf$-ring spectrum, with $\znx$ acting on it by $\einf$-ring automorphisms as described in \Cref{prop:jn_znx_action}.
	\end{prop}
	\begin{proof}
		This proposition contains three parts:
		\begin{enumerate}
			\item $J(N)$ is an $\einf$-ring spectrum since it is the homotopy pullback of a diagram of $\einf$ maps between $\einf$-ring spectra. 
			\item $J(N)$ is $KU$-local since $J(N)^\wedge_{p}\simeq S^0_{KU/p}\left(p^{v_p(N)}\right)$ is $KU/p$-local for all primes $p$ by \Cref{cor:jn_structure}.
			\item The action of $\zx{p^{v_p(N)}}$ on $J(N)^\wedge_{p}\simeq S^0_{KU/p}\left(p^{v_p(N)}\right)$ is $\einf$ by the Goerss-Hopkins theorem. Thus the action of $\znx\simeq\prod_{p\mid N}\zx{p^{v_p(N)}}$ is $\einf$ on the upper right corner of \eqref{eqn:jn}. This implies the induced $\znx$-action on lower right corner is also $\einf$. The trivial $\znx$-action on $S^0_\Q$ is $\einf$. We conclude $\znx$ acts by $\einf$-ring maps on $J(N)$ in \Cref{prop:jn_znx_action}, since the action is assembled from $\einf$-actions on the other three corners of \eqref{eqn:jn}.
		\end{enumerate}		
	\end{proof}
	\begin{prop}
		The homotopy fixed point spectrum $J(N)^{h\znx}$ is equivalent to $J$ after inverting $\prod_{p\mid N}(p-1)$. In particular, $J(N)^{h\znx}\simeq J$ when $N$ is a power of $2$.
	\end{prop}
	\begin{proof}
		Denote the product $\prod_{p\mid N}(p-1)$ by $\Pi$. It is equal to $1$ iff $N$ is a power of $2$. We need to verify the equivalence rationally and at primes not dividing $\Pi$. Rationally,  $\znx$ is a finite group acting trivially on $J(N)_\Q$ by \mbox{\Cref{prop:jn_znx_action}}. As a result, the $E_2$-page of the HFPSS to compute $\pi_*(J(N)_\Q)^{h\znx}$ is concentrated in the $0$-line, yielding
		\begin{equation*}
			\pi_*\left((J(N)_\Q)^{h\znx}\right)=\pi_*\left(J(N)_\Q\right)^{\znx}=\pi_*\left(J(N)_\Q\right).
		\end{equation*}  		
		This shows the natural map $(J(N)_\Q)^{h\znx}\to J(N)_\Q$ is a weak equivalence. Now fix a prime $p$ not dividing $\Pi$ and write $N=p^v\cdot N'$, where $p\nmid N'$. Then the group $(\Z/p^v)^\times$ is a summand of $\znx$. We have:
		\begin{equation*}
			\left(J(N)^{h\znx}\right)^\wedge_{p}\simeq S^0_{K(1)}(p^v)^{h\znx}
			\simeq  \left(S^0_{K(1)}(p^v)^{h(\Z/p^v)^\times}\right)^{h(\Z/N')^\times}
			\simeq  \left(S^0_{K(1)}\right)^{h(\Z/N')^\times},
		\end{equation*}
		where the group $(\Z/N')^\times$ acts trivially on $S^0_{K(1)}$. When the order of this group is coprime to $p$, the same argument as in that rational case shows the homotopy fixed point spectrum $\left(S^0_{K(1)}\right)^{h(\Z/N')^\times}$ is equivalent to $S^0_{K(1)}$. We claim the group order $|(\Z/N')^\times|$ and the product $\Pi=\prod_{q\mid N}(q-1)$ has the same $p$-valuations. To see this, notice 
		\begin{equation*}
			|(\Z/N')^\times|=\phi(N')=\frac{\phi(N)}{\phi(p^v)}=\prod_{q\mid N,q\ne p}(q-1)q^{v_q(N)-1},
		\end{equation*}
		where $\phi$ is the Euler's totient function. This implies $v_p(\phi(N'))=v_p\left(\prod_{q\mid N,q\ne p}(q-1)\right)$. Denote the latter product by $\Pi'$. The two products are related by
		\begin{equation*}
			\Pi=\left\{\begin{array}{cl}
				\Pi',& p\nmid N;\\
				(p-1)\Pi',& p\mid N
			\end{array}\right.
		\end{equation*}
		 Since $(p,p-1)=1$, we have $v_p(\Pi)=v_p(\Pi')=v_p(\phi(N'))$. This implies $p\nmid |(\Z/N')^\times|=\phi(N')$.
	\end{proof}
	\begin{prop}\label{prop:jn_not_gal}
		Let $\Pi$ be as above and let $p$ be a prime dividing $\Pi$. Then the $p$-completion of the homotopy fixed point spectrum $J(N)^{h\znx}$ is not equivalent to $J^\wedge_{p}\simeq S^0_{K(1)}$. As a result, $J(N)$ is NOT a $\znx$-Galois extension of $J$ unless $N$ is a power of $2$.
	\end{prop}
	\begin{proof}
		First by the ambidexterity theorem in the $K(1)$-local category \mbox{\cite{Greenlees-Sadofsky_Tate_cohomology,Hovey-Sadofsky_Tate_cohomology}}, we have
		\begin{equation*}
			\left(J(N)^{h\znx}\right)^\wedge_{p}\simeq \left(S^0_{K(1)}\right)^{h(\Z/N')^\times}\simeq \left(S^0_{K(1)}\right)_{h(\Z/N')^\times}\simeq (B(\Z/N')^\times_+)_{K(1)}.
		\end{equation*}
		Using \mbox{\Cref{lem:BA_k1}} that will be proved later, one can check that
		\begin{equation*}
			(B(\Z/N')^\times_+)_{K(1)}\simeq (S^0_{K(1)})^{\vee p^{v_p(|(\Z/N')^\times|)}}=(S^0_{K(1)})^{\vee p^{v_p(\Pi)}}.
		\end{equation*}
		This is not equivalent to $S^0_{K(1)}$ since $p$ divides $\Pi$ by assumption.
	\end{proof}	
	Parallel to \eqref{eqn:pi_j}, we now compute $\pi_*(J(N))$.
 	\begin{prop}\label{prop:pi_jn}
 		The computation of $\pi_*(J(N))$ has two cases: $4\mid N$ and $N$ is odd (since $J(N)\simeq J(2N)$ for odd $N$). Let $\Pi_{2k,N}$ be the product of all prime divisors $p$ of $N$ which satisfy the condition that $(p-1)$ divides $2k$. Define $D_{2k,N}$ by
 		\begin{equation*}
 			D_{2k,N}=\left\{\begin{array}{cl}
 			ND_{2k}/(2\Pi_{2k,N}),& \text{if }4\mid N;\\
 			ND_{2k}/\Pi_{2k,N},&\text{if }2\nmid N.
 			\end{array}\right. 
 		\end{equation*}
 		When $4\mid N$, we get
 		\begin{equation}\label{eqn:pi_J4N}
 		\pi_i(J(N))=\left\{\begin{array}{cl}
 		\Z, &i=0;\\
 		\Q/\Z, &i=-2;\\
 		\Z/D_{|2k|,N},&i=4k-1\neq -1;\\
 		\Z/N,&i\equiv 1\mod 4;\\
 		0,&\textup{otherwise.}
 		\end{array}\right.
 		\end{equation}
 		When $N$ is odd, we get
 		\begin{equation*}
 		\pi_i(J(N))=\left\{\begin{array}{cl}
 		\Z\oplus \Z/2, &i=0;\\
 		\Q/\Z, &i=-2;\\
 		\Z/D_{|2k|,N},&i=4k-1\neq -1;\\
 		\Z/N\oplus \Z/2\oplus \Z/2,&i\equiv 1\mod 8;\\
 		\Z/N, &i\equiv 5\mod 8;\\
 		\Z/2, &i\equiv 0,2\mod 8\textup{ and }i\neq 0;\\
 		0,&\textup{otherwise.}
 		\end{array}\right.
 		\end{equation*}
 	\end{prop}
 	\begin{rem}\label{rem:jn_duality}
 		One can check from \eqref{eqn:pi_J4N} that 
 		\begin{equation*}
 		\hom(\pi_{i}(J(4N)),\Q/\Z)\simeq (\pi_{-2-i}(J(4N)))^{\wedge}
 		\end{equation*}
 		holds for all $N$ and $i$, where $(-)^{\wedge}$ is the profinite completion of a group. The formula is true up to summands of $\Z/2$ for $J(N)$ when $N$ is odd. We will see in \Cref{cor:J_4N_BC_dual} that this isomorphism is the result of Brown-Comenetz duality $I_{KU}(J(4N))\simeq \Sigma^{2}J(4N)\wedge M(\wh{\Z})$. In particular, $\pi_{4k-1}(J(4))\simeq \pi_{4k-1}(J)=\Z/D_{|2k|}$, whose order is equal to the denominator of $\zeta(1-2k)$ (expressed as a fraction in lowest terms). This  Brown-Comenetz duality for $J(4)$  that relates $\pi_{4k-1}(J(4))$ with $\pi_{-1-4k}(J(4))$ is similar to the functional equation of the Riemann $\zeta$-function that relates $\zeta(1-2k)$ with $\zeta(2k)$:
 		\begin{equation*}
 			\zeta(2k)=\frac{(2\pi i)^{2k}}{2(2k-1)!}\cdot \zeta(1-2k).
 		\end{equation*}		
 	\end{rem}	
 	\subsection{Constructing Moore spectra with group actions}\label{Subsec:Moore}
 	Another ingredient needed to construct the Dirichlet $J$-spectra and $K(1)$-local spheres is a Moore spectrum with a $\znx$-action induced by a ($p$-adic) Dirichlet character $\chi\colon \znx\to \Cx$ (or $\Cpx$). The first observation is following:
 	\begin{lem}\label{lem:chi_fac}
 		There is a unique number $n$ such that $\chi$ factors as 
 		\begin{equation*}
 			\begin{tikzcd}[row sep=0 em]
 			\chi\colon  \znx\rar[twoheadrightarrow]&C_n\rar[hook]& (\Z[\zeta_n])^\x\rar[hook]&\Cx,&\text{when $\chi$ is $\Cbb$-valued;}\\
			\chi\colon  \znx\rar[twoheadrightarrow]&C_n\rar[hook]& (\Zp[\zeta_n])^\x\rar[hook]&\Cpx,&\text{when $\chi$ is $\Cp$-valued,} 
 			\end{tikzcd}
 		\end{equation*}
 		where $C_n$ is the cyclic group of order $n$ and the second maps send a generator $g\in C_n$ to a primitive $n$-th root of unity $\zeta_n$.
 	\end{lem}
 	Then it suffices to construct the Moore spectra $M(\Z[\zeta_n])$ and $M(\Zp[\zeta_n])$ with $C_n$-actions such that the induced $C_n$-action on $H_0$ (equivalently $\pi_0$) is equivalent to that on $\Z[\zeta_n]$ and $\Zp[\zeta_n]$, where $a\in C_n$ acts by multiplication by $\zeta_n^a$. The latter is called the integral/$p$-adic cyclotomic representation of $C_n$. Properties of such representations needed in this subsection are summarized in \Cref{appen:cyclo_rep}. 
 	
 	We can further reduce to cases $n=p^v$ by noting from \Cref{lem:cyclo_decomp}:
 	\begin{equation*}
 		\begin{array}{rcc}
 		&\displaystyle\Z[\zeta_n]\simeq \bigotimes_{p\mid n}\Z[\zeta_{p^{v_p(n)}}]& \displaystyle\Zp[\zeta_n]\simeq \bigotimes_{q\mid n}\Zp[\zeta_{q^{v_{q}(n)}}],\\
 		\xRightarrow{\text{non-equivariantly}} &\displaystyle M(\Z[\zeta_n])\simeq \bigwedge_{p\mid n} M(\Z[\zeta_{p^{v_p(n)}}])&\displaystyle M(\Zp[\zeta_n])\simeq \bigwedge_{q\mid n} M(\Zp[\zeta_{q^{v_q(n)}}]).
 		\end{array}
 	\end{equation*}
 	The constructions now split into three cases:
 	\begin{enumerate}
 		\item In the integral case, we give an explicit construction suggested by Charles Rezk.
 		\item The $p$-adic case where $n=p^v$ is the $p$-completion of the corresponding integral case.
 		\item The $p$-adic case where $(n,p)=1$ uses Cooke's obstruction theory \cite{Cooke_obstruction} to lift group actions on homotopy groups to the homotopy category of spectra. The comparison of this case with the integral case uses the obstruction theory to uniqueness of the lifting.
 	\end{enumerate}
	\textbf{The integral case.}\quad 
	\begin{con}[Charles Rezk]\label{con:integral_moore}
		From the short exact sequence of $C_{p^v}$-representations in \Cref{lem:cyclo_cyclic_res}:
		\begin{equation}\label{eqn:cyclo_rep_resolution}
		\begin{tikzcd}
		0\rar&\Z[\zeta_{p^v}]\rar&\Z[C_{p^v}]\rar&\Z[C_{p^{v-1}}]\rar&0,
		\end{tikzcd}
		\end{equation}
		we define $M(\Z[\zeta_{p^v}])$ as the de-suspension of the cofiber of the quotient map $C_{p^v}\twoheadrightarrow C_{p^{v-1}}$. That is, there is a cofiber sequence:
		\begin{equation}\label{eqn:cofib}
		\begin{tikzcd}
		S^0\bigwedge (C_{p^v})_+\rar&S^0\bigwedge (C_{p^{v-1}})_+\rar&\Sigma M(\Z[\zeta_{p^v}]).
		\end{tikzcd}
		\end{equation} 
		$M(\Z[\zeta_{p^v}])$ inherits a natural $C_{p^v}$-action from its suspension as the cofiber of a $C_{p^v}$-equivariant map.
	\end{con}
	\begin{prop}
		$M(\Z[\zeta_{p^v}])$ constructed above is a Moore spectrum for $\Z[\zeta_{p^v}]$. The induced $C_{p^v}$-action on $H_0(M(\Z[\zeta_{p^v}]);\Z)$  is equivalent to the cyclotomic action.
	\end{prop}
	\begin{proof}
		Applying $H_*(-;\Z)$ to the cofiber sequence \eqref{eqn:cofib}, we can show that $M(\Z[\zeta_{p^v}])$ is a Moore spectrum. The rest follows from \eqref{eqn:cyclo_rep_resolution}. 
	\end{proof}
	Below are some examples of the $C_{p^v}$-equivariant cell structures of $\Sigma M(\Z[\zeta_{p^v}])$: 
	\begin{figure}[h]
		\begin{tikzcd}[row sep=2.5 em]
		\star\ar[dd,dash,bend right=45,"0" description]\ar[dd,dash,bend left=45,"{1}" description]\\\\{[0]_1}
		\end{tikzcd}\quad\begin{tikzcd}[row sep=2.5 em]
		\star\ar[dd,dash,bend right=45,"0" description]\ar[dd,dash,"1" description]\ar[dd,dash,bend left=45,"2" description]\\\\{[0]_1}
		\end{tikzcd}\quad
		\begin{tikzcd}[row sep=2.5 em]
		\star\ar[dd,dash,bend right=72,"0" description]\ar[dd,dash,bend right=45,"1" description]\ar[dd,dash,bend right=24,"2" description]\ar[dd,dash,"3" description]\ar[dd,dash,bend left=24,"4" description]\ar[dd,dash,bend left=45,"5" description]\ar[dd,dash,bend left=72,"6" description]\\\\{[0]_1}
		\end{tikzcd}\quad		
		\begin{tikzcd}[sep=1.5 em]
		{[1]_4}&&{[0]_4}\\&\star\ar[ur,dash,bend right=40,"0" description]\ar[ur,dash,bend left=40,"4" description]\ar[ul,dash,bend right=40,"1" description]\ar[ul,dash,bend left=40,"5" description]\ar[dl,dash,bend right=40,"2" description]\ar[dl,dash,bend left=40,"6" description]\ar[dr,dash,bend right=40,"3" description]\ar[dr,dash,bend left=40,"7" description]&\\{[2]_4}&&{[3]_4}
		\end{tikzcd}\quad
		\begin{tikzcd}[column sep=1.73205 em, row sep=1 em]
		&{[0]_3}\ar[dd,bend right=40,dash,"6" description]\ar[dd,dash,"3" description]\ar[dd,bend left=40,dash,"0" description]&\\&&\\
		&\star\ar[dl,bend left=40,dash,"7" description]\ar[dl,dash,"4" description]\ar[dl,bend right=40,dash,"1" description]\ar[dr,bend left=40,dash,"8" description]\ar[dr,dash,"5" description]\ar[dr,bend right=40,dash,"2" description]\\{[1]_3}&&{[2]_3}
		\end{tikzcd}
		\caption{$C_{p^v}$-cell structures of $\Sigma M(\Z[\zeta_{p^v}])$ for $p^v=2,3,7,8,9$}\label{fig:moore_cell}
		\begin{itemize}
			\item $\star$ is the base point and is fixed by the $C_n$-action.
			\item $[a]_b= (a \mod b)$ is the label of (non-equivariant) $0$-cells.
			\item $a=( a\mod n)$ is the label of (non-equivariant) $1$-cells.
			\item $g\in C_n\simeq \Z/n$ acts on the labels by mapping $(a\mod b)$ to $(a+g\mod b)$.
		\end{itemize}
	\end{figure}
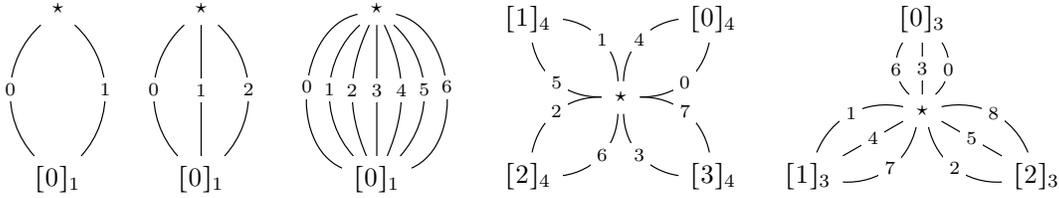

	Here is another description of this construction:
		\begin{enumerate}
			\item $M(\Z[\zeta_2])\simeq S^{\sigma-1}$, where $\sigma$ is the sign representation of $C_2$.
			\item $C_n$ acts on $\Cbb$ by multiplication by $n$-th roots of unity. Denote the associated $C_n$-representation by $\rho_{\text{cyclo}}$ and the representation sphere by $S^{\rho_\text{cyclo}}$. When $n=p$, the $C_p$-cell structure of $\Sigma M(\Z[\zeta_p])$ above shows
			\begin{equation*}
				S^{\rho_\text{cyclo}}\simeq \Sigma M(\Z[\zeta_p])\cup (C_p\x D^2).
			\end{equation*}   
			As a result,  $M(\Z[\zeta_p])$ is the $1$-skeleton in this equivariant cell decomposition of the representation sphere $S^{\rho_\text{cyclo}}$.
			\item Foling Zou has observed and proved the following relation between $M(\Z[\zeta_{p^v}])$ and $M(\Z[\zeta_p])$ via private conversations with the author:
			\begin{prop}[Foling Zou]\label{prop:MZpv}
				There is a $C_{p^v}$-equivariant equivalence:
				\begin{equation*}
				M(\Z[\zeta_{p^v}])\simeq (C_{p^v})_{+}\bigwedge_{C_p}M(\Z[\zeta_p]),
				\end{equation*}
				where $a\in \Z/p\simeq C_p$ acts on $\Z/p^v\simeq C_{p^v}$ by sending $(b\mod p^v)$ to $(b+ap^{v-1}\mod p^v)$.
			\end{prop}
			\begin{proof} Notice that $C_{p^{v-1}}\simeq C_{p^v}/C_p$, we can rewrite this quotient as pointed sets by \begin{equation*}
					(C_{p^{v-1}})_+\simeq S^0\bigwedge_{C_p} (C_{p^{v}})_+,
				\end{equation*}
				where $C_p$ acts on $C_{p^v}$ as described in the proposition. From this we get:
			\begin{align*}
				\Sigma M(\Z[\zeta_{p^{v}}])&=\mathrm{Cofib}\left(S^0\bigwedge (C_{p^v})_+\longrightarrow S^0\bigwedge (C_{p^{v-1}})_+ \right)\\
				\simeq &\mathrm{Cofib}\left(S^0\bigwedge (C_p)_+\bigwedge_{C_p} (C_{p^v})_+\longrightarrow S^0\bigwedge S^0\bigwedge_{C_p} (C_{p^{v}})_+ \right)\\
				\simeq &\mathrm{Cofib}\left(S^0\bigwedge (C_p)_+\longrightarrow S^0\bigwedge S^0 \right)\bigwedge_{C_p} (C_{p^{v}})_+\\
				\simeq &\Sigma M(\Z[\zeta_p])\bigwedge_{C_p} (C_{p^v})_+.
			\end{align*}
			\end{proof}
		\end{enumerate}
	
	Taking external smash product of $M(\Z[\zeta_{p^v}])$ with the prescribed $C_{p^v}$-actions over all $p\mid n$, we have constructed a Moore spectrum $M(\Z[\zeta_n])$ with a $C_n$-action such that the induced action on $H^0(-;\Z)$ is equivalent to the cyclotomic action of $C_n$. We now give an explicit description of the $C_n$-equivariant simplicial structure of $M(\Z[\zeta_n])$.
	
	Write $n=p_1^{v_1}\cdots p_m^{v_m}$. The $(\Z/N)^\times$-equivariant cell decomposition of $X_n=\Sigma^m M(\Z[\zeta_n])$ is constructed as follows:
	\begin{enumerate}
		\item Set the $0$-th skeleton by $\sk_0 X_n= \star\coprod C_n/C_{p_1\cdots p_m}$, where $\star$ is the base point fixed by the $\znx$-action.
		\item Assuming we have defined $\sk_{k-1} X_n$ for $k<m$, then define the $k$-th skeleton to be:
		\begin{equation*}
			\sk_{k} X_n= \sk_{k-1} X_n\bigcup \left(\coprod_{~~{i_1}<\cdots<{i_{m-k}}}C_n/C_{p_{i_1}\cdots p_{i_{m-k}}}\right)\x \Delta^k.
		\end{equation*}
		The attaching map of an equivariant $k$-simplex $C_n/C_{p_{i_1}\cdots p_{i_{m-k}}}\x \Delta^k$ is described by the following:
		\begin{itemize}
			\item The $0$-th face $C_n/C_{p_{i_1}\cdots p_{i_{m-k}}}\x \Delta^k_{[0]}$ is attached to the base point $\star$.
			\item Let $\{j_1<\cdots<j_k\}$ be the complement of $\{i_1,\cdots i_{m-k}\}\subseteq\{1,\cdots,m\}$. Then the $l$-th face $C_n/C_{p_{i_1}\cdots p_{i_{m-k}}}\x \Delta^k_{[l]}$ for $1\le l\le k$ is attached to the equivariant $(k-1)$-complex \begin{equation*}C_n/C_{p_{i_1}\cdots p_{i_{m-k}}\cdot p_{j_l}}\x \Delta^{k-1}\end{equation*} via the quotient map of orbits.
		\end{itemize}
		\item The top simplex is $C_n\x \Delta^m$. The $0$-th face $C_n\x \Delta^m_{[0]}$ is attached to the base point $\star$. The $l$-th face $C_n\x \Delta^m_{[l]}$ for $1\le l\le m$ is attached to the $(m-1)$-equivariant simplex $C_n/C_{p_l}\x \Delta^{m-1}$ via the quotient map $C_n\twoheadrightarrow C_n/C_{p_l}$.
	\end{enumerate}
	\begin{rem}
		The non-equivariant Euler number of $X_n=\Sigma^mM(\Z[\zeta_n])$ is equal to $1+(-1)^m\phi(n)$ since it is non-equivariantly a wedge sum of $\phi(n)$ many copies of $S^m$. On the other hand, by counting the number of non-equivariant simplices in each dimension from the above construction, we get 
		\begin{align*}
			1+(-1)^m\phi(n)&=e(X_n)=1+\sum_{k=0}^{m-1} \left((-1)^k \sum_{~~{i_1}<\cdots<{i_{m-k}}}\frac{n}{p_{i_1}\cdots p_{i_{m-k}}}\right)+(-1)^m n\\
			\implies\quad \phi(n)&=n+\sum_{k=1}^{m}  \left((-1)^k\sum_{{i_1}<\cdots<{i_{k}}}\frac{ n}{p_{i_1}\cdots p_{i_{k}}}\right).
		\end{align*}
		This is precisely the formula of $\phi(n)=|\{a\in \mathbb{N}\mid 1\le a\le n, (a,n)=1\}|$ via the Inclusion and Exclusion Principle.  
	\end{rem}
	\begin{rem}\label{rem:C2_action_S0}
		The construction above is not unique. For example when $n=2$, $M(\Z[\zeta_2])$ is by definition $S^0$ with a $C_2$-action such that the induced action of $C_2$ on $\pi_*(S^0)$ is the sign representation in all degrees. \Cref{fig:moore_cell} shows our model for $M(\Z[\zeta_2])$ is $S^{\sigma-1}$. But one can check $S^{(2k-1)(\sigma-1)}$ also satisfies the assumptions for all $k\in \Z$ and these are non-equivalent $C_2$-actions on $S^0$.
	\end{rem}
	\textbf{The $p$-adic case when $n=p^v$.}\quad By \Cref{cor:Qp_pv}, $\left(\Z[\zeta_{p^v}]\right)^\wedge_{p}\simeq \Zp[\zeta_{p^v}]$. From this we can simply define the Moore spectrum with a $C_{p^v}$-action by setting
	\begin{equation}\label{eqn:Moore_p^v}
		M(\Zp[\zeta_{p^v}])=M(\Z[\zeta_{p^v}])^\wedge_{p}.
	\end{equation}
	\textbf{The $p$-adic case when $p\nmid n$.}\quad In this case, \Cref{prop:gal_Qp_zeta} implies that $(\Z[\zeta_{n}])^\wedge_{p}\not\simeq \Zp[\zeta_{n}]$, since the two sides have different ranks as $\Zp$-modules. As a result, the construction in the $n=p^v$ case does not apply. Instead, we use Cooke's obstruction theory in \cite{Cooke_obstruction} to lift the $C_n$-action on $\Zp[\zeta_n]=\pi_0(M(\Zp[\zeta_{n}]))$ to the Moore spectrum $M(\Z[\zeta_n])$.
	
	Let $X$ be a spectrum and $h\aut(X)$ be the group of its self-homotopy equivalences. Then $h\aut(X)$ is an associative $H$-space and $\pi_0(h\aut(X))$ is the group of homotopy classes of homotopy equivalences of $X$. Denote the identity component of $h\aut(X)$ by $h\aut_1(X)$. We have a short exact sequence of $H$-spaces:
	\begin{equation*}
		\begin{tikzcd}
		1\rar&h\aut_1(X)\rar&h\aut(X)\rar&\pi_0(h\aut(X))\rar &1.
		\end{tikzcd}
	\end{equation*}
	This induces a fiber sequence by taking classifying spaces:
	\begin{equation*}
	\begin{tikzcd}
	Bh\aut_1(X)\rar&Bh\aut(X)\rar&B\pi_0(h\aut(X)).
	\end{tikzcd}
	\end{equation*}
	An action of a group $G$ on $\pi_0(X)$ is then a group homomorphism $\alpha\colon  G\to \pi_0(h\aut(X))$. 
	\begin{thm}[{\cite[Theorem 1.1]{Cooke_obstruction}}] \label{thm:cooke}
		There is an obstruction theory to lift $\alpha$ to an action on $X$: 
		\begin{equation*}
			\begin{tikzcd}[column sep=4 em]
			&Bh\aut(X)\dar\\
			BG\rar["B\alpha"']\ar[dashed,ur,end anchor=south west]&B\pi_0(h\aut(X)).
			\end{tikzcd}
		\end{equation*}
		The obstruction classes to the existence of such liftings live in 
		\begin{equation*}
		H^n(G;\{\pi_{n-2}(h\aut_1(X))\}),\quad n\ge 3.
		\end{equation*} 
		In particular, one can always lift a $G$-action on $\pi_0(X)$ to $X$ if $G$ is finite and $|G|$ is invertible in $\pi_{n}(h\aut_1(X))$ for all $n\ge 1$.
	\end{thm} 
	\begin{cor}\label{cor:Cn_padic}
		When $p\nmid n$, any $C_{n}$-action on $\pi_*$ of a $p$-complete spectrum can be lifted to an action on the spectrum itself. 
	\end{cor}
	\begin{proof}
		As $n$ is invertible in $\Zp$, group cohomology of $C_n$ with coefficients in $\Zp$-modules vanishes in positive degrees. As a result, the obstruction classes in \Cref{thm:cooke} all vanish.
	\end{proof}
	As a result, there exists a $C_n$-action on the $p$-adic Moore spectrum $M(\Zp[\zeta_n])$ such that the induced action on $\pi_0$ agrees with $p$-adic cyclotomic representation of $C_n$.
	
	One last thing to check is the compatibility of the constructions in the integral and $p$-adic cases when $p\nmid n$. Fix an embedding $\iota\colon\Z[\zeta_n]\hookrightarrow \Zp[\zeta_n]$. It induces a map of Galois groups:
	\begin{equation*}
		\begin{tikzcd}
			\iota^*\colon \gal(\Qp(\zeta_{n})/\Qp)\rar[hook]& \gal(\Q(\zeta_{n})/\Q).
		\end{tikzcd}
	\end{equation*}
	By \Cref{prop:Phi_n_fac}, there is an equivalence of $p$-adic $C_n$-representations: 
	\begin{equation}\label{eqn:Cn_rep_padic}
		\Z[\zeta_n]\otimes \Zp\simeq \bigoplus_{[\sigma]\in \cok \iota^*} (\Zp[\zeta_n])_{\iota\circ \sigma},
	\end{equation}
	where $C_n$ acts on the summand $(\Zp[\zeta_n])_{\iota\circ \sigma}$ by
	\begin{equation*}
		\begin{tikzcd}
		C_n\rar[hook]& (\Z[\zeta_n])^\x\rar["\sigma"]&(\Z[\zeta_n])^\x\rar["\iota"]&(\Zp[\zeta_n])^\x.
		\end{tikzcd}
	\end{equation*}
	By \Cref{cor:Cn_padic}, there is a $C_n$-action on $M(\Zp[\zeta_n])^{\vee |\cok \iota^*|}$ such that the induced $C_n$-action on $\pi_0$ agrees with the right hand side of \eqref{eqn:Cn_rep_padic}. On the other hand, the $C_n$-action $M(\Z[\zeta_n])^\wedge_{p}$ induces an equivalent $C_n$-representation on $\pi_0$. To check the two $C_n$-actions on the $p$-adic Moore spectrum are equivalent, we use the uniqueness part of Cooke's obstruction theory.
	\begin{prop}
		In \Cref{thm:cooke}, the obstruction classes to the uniqueness of the liftings live in
		\begin{equation*}
		H^n(G;\{\pi_{n-1}(h\aut_1(X))\}),\quad n\ge 2.
		\end{equation*}
	\end{prop}
	\begin{cor}
		Let $X$ be a $p$-complete spectrum. When $p\nmid n$, any two lifts of a $C_n$-action from $\pi_*(X)$ to $X$ are $C_n$-equivariantly equivalent.
	\end{cor}
	As a result, there is a $C_n$-equivalence:
	\begin{equation*}
		M(\Z[\zeta_n])^\wedge_{p}\simeq \bigvee_{[\sigma]\in \cok \iota^*} (M(\Zp[\zeta_n]))_{\iota\circ \sigma}.
	\end{equation*}
	\begin{rem}
		When $n=p^v$, there could be non-equivalent $C_{p^v}$-actions on $M(\Zp[\zeta_{p^v}])$ inducing the same action on $\pi_0$. One counterexample in the integral case is $C_2$-equivariant spheres $S^{2\sigma-2}$ and $S^0$ -- both induce the trivial action on the homotopy groups.
	\end{rem}
	Pre-composing with the map $\znx\twoheadrightarrow C_n$ in \Cref{lem:chi_fac}, we have shown in this subsection:
	\begin{thm}\label{thm:moore_spec_action}
		Let $\chi\colon  \znx\to \Cx$ or $\Cpx$ be a Dirichlet character.
		\begin{enumerate}
			\item  There is a Moore spectrum $M(\Z[\chi])$ or $M(\Zp[\chi])$ with a $\znx$-action such that the induced action on $\pi_0$ is equivalent to that induced by $\chi$.
			\item  Let $\iota:\Z[\chi]\hookrightarrow \Zp[\chi]$ be an embedding. There is a $\znx$-equivariant equivalence:
			\begin{equation}\label{eqn:padic_decomp}
				M(\Z[\chi])^\wedge_{p}\simeq \bigvee_{[\sigma]\in \cok \iota^*} M(\Zp[\iota\circ \sigma\circ \chi]).
			\end{equation}
		\end{enumerate}
	\end{thm}
 	\subsection{The homotopy eigen-spectra}
 	Now we are ready to twist the $J$-spectrum and the $K(1)$-local spheres with a Dirichlet character. Analogous to \Cref{Prop:Dirichlet_char_moduli}, the twisting is realized as the "homotopy $\chi$-eigen-spectrum".
	\begin{con}\label{con:twisted_j}
	Let $\chi\colon \znx\to \Cx$ be a primitive Dirichlet character of conductor $N$. We define the \textbf{Dirichlet $J$-spectrum} by:
		\begin{equation}\label{eqn:twisted_j}
			J(N)^{h\chi}=\map(M(\Z[\chi]),J(N))^{h\znx},
		\end{equation}		
		Let $\chi\colon \znx\to \Cpx$ be a primitive $p$-adic Dirichlet character of conductor $N$ and set $v=v_p(N)$. We define the \textbf{Dirichlet $K(1)$-local sphere} to be
		\begin{equation}\label{eqn:twisted_k1}
			S^0_{K(1)}(p^v)^{h\chi}=\map_{\Zp}\left(M(\Z_p[\chi]),S^0_{K(1)}(p^v)\right)^{h\znx}.
		\end{equation}
		The $\znx$-actions on the Moore spectrum and $J(N)$ are described in \Cref{thm:moore_spec_action} and \Cref{prop:jn_znx_action}, respectively.  The group $\znx$ acts on $S^0_{K(1)}(p^v)$ through the Galois action of its quotient group $\zx{p^v}$.
	\end{con}
	\begin{rem}
		The spectra $J(N)^{h\chi}$ and $S^0_{K(1)}(p^v)^{h\chi}$ depend on the constructions of the $\znx$-actions on $M(\Z[\chi])$ and $M(\Zp[\chi])$, which is not unique in general as illustrated in \Cref{rem:C2_action_S0}. When $N=4, p=2$ and $\chi\colon \zx{4}\simeq C_2\to\Cx_2$, different models of $M(\Z_2[\chi])$ lead to different $S^0_{K(1)}(4)^{h\chi}$. We will explain the differences in more detail in \Cref{rem:exotic_pic}.
	\end{rem}
	One immediate consequence of this construction is
	\begin{prop}\label{prop:Dirichlet_J_isom_rep}
		If $\chi_1$ and $\chi_2$ are Dirichlet characters of conductor $N$ with isomorphic induced representations, then $J(N)^{h\chi_1}\simeq J(N)^{h\chi_2}$. In particular, $J(N)^{h\chi}\simeq J(N)^{h(\sigma\circ\chi)}$ for any $\sigma\in \gal(\Q(\chi)/\Q)$.
	\end{prop}
	\begin{rem}
		As $S^0_{K(1)}(p^v)$ is $K(1)$-local, we have
		\begin{equation*}
			S^0_{K(1)}(p^v)^{h\chi}\simeq \map_{K(1)\text{-loc}}\left(M(\Z_p[\chi])_{K(1)},S^0_{K(1)}(p^v)\right)^{h\znx}
		\end{equation*}
		is also $K(1)$-local.
	\end{rem}
	\begin{prop}\label{prop:hess}
		The $E_2$-pages of the HFPSS to compute $\pi_*\left((J(N))^{h\chi}\right)$ and $\pi_*\left(S^0_{K(1)}(p^v)^{h\chi}\right)$ are identified with 
		\begin{align}\label{eqn:hess_e2}
			E_2^{s,t}\simeq \ext^s_{\Z[\znx]}\left(\Z[\chi],\pi_t(J(N))\right)&\Longrightarrow \pi_{t-s}\left(J(N)^{h\chi}\right),\\\label{eqn:hess_e2_p}
			E_2^{s,t}\simeq \ext^s_{\Zp[\znx]}\left(\Zp[\chi],\pi_t\left(S^0_{K(1)}(p^v)\right)\right)&\Longrightarrow\pi_{s-t}\left(S^0_{K(1)}(p^v)^{h\chi}\right),
		\end{align}
		where $a\in\znx$ acts on $\Z[\chi]$ and $\Zp[\chi]$ by multiplication by $\chi(a)$.
	\end{prop}
	\begin{proof}
		We give a proof of \eqref{eqn:hess_e2}. The proof of \eqref{eqn:hess_e2_p} is similar. By construction, the $E_2$-page of the HFPSS for \eqref{eqn:twisted_j} is 
		\begin{equation*}
			E_2^{s,t}=H^s(\znx;\pi_t(\map(M(\Z[\chi]),J(N)))).
		\end{equation*}
		Denote the rank of $\Z[\chi]$ as a free $\Z$-module by $r$. Then $M(\Z[\chi])$ is non-equivariantly equivalent to $\left(S^0\right)^{\vee r}$. The Atiyah-Hirzebruch spectral sequence:
		\begin{equation*}
			E_2^{s,t}=H^s(M(\Z[\chi]);\pi_{t}(J(N)))\Longrightarrow \pi_{s+t}(\map(M(\Z[\chi]),J(N)))
		\end{equation*} 
		collapses on the $E_2$-page since $H^*(M(\Z[\chi]);-)$ is concentrated in degree $0$.
		Together with the universal coefficient theorem, this implies:
		\begin{align*}
			\pi_t(\map(M(\Z[\chi]),J(N)))\simeq &H^0(M(\Z[\chi]);\pi_{t}(J(N)))\\
			\simeq&\hom_\Z(H^0(M(\Z[\chi]);\Z),\pi_t(J(N)))\\
			\simeq& \hom_\Z(\Z[\chi],\pi_t(J(N))).
		\end{align*}
		By \Cref{thm:moore_spec_action}, $\znx$ acts on $\Z[\chi]\simeq H^0(M(\Z[\chi]);\Z)$ by $\chi$. Since $\Z[\chi]$ is a finite free $\Z$-module, the Grothendieck spectral sequence
		\begin{align*}
		E^{s,t}_2=H^s(\znx;\ext^t_\Z(\Z[\chi],\pi_t(J(N))))\Longrightarrow \ext^{s+t}_{\Z[\znx]}\left(\Z[\chi],\pi_t(J(N))\right)
		\end{align*}
		collapses on the $E_2$-page, yielding
		\begin{equation*}
			H^s(\znx;\hom_\Z(\Z[\chi],\pi_t(J(N))))\simeq  \ext^{s}_{\Z[\znx]}\left(\Z[\chi],\pi_t(J(N))\right).
		\end{equation*}
	\end{proof}
	\begin{rem}
		The $E_2$-page of \eqref{eqn:hess_e2} consists of the derived $\chi$-eigenspaces of $\pi_*(J(N))$. Moreover, $J(N)^{h\chi}$ is defined as the \textbf{homotopy $\chi$-eigen-spectrum} of $J(N)$. In this sense, we will call \eqref{eqn:hess_e2} the \textbf{homotopy eigen spectral sequence} (HESS). \footnote{The alternative name "homotopy eigen-spectrum spectral sequence" would be too redundant.}
	\end{rem}
	\subsection{Local structures of the Dirichlet $J$-spectra}
	While it is not hard to compute the $E_2$-page of \eqref{eqn:hess_e2} directly, the differentials are non-trivial as the cohomological dimension of $\znx$ with coefficients in $\Z$-modules is infinite. Instead, we will compute $\pi_*\left(J(N)\right)^{h\chi}$ rationally and completed at each prime $p$. 
	
	Over $\Q$, the spectral sequence is concentrated in the $0$-th line, since $\znx$ is a finite group. By \Cref{cor:jn_structure}, $J(N)_\Q\simeq S^0_\Q$ and  $\znx$ acts on it trivially. We conclude from these facts:
	\begin{prop}\label{prop:JN_hchi_Q_contractible}
		The homotopy groups of $\left(J(N)^{h\chi}\right)_\Q$ are given by
		\begin{equation*}
			\pi_i\left(\left(J(N)^{h\chi}\right)_\Q\right)\simeq \left\{\begin{array}{cl}
			\Q,& i=0 \textup{ and }\chi=\chi^0;\\
			0,& \textup{otherwise.}
			\end{array}\right.
		\end{equation*}
	\end{prop}
	\begin{cor}\label{cor:dirichlet_J_Q}
		$\left(J(N)^{h\chi}\right)_\Q$ is contractible unless $\chi=\chi^0$ is trivial. In that case, $N=0$ and $J(N)^{h\chi}_\Q\simeq J_\Q\simeq S^0_\Q$.
	\end{cor}
	\begin{proof}
		By \Cref{cor:jn_structure}, $J(N)_\Q\simeq S^0_\Q$. Then $E_2^{s,t}\otimes \Q=0$ for all $(s,t)\ne (0,0)$ \eqref{eqn:hess_e2}. The remaining entry
		$E_2^{0,0}\simeq \Q(\chi^{-1})^{\znx}$
		is zero unless $\chi=\chi^0$ is trivial, yielding the claim.
	\end{proof}
	\begin{prop}\label{prop:jnchi_p_decomposition}
		Fix an embedding $\iota:\Q(\chi)\hookrightarrow \Cp$. The $p$-completion of the Dirichlet $J$-spectrum decomposes as 
		\begin{equation*}
		\left(J(N)^{h\chi}\right)^{\wedge}_p\simeq \bigvee_{[\sigma]\in\cok \iota^* }S^0_{K(1)}(p^v)^{h(\iota\circ\sigma \circ \chi)},
		\end{equation*}
		where $\iota^*\colon\gal(\Qp({\iota\circ\chi})/\Qp)\hookrightarrow \gal(\Q(\chi)/\Q)$ is defined in \eqref{eqn:iota_star}.
	\end{prop}
	\begin{proof}
	Since homotopy fixed points and $p$-completions commute and that the $p$-completion of $J(N)$ is $S^0_{K(1)}(p^v)$
	\begin{equation*}
	\left(J(N)^{h\chi}\right)^{\wedge}_p\simeq \map_{\Zp}\left(M(\Z[\chi])^\wedge_{p},S^0_{K(1)}(p^v)\right)^{h\znx}
	\end{equation*}
	The rest follows from \eqref{eqn:padic_decomp}.
	\end{proof}
	Now we give explicit descriptions of how $\left(J(N)^{h\chi}\right)^\wedge_{p}$ decomposes when $N=p^v$.
	\begin{exmps}\label{exmp:Dirichlet_j_decomp}
		Let $\chi\colon \znx\to \Cx$ be a Dirichlet character of conductor $N=p^v$. Fix an embedding $\iota\colon \Z[\chi]\hookrightarrow \Cp$. There are two cases.
		\begin{itemize}
			\item $p=2$. The $v=1$ case is trivial. For $v>1$, $\zx{2^v}\simeq \{\pm 1\}\x \Z/{2^{v-2}}$. When $v=2$, $\chi$ is primitive when it is non-trivial, i.e. $\chi(-1)=-1$. When $v>2$, $\chi$ is primitive of conductor $2^v$ iff $\Z[\chi]\simeq \Z[\zeta_{2^{v-2}}]$. In both cases, we have by \Cref{prop:gal_Qp_zeta},  $(\Z[\zeta_{2^{v-2}}])^\wedge_{2}\simeq \Z_2[\zeta_{2^{v-2}}]$. As a result,
			\begin{equation*}
				\left(J(2^v)^{h\chi}\right)^\wedge_{2}\simeq S^0_{K(1)}(2^v)^{h(\iota\circ\chi)}.
			\end{equation*}
			Notice for any two $2$-adic Dirichlet characters $\chi_1$ and $\chi_2$ of conductor $2^v$ with the same parity, there is a $\sigma\in \gal(\Q_2(\zeta_{2^{v-2}})/\Q_2)$ such that $\chi_1=\sigma\circ \chi_2$. By \Cref{prop:Dirichlet_J_isom_rep}, the above isomorphism does not depend on $\iota$, since $\iota\circ \chi(-1)$ is independent of the choice of $\iota$. 
			\item $p>2$. In this case, $\zx{p^v}\simeq \zpx\x \Z/{p^{v-1}}$. When $v=1$, $\chi$ is primitive iff it is non-trivial. When $v>1$, $\chi$ is primitive iff $\zeta_{p^{v-1}}\in \Z[\chi]$, i.e. $\chi|_{\Z/p^{v-1}}$ is injective. By \Cref{cor:cyclo_rep_p_decomp_pv}, there is an isomorphism of $p$-adic $\zx{p^v}$-representations:
			\begin{equation*}
				\left(\Z[\chi]\right)^\wedge_{p}\simeq \bigoplus_{\substack{0\le a\le p-2\\\ker \omega^a=\ker\chi|_{\zpx}}} \Zp[\chi_a],
			\end{equation*}
			where $\chi_a=\omega^a\cdot (\iota\circ\chi|_{\Z/p^{v-1}})$ and $\omega\colon\zpx\to\Zpx$ is the Teichm\"{u}ller character. This implies a decomposition of the $p$-completion of the Dirichlet $J$-spectrum as in \Cref{prop:jnchi_p_decomposition}:
			\begin{equation}\label{eqn:jp_chi_decomp}
			\left(J(p^v)^{h\chi}\right)^\wedge_{p}\simeq \bigvee_{\substack{0\le a\le p-2\\\ker \omega^a=\ker\chi|_{\zpx}}} S^0_{K(1)}(p^v)^{h\chi_a}.
			\end{equation}
		\end{itemize}
	\end{exmps}
	Now we need to compute the homotopy groups of the Dirichlet $K(1)$-local spheres. Like the integral case, while the $E_2$-page of \eqref{eqn:hess_e2_p} are not hard to compute in general, there are infinitely many differentials unless $p\nmid\phi(N)=|\znx|$. We now set up another spectral sequence that we will use in \Cref{Sec:pi_Dirichlet_j}.
	\begin{prop}\label{prop:Dirichlet_k1_Kp}
		Let $\chi\colon \znx\to \Cpx$ be a $p$-adic Dirichlet character of conductor $N$. Write $N=p^v\cdot N'$ where $p\nmid N'$. There is an equivalence of $K(1)$-local spectra:
		\begin{equation*}
		S^0_{K(1)}(p^v)^{h\chi}\simeq \map_{\Zp}\left(M(\Zp[\chi]),\Kp \right)^{h\left(\Zpx\x \zx{N'}\right)},
		\end{equation*}
		where $\Zpx\x\zx{N'}$ acts on the Moore spectra through the action of $\znx$ described in \Cref{Subsec:Moore} via the quotient map $\Zpx\x \zx{N'}\twoheadrightarrow \zx{p^v}\x \zx{N'}\cong \znx$; and on $\Kp$ via the Adams operations by the factor $\Zpx$.
	\end{prop}
	\begin{proof}
		Recall from \Cref{defn:s_k1_pv} that $S^0_{K(1)}(p^v)=\left(\Kp\right)^{h(1+p^v\Zp)}$. From this, we have
		\begin{align*}
		S^0_{K(1)}(p^v)^{h\chi}= &\map_{\Zp}\left(M(\Zp[\chi]),S^0_{K(1)}(p^v) \right)^{h\znx}\\
		\simeq& \map_{\Zp}\left(M(\Zp[\chi]),\left(\Kp\right)^{h(1+p^v\Zp)} \right)^{h\znx}\\
		\simeq & \map_{\Zp}\left(M(\Zp[\chi]),\Kp \right)^{hG}
		\end{align*}
		where $G$ is an extension of $1+p^v$ by $\znx$. The subgroup $1+p^v\Zp$ acts trivially on the Moore spectrum and by the Adams operations on $\Kp$. Notice that the $\zx{p^v}$-action on $S^0_{K(1)}(p^v)=\left(\Kp\right)^{h(1+p^v\Zp)}$ is via the residual Adams operations by viewing $\zx{p^v}$ as a group group of $\Zpx$. The group $G$ is then isomorphic  to $\Zpx\x \zx{N'}$, with actions on $M(\Zp[\chi])$ and $\Kp$ as claimed.
	\end{proof}
	\begin{cor}\label{cor:hfpss_d_k1_gp_coh}
		Write $\chi=\chi_p\cdot \chi'$ where $\chi_p$ and $\chi'$ have conductors $p^v$ and $N'$, respectively. There is another spectral sequence to compute homotopy groups of Dirichlet $K(1)$-local spheres:
		\begin{equation}\label{eqn:D_k1_gp_coh}
		E_2^{s,2t}=\ext_{\Zp\llb \Zpx\x \zx{N'}\rrb}^s(\Zp[\chi],\Zp^{\otimes t})
		\simeq H_c^s(\Zpx\x \zx{N'};\Zp^{\otimes t}[\chi^{-1}])\Longrightarrow \pi_{2t-s}\left(S^0_{K(1)}(p^v)^{h\chi}\right),
		\end{equation}
		where $\Zp^{\otimes t}[\chi]$ is the $\Zpx\x \zx{N'}$-representation associated to the character $\Zpx\x\zx{N'}\xrightarrow{(a,b)\mapsto \chi_p(a)\chi'(b)a^t}(\Zp[\chi])^\x$.
	\end{cor}
	\begin{rem}
		When $N=p>2$ or $N=4$ and $p=2$, $M(\Zp[\chi])$ is non-equivariantly equivalent to $S^0_p$ and $\chi=\omega^a$ for some $a$. \Cref{prop:Dirichlet_k1_Kp} identifies the Dirichlet $K(1)$-local spheres $S^0_{K(1)}(p)^{h\omega^a}$ and $S^0_{K(1)}(4)^{h\omega}$ with elements of finite order in the $K(1)$-local Picard group $\Pic_{K(1)}$, first defined in \cite{HMS_picard}.
	\end{rem}
	
	\section{Computations of the Dirichlet $K(1)$-local spheres}\label{Sec:pi_Dirichlet_j}
	In this section, we compute homotopy groups of the Dirichlet $K(1)$-local spheres and $J$-spectra. By \Cref{prop:jnchi_p_decomposition}, we can recover the $p$-primary parts of the homotopy groups of Dirichlet $J$-spectra from the corresponding summands of Dirichlet $K(1)$-local spheres. Let $\chi\colon \znx\to\Cpx$ be a $p$-adic Dirichlet character of conductor $N$. The computations of $\pi_{*}\left(S^0_{K(1)}(p^v)^{h\chi}\right)$ break up into four cases:
	\begin{enumerate}
		\item $N=1$. 
		\item $N=p^v$ and $p>2$.
		\item $N=2^v$.
		\item $N$ is not a $p$-power.
	\end{enumerate}
	In the $N=1$ case, we recover the classical $K(1)$-local sphere, whose homotopy groups are computed in \eqref{k1spi} when $p>2$, and in \eqref{pis0_2} when $p=2$. When $N$ is power of $p$, we use HFPSS/HESS \eqref{eqn:hess_e2_p} and \eqref{eqn:D_k1_gp_coh} to compute homotopy groups of the Dirichlet $K(1)$-local spheres. When $N$ has prime factors other than $p$, the character $\chi$ factors as a product $\chi=\chi_p\cdot \chi'$, where $\chi_p$ has conductor $p^{v_p(N)}$. The Dirichlet $K(1)$-local spheres are contractible when $|\imag \chi'|$ is not a power of $p$. When $|\imag \chi'|$ is a power of $p$, we compute the homotopy groups of the Dirichlet $K(1)$-local spheres from its construction. 
	\subsection{The $N=p^v$ and $p>2$ case}\label{Subsec:pi_twisted_k1}
	Let's start with the $N=p>2$ case. We will compute $\pi_*\left(S_{K(1)}^0(p)^{h\chi}\right)$ when $p>2$ first using the homotopy eigen spectral sequence (HESS) \eqref{eqn:hess_e2_p}. The $E_2$-page of this spectral sequence is:
	\begin{equation}\label{HESS}
	E^{s,t}_2=\ext_{\Zp[\zpx]}^{s}\left((\Zp)_{\chi},\pi_t\left(S_{K(1)}^0(p)\right)\right)\Longrightarrow \pi_{t-s}\left(S_{K(1)}^0(p)^{h\chi}\right),
	\end{equation}
	where $a\in\zpx$ acts on $(\Zp)_\chi$ by multiplication by $\chi(a)$.
	\begin{rem}
		When $\chi$ is the trivial character $\chi^0$, we recover the HFPSS in \eqref{HFPSS}.
	\end{rem}
	Let $g\in\zpx$ be a generator. A projective resolution of $(\Zp)_\chi$ as a $\Zp[\zpx]$-module is 
	\begin{equation*}		
	\cdots\longrightarrow\Zp[\zpx]\xrightarrow{\x\left(\sum\chi(g)^{-i}g^i\right)}\Zp[\zpx]\xrightarrow{\x\left(g-\chi(g)\right)}\Zp[\zpx]\xrightarrow{g\mapsto\chi(g)}(\Zp)_\chi.
	\end{equation*}
	By \eqref{k1spi}, the homotopy groups of $S^0(p)$ are
	\begin{equation*}
	\pi_t\left(S_{K(1)}^0(p)\right)=\left\{\begin{array}{cl}
	\Z_p, & t=0 \text{ or}-1;\\
	\Z/p^{v_{p}(k)+1},& t=2k-1\neq -1;\\
	0,& \textup{otherwise.}
	\end{array}\right.
	\end{equation*}
	Descending from the Adams operations on $\left(\Kp\right)_t$,  $\zpx$ acts trivially on $\pi_0$ and $\pi_{-1}$ and by $\chi=\omega^k$ on $\pi_{2k-1}$ of $S_{K(1)}^0(p)$. A direct computation shows
	\begin{prop}
		When $\chi=\omega^a$, $a\neq0$, the $E_2$-page of \eqref{HESS} is
		\begin{equation*}
		E_2^{s,t}=\left\{\begin{array}{cl}
		\Z/p^{v_{p}(k)+1}, & s=0, t=2k-1, \textup{ and }(p-1)\mid(k-a);\\
		0,&\textup{otherwise.}
		\end{array}\right.
		\end{equation*}
		As the spectral sequence collapses on the $E_2$-page, we conclude
		\begin{equation}\label{eqn:pi_DK1_p}
		\pi_t\left(S_{K(1)}^0(p)^{h\omega^a}\right)=\left\{\begin{array}{cl}
		\Z/p^{v_{p}(k)+1}, & t=2k-1, \textup{ and }(p-1)\mid(k-a);\\
		0,& \textup{otherwise.}
		\end{array}\right.
		\end{equation}
	\end{prop}
	We can also use the HFPSS in \eqref{eqn:D_k1_gp_coh} to compute $\pi_*\left(S_{K(1)}^0(p)^{h\omega^a}\right)$. The $E_2$-page of this spectral sequence is:
	\begin{equation}\label{eqn:hess_Zpx_p}
	E_2^{s,2t}=\ext_{\Zp\llb\Zpx\rrb}^{s}\left((\Zp)_{\wtchi},\Zp^{\otimes t}\right)\simeq H_c^s(\Zpx; \Zp^{\otimes t}[\chi^{-1}])\Longrightarrow\pi_{2t-s}\left(S_{K(1)}^0(p)^{h\chi}\right),
	\end{equation}
	where $(\Zp)_{\wtchi}$ is the $\Zpx$-representation associated to the character $\wtchi$:
	\begin{equation*}
	\begin{tikzcd}
	\wtchi\colon \Zpx\rar[two heads]&\zpx\rar["\chi"]&\Zpx.
	\end{tikzcd}
	\end{equation*}
	The two approaches to compute $\pi_*\left(S_{K(1)}^0(p)^{h\chi}\right)$ are related by the diagram:
	\begin{equation}\label{diagram_ss}
	\begin{tikzcd}
	\ext_{\Zp[\zpx]}^{r}\left((\Zp)_{\chi},H_c^s\left(1+p\Zp;\left(\Kp\right)_t\right)\right)\dar[Rightarrow,"\text{HFPSS}"']\rar[Rightarrow,"\text{HSSS}"]&\ext_{\Zp\llb\Zpx\rrb}^{r+s}\left((\Zp)_{\wtchi},\left(\Kp\right)_t\right)\dar[Rightarrow,"\text{\eqref{eqn:hess_Zpx_p}}"]\\
	\ext_{\Zp[\zpx]}^{r}\left((\Zp)_{\chi},\pi_{t-s}\left(S_{K(1)}^0(p)\right)\right)\rar[Rightarrow,"\text{\eqref{HESS}}"']
	&\pi_{t-r-s}\left(S_{K(1)}^0(p)^{h\chi}\right)
	\end{tikzcd}
	\end{equation}
	Here, the top line is a Hochschild-Serre spectral sequence. Retrospectively from this diagram, we get when $\chi=\omega^a$, $a\neq 0$:
	\begin{equation}\label{hess_e2}
	\ext_{\Zp\llb\Zpx\rrb}^{s}\left((\Zp)_{\wtchi},\left(\Kp\right)_{2t}\right)\simeq H_c^s(\Zpx;\Zp^{\otimes t}[\chi^{-1}])=\left\{\begin{array}{cl}
	\Z/p^{v_{p}(t)+1}, & s=1, (p-1)\mid(t-a);\\
	0,& \textup{otherwise.}
	\end{array}\right.
	\end{equation}
	\textbf{When $N=p^v>p>2$}, we compute the homotopy groups of the Dirichlet $K(1)$-local spheres using \eqref{eqn:D_k1_gp_coh}. The other spectral sequence \eqref{eqn:hess_e2_p} does not quite work in this case. This is because $\cdp\left(\zx{p^v}\right)=\infty$ when $v>1$, whereas $\cdp\left(\Zpx\right)=1$. As in \Cref{prop:Dirichlet_k1_Kp}, there is an identification:
	\begin{equation*}
	S^0_{K(1)}(p^v)^{h\chi}\simeq \map_{\Zp}\left(M(\Zp[\chi]),\Kp\right)^{h\Zpx},
	\end{equation*}
	Using the resolution in $\eqref{eqn:c_res}$, we get the $E_2$-page of the HESS:
	\begin{equation}\label{eqn:hess_Zpx_pv}
	E^{s,2t}_2=H_c^s(\Zpx;\Zp^{\otimes t}[\chi^{-1}])=\left\{\begin{array}{cl}
	\Zp[\chi]\left/\left(\chi(g)-g^t\right)\right.,&s=1;\\
	0,&\textup{otherwise,} 
	\end{array}\right.
	\end{equation}
	where $g$ is a topological generator of $\Zpx$.
	\begin{lem}\label{lem:gt_chi_g_odd} Let $\chi|_{\zpx}=\omega^a$. Then
		\begin{equation*}
		\Zp[\chi]\left/\left(\chi(g)-g^{t}\right)\right.=\left\{\begin{array}{cl}
		\Z/p,& t\equiv a \mod (p-1);\\
		0,&\textup{otherwise.}
		\end{array}\right.
		\end{equation*}
	\end{lem}
	\begin{proof}
		Since $\chi$ is primitive, we have $\chi(g)=\chi|_{\zpx}(g)\cdot \zeta_{p^{v-1}}=\omega^a(g)\zeta_{p^{v-1}}$. Rewrite $\chi(g)-g^{t}$ as 
		\begin{equation*}
		g^{t}-\chi(g)=g^t-\omega^a(g)\zeta_{p^{v-1}}=\omega^a(g)(1-\zeta_{p^{v-1}})+g^t-\omega^a(g).
		\end{equation*}
		As $1-\zeta_{p^{v-1}}$ is a uniformizer of $\Zp[\chi]\simeq\Zp[\zeta_{p^{v-1}}]$, $g^{t}-\chi(g)$ is invertible whenever $g^t-\omega^a(g)$ is. This happens when $t\not\equiv a\mod (p-1)$. When $t\equiv a\mod (p-1)$, $v_p(g^t-\omega^a(g))\ge 1> v_p(1-\zeta_{p^{v-1}})$, yielding
		\begin{equation*}
		(g^{t}-\chi(g))=(1-\zeta_p^{v-1})\implies \Zp[\chi]\left/\left(\chi(g)-g^{t}\right)\right.\simeq \Z/p.
		\end{equation*}
	\end{proof}
	Again let $\chi|_{\zpx}=\omega^a$. The spectral sequence collapses at the $E_2$-page and we conclude:
	\begin{equation}\label{eqn:pi_DK1_pv}
	\pi_i\left(S^0_{K(1)}(p^v)^{h\chi}\right)=\left\{\begin{array}{cl}
	\Z/p,&i=2(a+k(p-1))-1;\\
	0,&\textup{otherwise.}
	\end{array}\right.
	\end{equation}
	We can also compute $\pi_*\left(S^0_{K(1)}(p^v)^{h\chi}\right)$ from the following identification: 
	\begin{prop}\label{prop:S_k1_pv_cofib}
		$S^0_{K(1)}(p^v)^{h\chi}\simeq\mathrm{Cofib}\left(S^0_{K(1)}(p^{v-1})\to S^0_{K(1)}(p^v)\right)^{h\chi|_{\zpx}}$.
	\end{prop}
	\begin{proof}
	As $\chi$ is primitive of conductor $p^v$, $\Zp[\chi]=\Zp[\zeta_{p^{v-1}}]$ as an algebra. Recall from \eqref{eqn:Moore_p^v}, $M((\Zp[\zeta_{p^{v-1}}])$ is the $p$-completion of the integral Moore spectrum $M(\Z[\zeta_{p^{v-1}}])$ with a prescribed $C_{p^{v-1}}$-action. 
	As the Moore spectrum is defined by 
	\begin{equation*}
		M(\Z[\zeta_{p^{v-1}}])=\Sigma^{-1}\mathrm{Cofib}\left(S^0\bigwedge (C_{p^{v-1}})_+\longrightarrow S^0\bigwedge (C_{p^{v-2}})_+ \right),
	\end{equation*}
	we have
	\begin{align*}
		S^0_{K(1)}(p^v)^{h\chi}&=\map_{\Zp}\left(M(\Zp[\chi]),S^0_{K(1)}(p^v)\right)^{h\zx{p^v}}\\
		\simeq & \map_{\Zp}\left(S^0_{\chi|_{\zpx}},\map_{\Zp}\left(M(\Zp[\zeta_{p^{v-1}}]),S^0_{K(1)}(p^v)\right)^{h\Z/p^{v-1}}\right)^{h\zpx}\\
		\simeq &\left(\mathrm{Cofib}\left(\map((C_{p^{v-2}})_+,S^0_{K(1)}(p^v))\to\map((C_{p^{v-1}})_+,S^0_{K(1)}(p^v))\right)^{h\Z/p^{v-1}}\right)^{h\chi|_{\zpx}}\\
		\simeq &\mathrm{Cofib}\left(S^0_{K(1)}(p^{v-1})\to S^0_{K(1)}(p^v)\right)^{h\chi|_{\zpx}}.
	\end{align*}
	In the last step, we used the facts that
	\begin{align*}
		\map((C_{p^{v-2}})_+,S^0_{K(1)}(p^v))^{h\Z/p^{v-1}}\simeq S^0_{K(1)}(p^v)^{h(\Z/p^{v-1})/(\Z/p^{v-2})}&\simeq S^0_{K(1)}(p^v)^{h\Z/p}\simeq S^0_{K(1)}(p^{v-1})\\
		\map((C_{p^{v-1}})_+,S^0_{K(1)}(p^v))^{h\Z/p^{v-1}}\simeq S^0_{K(1)}(p^v)^{h(\Z/p^{v-1})/(\Z/p^{v-1})}&\simeq S^0_{K(1)}(p^v)
	\end{align*}
	\end{proof}
	Using the long exact sequence associated to this cofiber sequence, we can recover \eqref{eqn:pi_DK1_pv} from the homotopy groups $\pi_*\left(S^0_{K(1)}(p^v)\right)$ in \eqref{eqn:pi_Sk1_pv} and the spectral sequence \eqref{HESS}. Another consequence of this identification is the following:
	\begin{cor}\label{cor:Kp_Dk1}
		$\left(\Kp\right)_*\left(S^0_{K(1)}(p^{v})^{h\chi}\right)\simeq \hom_{\Zp}\left(\Zp[\chi],\left(\Kp\right)_*\right)$ as $\Zpx$-$\left(\Kp\right)_*$-modules.
	\end{cor}
	Let $\chi_1$ and $\chi_2$ be two $p$-adic Dirichlet characters of conductors $p^{v_1}$ and $p^{v_2}$, respectively. From \eqref{eqn:pi_DK1_pv}, $\pi_i\left(S^0_{K(1)}(p^{v_1})^{h\chi_1}\right)\simeq \pi_i\left(S^0_{K(1)}(p^{v_2})^{h\chi_2}\right)$ whenever $\chi_1|_{\zpx}=\chi_2|_{\zpx}$ and $v_1,v_2>1$. But this does NOT imply $S^0_{K(1)}(p^{v_1})^{h\chi_1}\simeq S^0_{K(1)}(p^{v_2})^{h\chi_2}$ as spectra.
	\begin{prop}\label{prop:comparsion_d_k1_pv}
		$S^0_{K(1)}(p^{v_1})^{h\chi_1}\simeq S^0_{K(1)}(p^{v_2})^{h\chi_2}$ iff $\chi_1|_{\zpx}=\chi_2|_{\zpx}$ and $v_1=v_2$.
	\end{prop}
	\begin{proof}
		This follows from \Cref{cor:Kp_Dk1} and the lemma below.
	\end{proof}
		\begin{lem}[{Bousfield, \cite{Bousfield_K_local}}]\label{lem:Sp_k1_alg_p}
			Two $K(1)$-local spectra $X$ and $Y$ are equivalent iff $\left(\Kp\right)_*X\simeq \left(\Kp\right)_*Y$ as $\Zpx$-$\left(\Kp\right)_*$-modules.
		\end{lem}
		\begin{proof}
			The only if direction is clear. Let $f\colon \left(\Kp\right)_*X\simto \left(\Kp\right)_*Y$ be an isomorphism of $\Zpx$-$\left(\Kp\right)_*$-modules. There is a HFPSS to compute $\pi_*\left(\map(X,Y)\right)$:
			\begin{equation*}
			E_2^{s,t}=H_c^s\left(\Zpx;\left(\Kp\right)_t\left(\map(X,Y)\right)\right)\Longrightarrow \pi_{t-s}\left(\map(X,Y)\right).
			\end{equation*}
			The isomorphism $f$ is an element of $E_2^{0,0}$, since it is isomorphic to $\hom_{\left(\Kp\right)_*}\left(\left(\Kp\right)_*X, \left(\Kp\right)_*Y\right)^{\Zpx}$. The HFPSS collapses on the $E_2$-page, as $\mathrm{cd}(\Zpx)=1$. This implies $f\in E_2^{0,0}$ is a permanent cycle, and is represented by a map of spectra $\alpha\colon  X\to Y$ such that $\left(\Kp\right)_*\alpha=f$. 
			
			We claim $\alpha$ is a weak equivalence. As $f=\left(\Kp\right)_*\alpha$ is an isomorphism of $\Zpx$-$\left(\Kp\right)_*$-modules, $\alpha$ induces an isomorphism on the $E_2$-page of the HFPSS to compute $\pi_*(X)$ and $\pi_*(Y)$. It now follows from \cite[Theorem 5.3]{Boardman_ccss} that $\alpha$ is a weak equivalence.
		\end{proof}

	\subsection{The $N=2^v$ case}
	We start with the $N=4$ case, when the only non-trivial $2$-adic Dirichlet character of conductor $4$ is the \Teichmuller character $\omega\colon \zx{4}\to\Z_2^\x$. By \Cref{prop:Dirichlet_k1_Kp}, the Dirichlet $K(1)$-local sphere is identified with
	\begin{align}\label{eqn:hes_lift_p2}
		S_{K(1)}^0(4)^{h\omega}\simeq& \map_{\Z_2}(M(\Z_2[\omega]),KU^\wedge_{2})^{h\Z_2^\x}\\
		\simeq&\left(\map_{\Z_2}(M(\Z_2[\omega]),KU^\wedge_{2})^{h\{\pm 1\}}\right)^{h(1+4\Z_2)}\nonumber\\&=\left(\left(KU^\wedge_2\right)^{h\omega}\right)^{h(1+4\Z_2)}.\nonumber
	\end{align}
	Parallel to the computation of the classical $K(1)$-local sphere at $p=2$ in \Cref{Subsec:HFPSS}, we will first identify $\left(KU^\wedge_2\right)^{h\omega}$ geometrically.
	\begin{prop}\label{prop:imaginary_ko}
		Let $\sigma$ be the sign representation of $C_2$ on $\Z$. Define $\KR^{h\sigma}$ to be the homotopy $\sigma$-eigen-spectrum of Atiyah's $C_2$-equivariant $\KR$-spectrum in \textup{\cite{Atiyah_KR}}. Then we have an identification:
		\begin{equation*}
			\KR^{h\sigma}= \map(M(\Z[\sigma]),\KR)^{hC_2}\simeq \Sigma^2 KO.
		\end{equation*}
	\end{prop}
	\begin{proof}
		By \Cref{fig:moore_cell}, $M(\Z[\sigma])$ is $C_2$-equivariantly equivalent to $S^{\sigma-1}$. Now using the $(1+\sigma)$-periodicity of $\KR$  \cite[Theorem 2.1]{Atiyah_KR}, we have a $C_2$-equivalence
		\begin{equation*}
			\map(S^{\sigma-1},\KR)\simeq \Sigma^{1-\sigma}\KR \simeq \Sigma^2\KR.
		\end{equation*}
		The claim now follows from the equivalence $\KR^{hC_2}\simeq KO$. 
	\end{proof}
	\begin{rem}\label{rem:KR_FP}
		This statement depends on the actual model of $M(\Z[\sigma])$. If we started with $S^{1-\sigma}$, where $C_2$ also acts by the sign representation on $\pi_*(S^0)$, we would have
		\begin{equation*}
			\map(S^{1-\sigma},\KR)^{hC_2}\simeq \Sigma^{-2}KO.
		\end{equation*}
		In terms of the HFPSS computations, the $E_2$-pages of $\map(S^{\sigma-1},\KR)^{hC_2}$ and $\map(S^{1-\sigma},\KR)^{hC_2}$ are the same. The difference is the $d_3$-differentials, which are invisible in algebra. Likewise, one can check the HFPSS for
		\begin{equation*}
			\map(S^{2\sigma-2},\KR)^{hC_2}\simeq \Sigma^4KO\simeq KSp
		\end{equation*}
		has the same $E_2$-page as that for $\KR^{hC_2}\simeq KO$. Again the difference is the $d_3$-differentials that are invisible in algebra.
		\begin{figure}[ht]
			\centering
			\def\drawnodiff#1{
				\ifnum#1=0
				\else
				\ssdrop{\square} 
				\ssmove{1}{1} \ssdropbull
				\ssmove{1}{1} \ssdropbull
				\ssmove{6}{-2}
				\cnti=#1
				\advance \cnti by -1
				\drawnodiff
				{\the\cnti}
				\fi
			}
			\def\drawdiff#1{
			\ifnum#1=0
			\else
			\ssdrop{\color{red}\bullet} \ssname{a}
			\ssmove{-1}{3} 
			\ssdrop{\color{red}\bullet} \ssname{b}
			\ssgoto{a} \ssgoto{b}
			\ssstroke[color=red]
			\ssarrowhead
			\ssmove{2}{-2}
			\cntii=#1
			\advance \cntii by -1
			\drawdiff
			{\the\cntii}
			\fi
			}
			\def\drawboxdiff#1{
				\ifnum#1=0
				\else
				\ssdrop{\square} \ssname{a}
				\ssmove{-1}{3} 
				\ssdrop{\color{red}\bullet} \ssname{b}
				\ssgoto{a} \ssgoto{b}
				\ssstroke[color=red]
				\ssarrowhead
				\ssmove{2}{-2}
				\drawdiff{8}
				\ssmove{-1}{-9}
				\cnti=#1
				\advance \cnti by -1
				\drawboxdiff
				{\the\cnti}
				\fi
			}
			\begin{subfigure}[b]{.48\linewidth}
				\centering \begin{sseq}{-8...8}{0...8}				
					\ssmoveto{-8}{0}
					\drawnodiff{3}
					\ssresetprefix
					\ssmoveto{-12}{0}
					\drawboxdiff{3}		
				\end{sseq}
				\caption{$\left(\Sigma^{4k(1-\sigma)}\KR\right)^{hC_2}\simeq \Sigma^{8k}KO\simeq KO$}
			\end{subfigure}
			\begin{subfigure}[b]{.48\linewidth}
				\centering\begin{sseq}{-8...8}{0...8}				
					\ssmoveto{-4}{0}
					\drawnodiff{3}
					\ssresetprefix
					\ssmoveto{-16}{0}
					\drawboxdiff{4}		
				\end{sseq}
				\caption{$\left(\Sigma^{(4k+2)(1-\sigma)}\KR\right)^{hC_2}\simeq \Sigma^{8k+4}KO\simeq\Sigma^4KO$}
			\end{subfigure}
			\begin{subfigure}[b]{.48\linewidth}
				\centering\begin{sseq}{-8...8}{0...8}				
					\ssmoveto{-6}{0}
					\drawnodiff{3}
					\ssresetprefix
					\ssmoveto{-10}{0}
					\drawboxdiff{3}		
				\end{sseq}
				\caption{$\left(\Sigma^{(4k+1)(1-\sigma)}\KR\right)^{hC_2}\simeq\Sigma^{8k+2}KO\simeq \Sigma^2KO$}
			\end{subfigure}
			\begin{subfigure}[b]{.48\linewidth}
				\centering\begin{sseq}{-8...8}{0...8}				
					\ssmoveto{-10}{0}
					\drawnodiff{3}
					\ssresetprefix
					\ssmoveto{-14}{0}
					\drawboxdiff{4}		
				\end{sseq}
				\caption{$\left(\Sigma^{(4k+3)(1-\sigma)}\KR\right)^{hC_2}\simeq \Sigma^{8k+6}KO\simeq\Sigma^{6}KO$}
			\end{subfigure}
			\caption{$d_3$-differentials in the HFPSS for different $C_2$-actions on the complex $K$-theory spectrum }
			{(Adams grading. $\square=\Z$ and $\bullet=\Z/2$. (A) and (B) are the same as Figures 3 and 6 in \cite{KO_duality}.)}
		\end{figure}
	\end{rem}
	\begin{rem}
		A more geometric construction is the following. For any compact space $X$, $\KR^{h\sigma}(X)$ consists of virtual complex vector bundles $[E]$ over $X$ such that $\psi^{-1}([E])=[\overline{E}]=-[E]$. For any such virtual vector bundle, its tensor product with the complexification of a real vector also satisfies this condition. As a result, $\KR^{h\sigma}$ is a $KO$-module spectrum.
		
		Let $\xi$ be the tautological complex line bundle over $\CP^{1}\simeq S^2$. Then $[\xi]-[\overline{\xi}]\in \KR^{h\sigma}(S^2)$. \Cref{prop:imaginary_ko} implies the external tensor product with $\xi-\overline{\xi}$ induces an isomorphism:
		\begin{equation*}
			(\xi-\overline{\xi})\boxtimes (-)_\Cbb\colon  KO(X)\simto \KR^{h\sigma}(S^2\x X).
		\end{equation*}
		As elements in $\KR^{h\sigma}(X)$ satisfy $[\overline{E}]=-[E]$, $\KR^{h\sigma}$ can be thought of as the \emph{purely imaginary} $K$-theory, compared to the \emph{real} $K$-theory $KO\simeq \KR^{hC_2}$.
	\end{rem}
	\begin{cor}
		$\left(KU^\wedge_2\right)^{h\omega}\simeq \Sigma^2 KO^\wedge_2$ and its homotopy groups are given by:
		\begin{equation*}
		\begin{array}{c|c|c|c|c|c|c|c|c}
		i\mod 8& 0&1&2&3&4&5&6&7\\\hline
		\pi_i\left(\left(KU^\wedge_2\right)^{h\omega}\right)&0&0&\Z_2&\Z/2&\Z/2&0&\Z_2&0
		\end{array}
		\end{equation*}
	\end{cor}
	\begin{rem}
		The equivalence $\left(KU^\wedge_2\right)^{h\omega}\simeq \Sigma^2 KO^\wedge_2$ is NOT $(1+4\Z_2)$-equivariant.
	\end{rem}
	The next step is to compute the HFPSS:
	\begin{equation*}
		E_2^{s,t}=H_c^s\left(1+4\Z_2;\pi_t\left(\left(KU^\wedge_2\right)^{h\omega}\right)\right)\Longrightarrow \pi_{t-s}\left(S_{K(1)}^0(4)^{h\omega}\right).
	\end{equation*}
	Let $g\in 1+4\Z_2$ be a topological generator. Descending the Adams operations on $KU^\wedge_2$ to $\left(KU^\wedge_2\right)^{h\omega}$, we get $g$ acts on $\pi_{4t+2}\left(\left(KU^\wedge_2\right)^{h\omega}\right)$ by $g^{2t+1}$. The actions on the $\Z/2$-terms are trivial since $\Z/2$ has only trivial automorphism. Using the continuous resolution \eqref{eqn:c_res}, we compute the $E_2$-page of the HFPSS:
	\begin{equation}\label{eqn:hfpss_d_k1_4}
		E_2^{s,t}=H_c^s\left(1+4\Z_2;\pi_t\left(\left(KU^\wedge_2\right)^{h\omega}\right)\right)=\left\{\begin{array}{cl}
		\Z/4,&s=1, t\equiv 2\mod 4;\\
		\Z/2,& s=0,1, t\equiv 3,4\mod 8;\\
		0,&\textup{otherwise.}
		\end{array}\right.
	\end{equation}
	\begin{prop}\label{prop:extn_prob_tw}
		The extension problems of this spectral sequence are trivial.
	\end{prop}
	\begin{proof}
		We need solve the extension problems at $t-s\equiv 3 \mod 8$. The argument here is analogous to \Cref{prop:extn_prob}. As $\left(KU^\wedge_2\right)^{h\omega}\simeq\Sigma^2 KO^\wedge_2 $ is a $KO^\wedge_2$-module spectrum, we denote the non-zero element in $\pi_3\left(\left(KU^\wedge_2\right)^{h\omega}\right)$ by $\Sigma^2\eta$. This is an element of order $2$ and represents a permanent cycle in $E^{0,1}_2$ of \eqref{eqn:hfpss_d_k1_4}. As $\Sigma^2\eta$ represents an element of order $2$ in $\pi_3\left(S_{K(1)}^0(4)^{h\omega}\right)$, the extension problem is trivial. For general $t-s=8k+3$, replace $\Sigma^2\eta$ by $\beta^t\cdot\Sigma^2\eta$ in the argument above, where $\beta\in \pi_8(KO^\wedge_2)$ is the Bott element.
	\end{proof}
	From this, we conclude:
	\begin{equation}\label{eqn:Dirichlet_k1_4}
		\pi_i\left(S_{K(1)}^0(4)^{h\omega}\right)=\left\{\begin{array}{cl}
		\Z/4,&i\equiv 1\mod 4;\\
		\Z/2,&i\equiv 2,4\mod 8;\\
		\Z/2\oplus \Z/2,&i\equiv 3\mod 8;\\
		0,&\textup{otherwise}.
		\end{array}\right.
	\end{equation}
	We also record the $E_2$-page of the HESS associated to \eqref{eqn:hes_lift_p2}:
	\begin{equation}\label{eqn:hess_Zpx_4}
	\ext_{\Z_2\llb\Z_2^\x\rrb}^{s}\left((\Z_2)_{\wtomega},\left(KU^\wedge_2\right)_t\right)=\left\{\begin{array}{cl}
	\Z/4, & s=1, t\equiv 2\mod 4;\\
	\Z/2, & s>1, t\equiv 2\mod 4;\\
	\Z/2, & s>0, 4\mid t;\\
	0,& \textup{otherwise.}
	\end{array}\right.
	\end{equation}
	\begin{rem}\label{rem:exotic_pic}
		As explained in \Cref{rem:C2_action_S0}, we could have chosen $M(\Z[\zeta_2])=S^{1-\sigma}$ when defining the Dirichlet $J$-spectra and $K(1)$-local spheres. Denote the resulting homotopy eigen-spectra by 
		\begin{equation*}
			X^{h'\omega}=\map_{\Z_2}(S^{1-\sigma},X)^{hC_2},
		\end{equation*}
		where $\omega\colon C_2\simeq \zx{4}\to\Z_2^\x$ is the $2$-adic Teichm\"uller character. Then by \Cref{rem:KR_FP}, $\left(KU^\wedge_{2}\right)^{h'\omega}\simeq \Sigma^{-2}KO^\wedge_{2}$. A similar computation as above yields:
		\begin{equation*}
			\pi_i\left(S_{K(1)}^0(4)^{h'\omega}\right)=\left\{\begin{array}{cl}
			\Z/4,&i\equiv 1\mod 4;\\
			\Z/2,&i\equiv -2,0\mod 8;\\
			\Z/2\oplus \Z/2,&i\equiv -1\mod 8;\\
			0,&\textup{otherwise}.
			\end{array}\right.
		\end{equation*}
		Note that $\pi_{2k-1}\left(S^0_{K(1)}(4)^{h\chi}\right)=\pi_{2k-1}\left(S^0_{K(1)}(4)^{h'\chi}\right)$ when $(-1)^k=\chi(-1)$.
		
		Both $S_{K(1)}^0(4)^{h\omega}$ and $S_{K(1)}^0(4)^{h'\omega}$ are elements of order $4$ in the $K(1)$-local Picard group $\Pic_{K(1)}$ at prime $2$. Their difference in $\Pic_{K(1)}$ is the \textbf{exotic $K(1)$-local sphere} $\Ecal_{K(1)}$, an element whose HFPSS has the same $E_2$-page as that for the $K(1)$-local sphere.\footnote{The spectrum $\Ecal_{K(1)}$ has been referred to by different letters (like $P$) in the literature. Here we use the letter $\Ecal$ since it stands for exotic. In \Cref{thm:K1_BC_dual}, we will see it is the "error term" in the $K(1)$-local Brown-Comenetz duality at prime $2$.} One construction of this element is given in \cite[Section 9]{KO_duality}. It is also the $K(1)$-localization of a finite $CW$-spectra $\Ecal$, described in \Cref{thm:E1_BC_dual}. By identifying $\Ecal_{K(1)}$ with $\left(KSp^\wedge_{2}\right)^{h(1+4\Z_2)}$, we can compute its homotopy groups as in \eqref{pis0_2}:
		\begin{equation}\label{eqn:pi_Ek1}
			\pi_i\left(\Ecal_{K(1)}\right)=\left\{\begin{array}{cl}
			\Z_2, & i=0,-1;\\
			\Z/2, & i\equiv 4, 6 \mod 8;\\
			\Z/2\oplus \Z/2, & i\equiv 5 \mod 8;\\			
			\Z/2^{v_2(k)+3}, & i=4k-1\neq -1;\\
			0,& \textup{otherwise.}
			\end{array}\right.
		\end{equation}
	\end{rem}	
	When $p=2$ and $N=2^v>4$, we apply \Cref{prop:Dirichlet_k1_Kp} as before:
	\begin{equation*}
	S^0_{K(1)}(2^v)^{h\chi}\simeq \map_{\Z_2}\left(M(\Z_2[\chi]),KU^\wedge_2\right)^{h\Z_2^\x}.
	\end{equation*}
	\begin{lem}
		$S^0_{K(1)}(2^v)^{h\chi}\simeq\map_{\Z_2}\left(M(\Z_2[\chi]),\left(KU^\wedge_2\right)^{h\chi|_{\zx{4}}}\right)^{h(1+4\Z_2)}$.
	\end{lem}
	\begin{proof}
		We prove the claim by breaking the $\Z_2^\x$-homotopy fixed points into two steps.
		\begin{align*}
			S^0_{K(1)}(2^v)^{h\chi}\simeq& \map_{\Z_2}\left(M(\Z_2[\chi]),KU^\wedge_2\right)^{h\Z_2^\x}\\
			\simeq & \map_{\Z_2}\left(M(\Z_2[\chi]),\map_{\Z_2}\left(S^0_{\chi|_{\zx{4}}},KU^\wedge_2\right)^{h\zx{4}}\right)^{h(1+4\Z_2)}\\
			\simeq& \map_{\Z_2}\left(M(\Z_2[\chi]),\left(KU^\wedge_2\right)^{h\chi|_{\zx{4}}}\right)^{h(1+4\Z_2)}.
		\end{align*}
		Here, $S^0_{\omega}=S^{\sigma-1}$ and $S^0_{\omega^0}=S^{0}$. In the third line, we used the fact $\chi\cdot\chi|_{\zx{4}} $ is trivial when restricted to $\zx{4}$ and is equal to $\wtchi$ when restricted to $1+4\Z_2$ .
	\end{proof}
	Let $g$ be a topological generator of $1+4\Z_2$. Denote by $\mathrm{Ann}\left(\wtchi(g)-1\right)$ the ideal of annihilators of $\wtchi(g)-1$ in $ \Z_2[\chi]/(2)$. The computation now splits into two subcases depending on the parity of $\chi$:
	\begin{itemize}
		\item When $\chi(-1)=1$, $\left(KU^\wedge_2\right)^{h\chi|_{\zx{4}}}\simeq KO^\wedge_2$. By \eqref{eqn:c_res} and \eqref{eqn:pi_ko}, $E_2$-page of the HESS is:
		\begin{align}
		E_2^{s,t}&=\ext^s_{\Z_2\llb 1+4\Z_2\rrb}\left(\Z_2[\chi],\pi_t\left(KO^\wedge_2\right)\right)\nonumber\\\label{eqn:hess_Zpx_2v_even}&=\left\{\begin{array}{cl}
		\Z_2[\chi]\left/\left(\wtchi(g)-g^{2k}\right)\right.,& s=1, t=4k;\\
		\mathrm{Ann}\left(\wtchi(g)-1\right),& s=0, t\equiv 1,2\mod 8;\\
		\Z_2[\chi]\left/\left(2,\wtchi(g)-1\right)\right.,&s=1, t\equiv 1,2\mod 8;\\
		0,&\textup{otherwise.}
		\end{array}\right.
		\end{align}
		\item When $\chi(-1)=-1$, $\left(KU^\wedge_2\right)^{h\chi|_{\zx{4}}}\simeq \Sigma^2 KO^\wedge_2$ by \Cref{prop:imaginary_ko}. The $E_2$-page of the HESS is:
		\begin{align}
		E_2^{s,t}&=\ext^s_{\Z_2\llb 1+4\Z_2\rrb}\left(\Z_2[\chi],\pi_t\left(\Sigma^2 KO^\wedge_2\right)\right)\nonumber\\\label{eqn:hess_Zpx_2v_odd}&=\left\{\begin{array}{cl}
		\Z_2[\chi]\left/\left(\wtchi(g)-g^{2k+1}\right)\right.,& s=1, t=4k+2;\\
		\mathrm{Ann}\left(\wtchi(g)-1\right),& s=0, t\equiv 3,4\mod 8;\\
		\Z_2[\chi]\left/\left(2,\wtchi(g)-1\right)\right.,&s=1, t\equiv 3,4\mod 8;\\
		0,&\textup{otherwise.}
		\end{array}\right.
		\end{align}
	\end{itemize}
	In both cases, the spectral sequences collapse  at the $E_2$-pages. Analogous to \Cref{prop:extn_prob} (\Cref{prop:extn_prob_tw}), the extension problems at $t-s\equiv 1 \mod 8$ ($t-s\equiv 3 \mod 8$, resp.) are trivial. We further simplify the formulas using the following facts about $\Z_2[\chi]$ from \Cref{prop:gal_Qp_zeta}.
	\begin{lem}\label{lem:gt_chi_g_2}
		Let $\chi$ be a primitive $2$-adic Dirichlet character of conductor $2^v\ge 8$. Let $g$ be a topological generator of $1+4\Z_2$.
		\begin{enumerate}
			\item $\Z_2[\chi]$ is a totally ramified extension of $\Z_2$ of ramification index $2^{v-3}$.
			\item $1-\wtchi(g)$ is a uniformizer of $\Z_2[\chi]$ and $\Z_2[\chi]/(1-\wtchi(g))\simeq \Z/2$.
			\item The ideal of annihilators of $\wtchi(g)-1\in \Z_2[\chi]/(2)$ is isomorphic to $\Z/2$.
			\item $\Z_2[\chi]/(\wtchi(g)-g^k)=\Z/2$ for any $k$.
		\end{enumerate}
	\end{lem}
	\begin{proof}
		Only (4) needs a proof. Notice $\wtchi(g)=\zeta_{2^{v-2}}$ since $\chi$ is primitive. Write $\wtchi(g)-g^k=\wtchi(g)-1+1-g^k$. By (2), $\wtchi(g)-1$ is a uniformizer. On the other hand $v_2(1-g^k)\ge 2>v_2(\wtchi(g)-1)$, since $g\equiv 1\mod 4$. This implies:
		\begin{equation*}
		(\wtchi(g)-g^k)=(\wtchi(g)-1)\implies \Z_2[\chi]/(\wtchi(g)-g^k)=\Z/2.
		\end{equation*}
	\end{proof}
	\begin{prop}\label{prop:pi_dk1_2v}
		When $\chi(-1)=1$, we have		
		\begin{equation}\label{eqn:pi_dk1_2v_even}
		\pi_i\left(S^0_{K(1)}(2^v)^{h\chi}\right)=\left\{\begin{array}{cl}
		\Z/2,&i\equiv 0,2,3,7\mod 8;\\
		\Z/2\oplus \Z/2, &i\equiv 1\mod 8;\\
		0,&\textup{otherwise.}
		\end{array}\right.
		\end{equation}
		When $\chi(-1)=-1$, we have
		\begin{equation}\label{eqn:pi_dk1_2v_odd}
		\pi_i\left(S^0_{K(1)}(2^v)^{h\chi}\right)=\left\{\begin{array}{cl}
		\Z/2,&i\equiv 1,2,4,5\mod 8;\\
		\Z/2\oplus \Z/2, &i\equiv 3\mod 8;\\
		0,&\textup{otherwise.}
		\end{array}\right.
		\end{equation}
	\end{prop}
	\begin{rem}\label{rem:j_2v_Moore_model}The computations above depend on the actual model of the $C_2$-actions on the Moore spectra:
		\begin{itemize}
			\item When $\chi(-1)=1$, if we choose $S^{2-2\sigma}$ as a model for the $C_2$-action on $S^0$ with trivial induced action on $\pi_*$, \eqref{eqn:pi_dk1_2v_even} becomes:
			\begin{equation*}
			\pi_i\left(S^0_{K(1)}(2^v)^{h'\chi}\right)=\left\{\begin{array}{cl}
			\Z/2,&i\equiv 3,4,6,7\mod 8;\\
			\Z/2\oplus \Z/2, &i\equiv 5\mod 8;\\
			0,&\textup{otherwise.}
			\end{array}\right.
			\end{equation*}
			\item When $\chi(-1)=-1$, if we choose $S^{1-\sigma}$ as a model for the $C_2$-action on $S^0$ that induces sign representations on $\pi_*$, \eqref{eqn:pi_dk1_2v_odd} becomes:
			\begin{equation*}
			\pi_i\left(S^0_{K(1)}(2^v)^{h'\chi}\right)=\left\{\begin{array}{cl}
			\Z/2,&i\equiv 0,1,5,6\mod 8;\\
			\Z/2\oplus \Z/2, &i\equiv 7\mod 8;\\
			0,&\textup{otherwise.}
			\end{array}\right.
			\end{equation*}
		\end{itemize}
		Note that $\pi_{2k-1}\left(S^0_{K(1)}(2^v)^{h\chi}\right)=\pi_{2k-1}\left(S^0_{K(1)}(2^v)^{h'\chi}\right)$ when $(-1)^k=\chi(-1)$.
	\end{rem}
	Like the odd prime case, we can recover the results in \Cref{prop:pi_dk1_2v} using the following identification:
	\begin{prop}\label{prop:S_k1_2v_cofib}
		Let $\chi$ be a primitive $2$-adic Dirichlet character of conductor $2^v>4$. When $\chi(-1)=1$, we have
		\begin{equation*}
		S^0_{K(1)}(2^v)^{h\chi}\simeq \mathrm{Cofib}\left((KO^\wedge_{2})^{h (1+2^{v-1}\Z_2)}\to (KO^\wedge_{2})^{h (1+2^{v}\Z_2)} \right).
		\end{equation*}
		If $\chi(-1)=-1$, then
		\begin{equation*}
		S^0_{K(1)}(2^v)^{h\chi}\simeq \mathrm{Cofib}\left(\left(\left(KU^\wedge_2\right)^{h\omega}\right)^{h (1+2^{v-1}\Z_2)}\to \left(\left(KU^\wedge_2\right)^{h\omega}\right)^{h (1+2^{v}\Z_2)} \right).
		\end{equation*}
	\end{prop}
	\begin{proof}
		This is the same as proof of \Cref{prop:S_k1_pv_cofib}.
	\end{proof}
	\begin{cor}\label{cor:KO_Dk1}
		There is an equivalence of $(1+4\Z_2)$-$\left(KO^\wedge_{2}\right)_*$-modules:
		\begin{equation*}
			\left(KO^\wedge_{2}\right)_*\left(S^0_{K(1)}(2^v)^{h\chi}\right)\simeq\left\{\begin{array}{cl}
			\hom\left(\Z_2[\chi],\left(KO^\wedge_{2}\right)_*\right) & \chi(-1)=1;\\
			\hom\left(\Z_2[\chi],\left(KU^\wedge_{2}\right)^{h\omega}_*\right) & \chi(-1)=-1,\\
			\end{array}\right.
		\end{equation*}
	\end{cor}
	This computation leads to a $p=2$ version of \Cref{prop:comparsion_d_k1_pv}. Let $\chi_1$ and $\chi_2$ be two $2$-adic Dirichlet characters of conductors $2^{v_1}$ and $2^{v_2}$, respectively. From \Cref{prop:pi_dk1_2v}, $\pi_i\left(S^0_{K(1)}(2^{v_1})^{h\chi_1}\right)\simeq \pi_i\left(S^0_{K(1)}(2^{v_2})^{h\chi_2}\right)$ whenever $\chi_1(-1)=\chi_2(-1)$ and $v_1,v_2>2$. But this DOES NOT imply $S^0_{K(1)}(2^{v_1})^{h\chi_1}\simeq S^0_{K(1)}(2^{v_2})^{h\chi_2}$ as spectra. We have the $p=2$ version of \Cref{prop:comparsion_d_k1_pv}:
	\begin{prop}\label{prop:comparsion_d_k1_2v}
		$S^0_{K(1)}(2^{v_1})^{h\chi_1}\simeq S^0_{K(1)}(2^{v_2})^{h\chi_2}$ iff $\chi_1(-1)=\chi_2(-1)$ and $v_1=v_2$.
	\end{prop}
	\begin{proof}
		This follows from \Cref{cor:KO_Dk1} and the $p=2$ version of \Cref{lem:Sp_k1_alg_p}.
	\end{proof}
	\begin{lem}\label{lem:Sp_k1_alg_2}
		Let $X$ and $Y$ be two $K(1)$-local spectra and $p=2$. Then $X\simeq Y$ iff $\left(KO^\wedge_{2}\right)_*X\simeq \left(KO^\wedge_{2}\right)_*Y$ as $(1+4\Z_2)$-$\left(KO^\wedge_{2}\right)_*$-modules.
	\end{lem}
	\subsection{The $N$ is not a $p$-power case} Write $N=p^v\cdot N'$, where $p\nmid N'>1$. In this case, a primitive Dirichlet character $\chi\colon \znx\to \Cpx$ factors into a product $\chi=\chi_p\cdot \chi'$, where $\chi_p$ has conductor $p^v$ and $\chi'$ has conductor $N'$. The subgroup $\zx{N'}$ of $\znx$ acts trivially on $S^0_{K(1)}(p^v)$.
	\begin{prop}\label{prop:dirichlet_k1_contractible}
		$S^0_{K(1)}(p^v)^{h\chi}$ is contractible when $|\imag \chi'|$ is not a power of $p$.
	\end{prop}
	\begin{proof}
		We claim that when $|\imag \chi'|$ is not a power of $p$, the $E_2$-page of spectral sequence \eqref{eqn:D_k1_gp_coh} to compute homotopy groups of $S^0_{K(1)}(p^v)^{h\chi}$ is zero. That is, the group cohomology
		\begin{equation*}
			H_c^s(\Zpx\x\zx{N'};\Zp^{\otimes t}[\chi^{-1}])
		\end{equation*}
		is zero for all $s\ge 0$ and $t\in \Z$. Suppose $\zeta_n\in \imag \chi'$ and $p\nmid n$. The $\zx{N'}$ contains a subgroup $C_n$ such that $\chi'|_{C_n}$ is injective. We have a Hochschild-Serre spectral sequence to compute this group cohomology:
		\begin{equation}\label{eqn:E2_rst}
			E_2^{r,s,t}=H_c^r(\Zpx\x\zx{N'}/C_n;H^s(C_n;\Zp^{\otimes t}[\chi^{-1}]))\Longrightarrow H_c^s(\Zpx\x\zx{N'};\Zp^{\otimes t}[\chi^{-1}]).
		\end{equation}
		The group cohomology of $C_n$ vanishes in positive degrees since its order is invertible in $\Zp$. As a generator $g\in C_n$ acts on $\Zp^{\otimes t}[\chi^{-1}]=\Zp[\chi]$ by multiplication by $\zeta_n^{-1}$, the fixed points of this group action is zero. This shows $E_2^{r,s,t}=0$  for all $s\ge 0$ in \eqref{eqn:E2_rst}. Consequently, the $E_2$-page of \eqref{eqn:D_k1_gp_coh} to compute $\pi_*\left(S^0_{K(1)}(p^v)^{h\chi}\right)$ is zero and the Dirichlet $K(1)$-local sphere is contractible.
	\end{proof}

	\begin{cor}\label{cor:dirichlet_k1_contractible}
		 Suppose $\chi$ is primitive character of conductor $N=p^v\cdot N'$ as above. The spectrum  $S^0_{K(1)}(p^v)^{h\chi}$ is contractible when $p\nmid \phi(N')$, in particular when:
		 \begin{enumerate}
		 	\item $N=q\neq p$ is a prime such that $p\nmid (q-1)$.
		 	\item $N=q^v>2q$ for any prime $q$ not equal to $p$.
		 \end{enumerate}
	\end{cor}
	\begin{proof}
	 $\chi'$ is a primitive character since $\chi$ is. This guarantees $\imag \chi'$ is not trivial. The assumption $p\nmid \phi(N')$ implies that $\imag \chi'$ contains no $p$-power roots of unity. The claim now follows from \Cref{prop:dirichlet_k1_contractible}.
	\end{proof}
	When $\imag \chi'$ contains only $p$-power roots of unity, the spectral sequence \eqref{eqn:D_k1_gp_coh} does not collapse on the $E_2$-page. Instead, we compute $\pi_*\left(S^0_{K(1)}(p^v)^{h\chi}\right)$ by identifying the spectrum from its construction. 
	\begin{thm}\label{thm:Dk1_ho}
		Suppose $\imag \chi'\simeq C_{p^n}$ for some $n\ge 1$. Then we have:
		\begin{equation*}
		S^0_{K(1)}(p^v)^{h\chi}\simeq\left\{\begin{array}{cll}
		\Sigma\left(S^0_{K(1)}(2p)^{h\chi|_{\zx{2p}}}\right)^{\vee p},&\textup{if }p^v\le 2p;&\textup{(Case I)}\\
		\Sigma\left(S^0_{K(1)}(p)^{h\chi|_{\zx{p}}}\right)^{\vee p},&\textup{if }n\ge v-1>0\textup{ and } p>2;&\textup{(Case II)}\\
		\Sigma\left(S^0_{K(1)}(4)^{h\chi|_{\zx{4}}}\right)^{\vee p},&\textup{if }n\ge v-2>0\textup{ and } p=2;&\textup{(Case II')}\\
		\Sigma\left(S^0_{K(1)}(p^{v-n})^{h\chi|_{\zx{p}}}\right)^{\vee p},&\textup{if }n < v-1\textup{ and } p>2; &\textup{(Case III)}\\
		\Sigma\left(S^0_{K(1)}(2^{v-n})^{h\chi|_{\zx{4}}}\right)^{\vee p},&\textup{if }n<v-2\textup{ and } p=2. &\textup{(Case III')}
		\end{array}\right.
		\end{equation*}
		Here $\zx{2p}$ is thought of as a subgroup of $\znx$ via the inclusions $\zx{2p}\subseteq \zx{p^v}\subseteq \znx$.
	\end{thm}
	\begin{proof}
	We prove the cases when $p>2$. The $p=2$ cases are similar. Recall from \Cref{con:twisted_j} that 
	\begin{equation*}
	S^0_{K(1)}(p^v)^{h\chi}=\map_{\Zp}\left(M(\Z_p[\chi]),S^0_{K(1)}(p^v)\right)^{h\znx}.
	\end{equation*}
	We need to compare $\Zp[\chi_p]$, $\Zp[\chi']$, and $\Zp[\chi]$. As $\chi_p$ is primitive, $\Zp[\chi_p]=\Zp[\zeta_{p^{v-1}}]$ as an algebra. Our assumption says $\imag \chi'\simeq C_{p^{n}}$ for some $n\ge 1$. As a result, $\Zp[\chi]=\Zp[\zeta_{p^{\max\{v-1,n\}}}]$. In particular, when $p^v\le p$, $\Zp[\chi_p]=\Zp$ and $\Zp[\chi]=\Zp[\chi']$. So there are three cases depending on $n$ and $v-1$ ($v-2$ when $p=2$).
	
	In \textbf{Case I} when $p^v\le p$ and $p>2$, we have the following identification:
	\begin{align*}
		S^0_{K(1)}(p^v)^{h\chi}&=\map_{\Zp}\left(M(\Zp[\chi]),S^0_{K(1)}(p) \right)^{h(\zpx\x \zx{N'})}\\
		\simeq &\map_{\Zp}\left(M(\Zp[\zeta_{p^n}]),\map\left(S^0_{\chi|_{\zpx}},S^0_{K(1)}(p)\right) \right)^{h(\zpx\x \zx{N'})}\\
	(*)\quad	\simeq & \map_{\Zp}\left(M(\Zp[\zeta_{p^n}])_{h\zx{N'}},\map\left(S^0_{\chi|_{\zpx}},S^0_{K(1)}(p)\right)^{h\zpx} \right)\\
		\simeq& \map_{\Zp}\left(M(\Zp[\zeta_{p^n}])_{h\zx{N'}}, S^0_{K(1)}(p)^{h\chi|_{\zpx}}\right)\\
		\simeq&	\map_{K(1)}\left((M(\Zp[\zeta_{p^n}])_{h\zx{N'}})_{K(1)}, S^0_{K(1)}(p)^{h\chi_p}\right).	
	\end{align*}
	In (*), we used the facts that $\zpx$ acts trivially on the source, and that $\zx{N'}$ acts trivially on the target. Also, notice $S^0_{K(1)}(p)^{h\chi_p}\simeq S^0_{K(1)}$ when $\chi_p$ is trivial. We now show:
	\begin{equation*}
		(M(\Zp[\zeta_{p^n}])_{h\zx{N'}})_{K(1)}\simeq \left(S^{-1}_{K(1)}\right)^{\vee p}.
	\end{equation*}	
	In \eqref{eqn:Moore_p^v}, $M((\Zp[\zeta_{p^n}])$ is defined to be the $p$-completion of the integral Moore spectrum $M(\Z[\zeta_{p^n}])$. The $p$-completion commutes with the taking homotopy orbits, since it is equivalent to smashing with $M(\Zp)$ in this case. As a result, we should first find $M(\Z[\zeta_{p^n}])_{h\zx{N'}}$. By \eqref{eqn:cofib}, we have:
	\begin{align}
	&M(\Z[\zeta_{p^{n}}])=\Sigma^{-1}\mathrm{Cofib}\left((C_{p^n})_+\longrightarrow  (C_{p^{n-1}})_+ \right)\nonumber\\
	\label{eqn:Moore_cofib_ho}\implies& M(\Z[\zeta_{p^{n}}])_{h\zx{N'}}\simeq \Sigma^{-1}\mathrm{Cofib}\left(\left((C_{p^n})_+\right)_{h\zx{N'}}\longrightarrow \left((C_{p^{n-1}})_+\right)_{h\zx{N'}} \right).
	\end{align}
	\begin{lem}\label{lem:quotient_HO}
		Let $H$ be a closed subgroup of $G$. Then $(G/H)_{hG}\simeq BH$.
	\end{lem}
	\begin{proof}
		By \mbox{\cite[(1.6)]{May_alaska_notes}}, there is a $G$-homeomorphism for any $G$-space $X$:
		\begin{equation*}
			G\times_H X\cong (G/H)\times X.
		\end{equation*}
		Set $X=EG$ to be the classifying space of $G$ and take $G$-orbits on both sides of the homeomorphism, we get:
		\begin{equation*}
			(G\times_H EG)_G\cong ((G/H)\times EG)_G.
		\end{equation*}
		The right hand side of the homeomorphism is the homotopy orbit $(G/H)_{hG}$ by definition. The left hand side is equivalent to the $H$-orbit of $EG$. Notice $EG$ is a free $H$-space since it is a free $G$-space. As $EG$ is contractible and the $H$-action is free, the orbit $(EG)_H$ is equivalent to $BH$. This proves $BH\simeq (G/H)_{hG}$.
	\end{proof}
	The lemma implies that $\left((C_{p^{n}})_+\right)_{h\zx{N'}}\simeq \left(B\ker \chi'\right)_+$ and $\left((C_{p^{n-1}})_+\right)_{h\zx{N'}}\simeq \left(B\chi'^{-1}(C_p)\right)_+$. From the short exact sequence of abelian groups:
	\begin{equation*}
		\begin{tikzcd}
		0\rar& \ker \chi'\rar &\chi'^{-1}(C_p)\rar & C_p\rar&0, 
		\end{tikzcd}
	\end{equation*}
	we get a fiber sequence of classifying spaces:
	\begin{equation*}
		\begin{tikzcd}
		B\ker \chi'\rar &B\chi'^{-1}(C_p)\rar & BC_p.
		\end{tikzcd}
	\end{equation*}
	Together with \eqref{eqn:Moore_cofib_ho}, we have shown $M(\Z[\zeta_{p^{n}}])_{h\zx{N'}}\simeq \Sigma^{-1}\left(BC_p\right)_+$ as a spectrum. It now remains to identify $\left(BC_p\right)_+$ in $\Sp_{K(1)}$.
	\begin{lem}\label{lem:BA_k1}
		Let $A$ be a finite abelian group and $A_{(p)}$ be its Sylow $p$-subgroup. Then 
		\begin{equation*}
			\left(BA_+\right)_{K(1)}\simeq \left(S^0_{K(1)}\right)^{\vee |A_{(p)}|}.
		\end{equation*}
	\end{lem}
	\begin{proof}
		By \cite[Corollary 5.10]{HKR_gen_gp_char}, $\left(\Kp\right)_*(BA)\simeq \mathrm{Fun}(A_{(p)},\left(\Kp\right)_*)$, where $\Zpx$ acts on $A_{(p)}$ trivially. The claim now follows from \Cref{lem:Sp_k1_alg_p}.
	\end{proof}
	\begin{cor}
		$(M(\Zp[\zeta_{p^n}])_{h\zx{N'}})_{K(1)}\simeq \left(S^{-1}_{K(1)}\right)^{\vee p}$.
	\end{cor}	
	In \textbf{Case II} when $p>2$ and $n\ge v-1>0$, $\Zp[\chi]=\Zp[\zeta_{p^n}]$.
	From this, we have:
	\begin{align*}
		S^0_{K(1)}(p^v)^{h\chi}&=\map_{\Zp}\left(M(\Zp[\chi]),S^0_{K(1)}(p^v) \right)^{h(\zx{p^v}\x \zx{N'})}\\
		\simeq &\map_{\Zp}\left( M(\Zp[\zeta_{p^{n}}])_{h\zx{N'}},S^0_{K(1)}(p^v) \right)^{h\zx{p^v}}\\
		\simeq & \map_{\Zp}\left(\Sigma^{-1}(BC_p)_+,S^0_{K(1)}(p^v)\right)^{h\zx{p^v}}\\
		\simeq & \map_{K(1)}\left(\left(S_{K(1)}^{-1}\right)^{\vee p},S^0_{K(1)}(p^v)\right)^{h\zx{p^v}}.
	\end{align*}
	The subgroup $\zpx\subseteq \zx{p^v}$ acts on $\left(S_{K(1)}^{-1}\right)^{\vee p}\simeq \left( M(\Zp[\zeta_{p^{n}}])_{h\zx{N'}}\right)_{K(1)}$ by $\chi_p|_{\zpx}$. We claim the other summand $\Z/p^{v-1}\subseteq \zx{p^v}$ acts on the homotopy orbit trivially. This is because both actions of $\zx{N'}$ and $\Z/p^{v-1}$ on the Moore spectrum $M(\Zp[\zeta_{p^n}])$ factors through the action by $C_{p^n}$. As $\zx{N'}$ surjects onto $C_{p^n}$ via $\chi'$, the action of $\Z/p^{v-1}$ on $M(\Zp[\zeta_{p^{n}}])_{h\zx{N'}}$ is trivial,  yielding:
	\begin{align*}
			S^0_{K(1)}(p^v)^{h\chi}\simeq&\map_{K(1)}\left(\left(S_{K(1)}^{-1}\right)^{\vee p},S^0_{K(1)}(p^v)\right)^{h\zx{p^v}}\\
			\simeq& (S^1)^{\vee p}\wedge\map\left(S^0_{\chi|_{\zpx}},\left(S_{K(1)}^0(p^v)\right)^{h\Z/p^{v-1}}\right)^{h\zpx}\\
			\simeq& \Sigma\left(S^0_{K(1)}(p)^{h\chi|_{\zpx}}\right)^{\vee p}.
	\end{align*}	
	In \textbf{Case III} when $p>2$ and $v-1>n$, $\Zp[\chi]=\Zp[\zeta_{p^{v-1}}]$.  By \Cref{prop:MZpv}, there is a $C_{p^{v-1}}$-equivalence:
	\begin{equation*}
	M(\Zp[\zeta_{p^n}])\simeq (C_{p^{v-1}})_+\bigwedge_{C_{p^{n}}}(C_{p^{n}})_+\bigwedge_{C_{p}}M(\Zp[\zeta_p])\simeq (C_{p^{v-1}})_+\bigwedge_{C_{p^{n}}}M(\Zp[\zeta_{p^{n}}]).
	\end{equation*}	
	We have the identification:
	\begin{align*}
		S^0_{K(1)}(p^v)^{h\chi}&=\map_{\Zp}\left(M(\Zp[\chi]),S^0_{K(1)}(p^v) \right)^{h(\zx{p^v}\x \zx{N'})}\\
		\simeq&\left(\map_{\Zp}\left(M(\Zp[\zeta_{p^{v-1}}]),S^0_{K(1)}(p^v) \right)^{h \zx{N'}}\right)^{h\zx{p^v}}\\
		\simeq&\left(\map_{\Zp}\left((C_{p^{v-1}})_+\bigwedge_{C_{p^{n}}}M(\Zp[\zeta_{p^{n}}]),S^0_{K(1)}(p^v) \right)^{h \zx{N'}}\right)^{h\zx{p^v}}\\
		\simeq&\map_{\Zp}\left((C_{p^{v-1}}/C_{p^{n}})_+\bigwedge M(\Zp[\zeta_{p^{n}}])_{h \zx{N'}},S^0_{K(1)}(p^v) \right)^{h\zx{p^v}}\\
		\simeq&\map_{K(1)}\left((C_{p^{v-1}}/C_{p^{n}})_+\bigwedge (S^{-1}_{K(1)})^{\vee p},S^0_{K(1)}(p^v) \right)^{h\zx{p^v}}.
	\end{align*}
	Like in the previous cases, the subgroup $\zpx\subseteq \zx{p^v}$ acts on $\left( M(\Zp[\zeta_{p^{v-1}}])_{h\zx{N'}}\right)_{K(1)}$ by $\chi_p|_{\zpx}$. The other summand $\Z/p^{v-1}\subseteq \zx{p^v}$ acts on the source $(C_{p^{v-1}}/C_{p^{n}})_+\wedge (S^{-1}_{K(1)})^{\vee p}$ via the projection $\Z/p^{v-1}\simeq C_{p^{v-1}}\twoheadrightarrow C_{p^{v-1}}/C_{p^{n}}$, and on the target $S^0_{K(1)}(p^v)$ by the Galois action. As the latter action is free, we have
	\begin{align*}
		S^0_{K(1)}(p^v)^{h\chi}\simeq &\map_{K(1)}\left((C_{p^{v-1}}/C_{p^{n}})_+\bigwedge (S^{-1}_{K(1)})^{\vee p},S^0_{K(1)}(p^v) \right)^{h\zx{p^v}}\\
		\simeq &\Sigma\left(\map_{K(1)}\left(S^0_{\chi|_{\zpx}},S^0_{K(1)}(p^v)^{h\Z/p^n}\right)^{h\zpx}\right)^{\vee p}\\
		\simeq &\Sigma\left(S^0_{K(1)}(p^{v-n})^{h\chi|_{\zpx}}\right)^{\vee p}.
	\end{align*}
	This completes the proof.
	\end{proof}
	In Cases I and II (II') in \Cref{thm:Dk1_ho}, we can now compute  $\pi_*\left(S^0_{K(1)}(p^v)^{h\chi}\right)$ using \eqref{k1spi} and \eqref{eqn:pi_DK1_p} when $p>2$, and \eqref{pis0_2} and \eqref{eqn:Dirichlet_k1_4} when $p=2$, respectively. Computations in Case III (III') are similar. Here we list the results below. 
	\begin{cor}
		When $p>2$, suppose $\chi|_{\zpx}=\omega^{a}$ for some $0\le a\le p-2$. We have:
		\begin{equation*}
		\pi_i\left(S^0_{K(1)}(p^v)^{h\chi}\right)=\left\{\begin{array}{cl}
		\Zp^{\oplus p}, &a=0\textup{ and }i=0 \textup{ or }1;\\
		\left(\Z/p^{v_p(k)+1}\right)^{\oplus p}, &n\ge v-1, i=2k\neq 0\textup{, and }(p-1)\mid (k-a);\\
		\left(\Z/p^{v_p(k)+v-n}\right)^{\oplus p}, &n< v-1, i=2k\neq 0\textup{, and }(p-1)\mid (k-a);\\
		0,&\textup{otherwise.}
		\end{array}\right.
		\end{equation*}
		When $p=2$, if $\chi|_{\zx{4}}$ is trivial, then
		\begin{equation*}
		\pi_i\left(S^0_{K(1)}(2^v)^{h\chi}\right)=\left\{\begin{array}{cl}
		\Z_2^{\oplus 2},& i=0;\\
		(\Z_2\oplus\Z/2)^{\oplus 2},& i=1;\\
		(\Z/2\oplus \Z/2)^{\oplus 2},& i\equiv 2\mod 8;\\
		(\Z/2)^{\oplus 2},&i\equiv 1,3\mod 8\textup{ and }i\neq 1;\\
		\left(\Z/2^{v_2(k)+3}\right)^{\oplus 2}&n\ge v-2\textup{ and }i=4k\neq 0;\\
		\left(\Z/2^{v_2(k)+v-n+1}\right)^{\oplus 2}&n< v-2\textup{ and }i=4k\neq 0;\\
		0,&\textup{otherwise.}
		\end{array}\right.
		\end{equation*}
		If $\chi|_{\zx{4}}=\omega$, then
		\begin{equation*}
		\pi_i\left(S^0_{K(1)}(2^v)^{h\chi}\right)=\left\{\begin{array}{cl}
		(\Z/2)^{\oplus 2},&i\equiv 3,5\mod 8\textup{ and }i\neq 1;\\
		(\Z/2\oplus \Z/2)^{\oplus 2},& i\equiv 4\mod 8;\\		
		\left(\Z/4\right)^{\oplus 2}&n\ge v-2\textup{ and }i\equiv 2\mod 4;\\
		\left(\Z/2^{v-n}\right)^{\oplus 2}&n< v-2\textup{ and }i\equiv 2\mod 4;\\
		0,&\textup{otherwise.}
		\end{array}\right.
		\end{equation*}
	\end{cor}
	\subsection{Homotopy groups of Dirichlet $J$-spectra}
	In this subsection, we first compute homotopy groups of the Dirichlet $J$-spectra by assembling the computations in the previous subsection. From there, we compare these homotopy groups with the special values of the corresponding Dirichlet $L$-functions in \mbox{\Cref{thm:dirichlet_j_gbn}}.
	\begin{thm}\label{thm:pi_dirichlet_j}
		Let $\chi$ be a primitive Dirichlet character $\znx\to \Cx$ of conductor $N$. 
		\begin{enumerate}
			\item When $N=p>2$, if $|\imag \chi|>1$ is not a prime power, then 
			\begin{equation*}
				J(p)^{h\chi}\simeq \bigvee_{\substack{0\le a\le p-2\\\ker \omega^a=\ker\chi}} S_{K(1)}^0(p)^{h\omega^a}\implies \pi_i\left(J(p)^{h\chi}\right)=\left\{\begin{array}{cl}
				\Z/p^{v_p(k)+1}, &i=2k-1\textup{ and }\ker \omega^k=\ker\chi;\\
				0,&\textup{otherwise.}
				\end{array}\right.
			\end{equation*}
			If $|\imag \chi|>1$ is a power of a prime $\ell$, then
			\begin{equation*}
				J(p)^{h\chi}\simeq \left(\Sigma\left(S^0_{KU/\ell}\right)^{\vee \ell}\right)\bigvee\bigvee_{\substack{0\le a\le p-2\\\ker \omega^a=\ker\chi}} S_{KU/p}^0(p)^{h\omega^a} .
			\end{equation*}
			When $\ell>2$, we have 
			\begin{equation*}
			\pi_i\left(J(p)^{h\chi}\right)=\left\{\begin{array}{cl}
			\Z_\ell^{\oplus \ell}, & i=0; \\
			\Z_\ell^{\oplus \ell}, & i=1 \textup{ and }\ker\chi\neq 0;\\
			\Z/p\oplus\Z_\ell^{\oplus \ell}, & i=1 \textup{ and }\ker\chi=0;\\
			\Z/p^{v_p(k)+1}, &i=2k-1\neq 1\textup{ and }\ker \omega^k=\ker\chi;\\
			\left(\Z/\ell^{v_\ell(k)+1}\right)^{\oplus \ell}, &i=2k\neq 0\textup{, and }(\ell-1)\mid k;\\
			0,&\textup{otherwise.}
			\end{array}\right.
			\end{equation*}
			When $\ell=2$ (in particular whenever $p=2^{2^n}+1$ is a Fermat prime), we have
			\begin{equation*}
			\pi_i\left(J(p)^{h\chi}\right)=\left\{\begin{array}{cl}
			\Z_2^{\oplus 2}, &i=0;\\
			\left(\Z_2\oplus \Z/2\right)^{\oplus 2}, & i=1 \textup{ and }\ker\chi\neq 0;\\
			\Z/p\oplus\left(\Z_2\oplus \Z/2\right)^{\oplus 2}, & i=1 \textup{ and }\ker\chi=0;\\
			(\Z/2\oplus \Z/2)^{\oplus 2},& i\equiv 2\mod 8;\\
			\Z/p^{v_p(k)+1}\oplus(\Z/2)^{\oplus 2}, &i=2k-1\neq 1, i\equiv 1,3\mod 8\textup{, and }\ker \omega^k=\ker\chi;\\
			\Z/p^{v_p(k)+1}, &i=2k-1, i\equiv 5,7\mod 8\textup{, and }\ker \omega^k=\ker\chi;\\
			\left(\Z/2^{v_2(k)+3}\right)^{\oplus 2}&i=4k\neq 0;\\
			0,&\textup{otherwise.}
			\end{array}\right.
			\end{equation*}
			\item When $N=p^v$, $v>1$ and $p>2$, we have
			\begin{equation*}
				J(p^v)^{h\chi}\simeq \bigvee_{\substack{0\le a\le p-2\\\ker \omega^a=\ker\chi|_{\zpx}}} S_{K(1)}^0(p^v)^{h\chi_a}\implies\pi_i\left(J(p^v)^{h\chi}\right)=\left\{\begin{array}{cl}
				\Z/p, &i=2k-1\textup{ and }\ker \omega^k=\ker\chi|_{\zpx};\\
				0,&\textup{otherwise},
				\end{array}\right.
			\end{equation*}
			where $\chi_a=\omega^a\cdot (\iota\circ\chi|_{\Z/p^{v-1}})$ and $\iota\colon \Q(\chi)\hookrightarrow \Cp$ is an embedding as in \Cref{exmp:Dirichlet_j_decomp}.
			\item When $N=4$, the only non-trivial character satisfies $\chi(-1)=-1$. We have:
			\begin{equation*}
				J(4)^{h\chi}\simeq S^0_{K(1)}(4)^{h\omega}\implies\pi_i\left(J(4)^{h\chi}\right)=\left\{\begin{array}{cl}
				\Z/4,&i=4k+1;\\
				\Z/2,&i\equiv 2,4\mod 8;\\
				\Z/2\oplus \Z/2,&i\equiv 3\mod 8;\\
				0,&\textup{otherwise.}
				\end{array}\right.
			\end{equation*}
			\item When $N=2^v>4$, $J(4)^{h\chi}\simeq S^0_{K(1)}(2^v)^{h(\iota\circ\chi)}$, where $\iota\colon \Q(\chi)\hookrightarrow \Cbb_2$ is an embedding.
			If $\chi(-1)=1$, then
			\begin{equation*}
			\pi_i\left(J(2^v)^{h\chi}\right)=\left\{\begin{array}{cl}
			\Z/2,&i\equiv 0,2,3,7\mod 8;\\
			\Z/2\oplus \Z/2, &i\equiv 1\mod 8;\\
			0,&\textup{otherwise.}
			\end{array}\right.
			\end{equation*}
			If $\chi(-1)=-1$, then
			\begin{equation*}
			\pi_i\left(J(2^v)^{h\chi}\right)=\left\{\begin{array}{cl}
			\Z/2,&i\equiv 1,2,4,5\mod 8;\\
			\Z/2\oplus \Z/2, &i\equiv 3\mod 8;\\
			0,&\textup{otherwise.}
			\end{array}\right.
			\end{equation*}
			\item  Suppose $N$ has more than one prime factors.
			\begin{enumerate}
				\item $J(N)^{h\chi}$ is contractible unless there is a prime $p$ such that  $|\imag \chi|_{\zx{N'}}|=p^n$ where $N'=N/p^{v_p(N)}$. In particular, $J(N)^{h\chi}$ is contractible whenever $v_\ell(N)\ge 2$ for at least two distinct primes $\ell$.
				\item When there is such a prime $p$, then $J(N)^{h\chi}\simeq \left(J(N)^{h\chi}\right)^\wedge_{p}$. When $p$ is odd, we have
				\begin{equation*}
				\pi_i\left(J(N)^{h\chi}\right)=\left\{\begin{array}{cl}
				\Zp^{\oplus p}, &\chi|_{\zpx}\textup{ is trivial and }i=0 \textup{ or }1;\\
				\left(\Z/p^{v_p(k)+1}\right)^{\oplus p}, &n\ge v-1, i=2k\neq 0\textup{, and }\ker\omega^k=\ker\chi|_{\zpx};\\
				\left(\Z/p^{v_p(k)+v-n}\right)^{\oplus p}, &n< v-1, i=2k\neq 0\textup{, and }\ker\omega^k=\ker\chi|_{\zpx};\\
				0,&\textup{otherwise.}
				\end{array}\right.
				\end{equation*}	
				When $p=2$, if $\chi|_{\zx{4}}$ is trivial, then
				\begin{equation*}
				\pi_i\left(J(N)^{h\chi}\right)=\left\{\begin{array}{cl}
				\Z_2^{\oplus 2},& i=0;\\
				(\Z_2\oplus\Z/2)^{\oplus 2},& i=1;\\
				(\Z/2\oplus \Z/2)^{\oplus 2},& i\equiv 2\mod 8;\\
				(\Z/2)^{\oplus 2},&i\equiv 1,3\mod 8\textup{ and }i\neq 1;\\
				\left(\Z/2^{v_2(k)+3}\right)^{\oplus 2}&n\ge v-2\textup{ and }i=4k\neq 0;\\
				\left(\Z/2^{v_2(k)+v-n+1}\right)^{\oplus 2}&n< v-2\textup{ and }i=4k\neq 0;\\
				0,&\textup{otherwise.}
				\end{array}\right.
				\end{equation*}
				If $\chi|_{\zx{4}}=\omega$, we have
				\begin{equation*}
				\pi_i\left(J(N)^{h\chi}\right)=\left\{\begin{array}{cl}
				(\Z/2)^{\oplus 2},&i\equiv 3,5\mod 8\textup{ and }i\neq 1;\\
				(\Z/2\oplus \Z/2)^{\oplus 2},& i\equiv 4\mod 8;\\		
				\left(\Z/4\right)^{\oplus 2}&n\ge v-2\textup{ and }i\equiv 2\mod 4;\\
				\left(\Z/2^{v-n}\right)^{\oplus 2}&n< v-2\textup{ and }i\equiv 2\mod 4;\\
				0,&\textup{otherwise.}
				\end{array}\right.
				\end{equation*}
			\end{enumerate}
		\end{enumerate}	
	\end{thm}
	\begin{proof}
		By \Cref{prop:JN_hchi_Q_contractible}, $J(N)^{h\chi}_\Q$ is contractible unless $\chi$ is trivial.  By \Cref{prop:jnchi_p_decomposition}, $J(N)^{h\chi}$ decomposes as a wedge sum of Dirichlet $K(1)$-local spheres upon $p$-completions.  Homotopy groups of these summands were computed in the previous subsections. This means we can compute $\pi_*\left(J(N)^{h\chi}\right)$ by assembling homotopy groups of the summands in its $p$-completions at all primes $p$. To illustrate how this works, we now explain the computations in Case (1) above, where the conductor $N$ of the character $\chi$ is an odd prime $p$. Computations in the other cases are accomplished similarly.
		
		By \eqref{eqn:jp_chi_decomp}, the $p$-completion of $J(p)^{h\chi}$ is given by:
		\begin{equation*}
			\left(J(p)^{h\chi}\right)^\wedge_{p}\simeq \bigvee_{\substack{0\le a\le p-2\\\ker \omega^a=\ker\chi}} S^0_{K(1)}(p)^{h\chi_a}.
		\end{equation*}
		As $\chi$ is not the trivial character, $a\ne0$ in this wedge sum. 
		Completed at a different prime $\ell\ne p$, we have by \Cref{prop:jnchi_p_decomposition}:
		\begin{equation}\label{eqn:Jphchi_ell}
			\left(J(p)^{h\chi}\right)^\wedge_{\ell}\simeq \bigvee_{[\sigma]\in\cok \iota^* }\left(S^0_{KU/\ell}\right)^{h(\iota\circ\sigma \circ \chi)}.
		\end{equation}
		When $|\imag \chi|$ is not a power of $\ell$, summands in this decomposition are contractible by \Cref{prop:dirichlet_k1_contractible}. In particular when $|\imag\chi|$ is not a power of any prime , $\left(J(p)^{h\chi}\right)^\wedge_{\ell}$ is contractible for any $\ell\ne p$. So we have  
		\begin{equation*}
			J(p)^{h\chi}\simeq \left(J(p)^{h\chi}\right)^\wedge_{p}\simeq \bigvee_{\substack{0\le a\le p-2\\\ker \omega^a=\ker\chi}} S^0_{K(1)}(p)^{h\chi_a}.
		\end{equation*}
		Homotopy groups of the summands in this decomposition can be then read off from \eqref{eqn:pi_DK1_p}.
		
		When $|\imag \chi|$ is a power of $\ell$, there is only one summand in \eqref{eqn:Jphchi_ell}. This is because $\iota_*$ is an isomorphism under the assumption by \Cref{prop:gal_p_adic_cyclo}.  As a result, the $\ell$-completion of $J(p)^{h\chi}$ is $ \left(S^{0}_{KU/\ell}\right)^{h(\iota\circ\chi)}$. By \Cref{thm:Dk1_ho}, this spectrum is equivalent to $ \Sigma\left(S^0_{KU/\ell}\right)^{\vee \ell}$. Combining the above, we have in this case
		\begin{align*}
			J(p)^{h\chi}\simeq& \left(J(p)^{h\chi}\right)^\wedge_{\ell}\bigvee \left(J(p)^{h\chi}\right)^\wedge_{p}\\
			\simeq &  \left(\Sigma\left(S^0_{KU/\ell}\right)^{\vee \ell}\right)\bigvee\bigvee_{\substack{0\le a\le p-2\\\ker \omega^a=\ker\chi}} S^0_{KU/p}(p)^{h\chi_a}.
		\end{align*}
		Homotopy groups of the summands in this decomposition can then be read off from \eqref{k1spi} (when $\ell>2$) and \eqref{pis0_2} (when $\ell=2$) for the $\ell$-summands, and from \eqref{eqn:pi_DK1_p} for the $p$-summands.
	\end{proof}\color{black}
	\begin{thm}\label{thm:dirichlet_j_gbn}
		Let $\mathcal{D}_{k,\chi}$ be the ideal of $\Z[\chi]$ generated by the denominator of $\frac{B_{k,\chi}}{2k}\in \Q(\chi)$. Set $\mathcal{D}_{k,\chi}=(1)$ when $(-1)^k\neq \chi(-1)$ (i.e. when $B_{k,\chi}=0$). 
		\begin{enumerate}
			\item Assume $N=p>2$ or $N=4$ when $p=2$. For all integers $k$ satisfying $(-1)^k=\chi(-1)$, we have
			\begin{equation*}
			\pi_{2k-1}\left(J(N)^{h\chi}\left[\frac{1}{\ell(\chi)}\right]\right)\simeq \left.\Z\left[\chi\right]\right/\mathcal{D}_{|k|,\chi^{-1}},\quad \textup{where }\ell(\chi)=\left\{\begin{array}{cl}
			\ell,&\textup{if }|\imag(\chi)|\textup{ is a power of a prime }\ell;\\
			1,&\textup{otherwise.}
			\end{array}\right.
			\end{equation*}
			\item When $N=p^v>2p$, $\pi_{2k-1}\left(J(p^v)^{h\chi}\right)\simeq \left.\Z\left[\chi\right]\right/\mathcal{I}_{k,\chi^{-1}}$, where $I_{k,\chi}$ is an ideal of $\Z[\chi]$ such that its multiplicative difference with $\mathcal{D}_{k,\chi}$ contains the principal ideal $(2)$ in $\Z[\chi]$.
		\end{enumerate}		
	\end{thm}
	\begin{rem}
		By \Cref{rem:exotic_pic} and \Cref{rem:j_2v_Moore_model}, the statements above are independent of the models of $M(\Z[\chi])$ when $(-1)^k=\chi(-1)$.
	\end{rem}
	\begin{proof}
		In the first four cases in \Cref{thm:pi_dirichlet_j}, the Dirichlet $J$-spectra are equivalent to their $p$-completions after inverting $\ell(\chi)$ by \Cref{cor:dirichlet_J_Q}, \Cref{prop:jnchi_p_decomposition} and \Cref{cor:dirichlet_k1_contractible}. The only thing remains to check is $\pi_{2k-1}$ where $(-1)^k=\chi(-1)$ and $N=p^v>1$. For that, it suffices to compare the arithmetic properties of $B_{k,\chi}$ in \Cref{thm:GBN} with computations in \Cref{Subsec:pi_twisted_k1}.
		\begin{enumerate}
			\item $N=p>2$.  Comparing the decomposition in \Cref{exmp:Dirichlet_j_decomp} and computation in \eqref{eqn:pi_DK1_p} with \Cref{thm:GBN}, we need to check the following:
			\begin{itemize}
				\item Let $g$ be a primitive $(p-1)$-st root of unity mod $p$. The ideal $\mathfrak{p}=(p,1-\chi(g)g^k)$ of $\Z[\chi]$ is not equal to $(1)$ iff $\ker \chi=\ker\omega^{-k}$. To see this, notice by \Cref{cor:cyclo_rep_p_decomp}, there is an isomorphism of $\zpx$-representations:
				\begin{equation*}
				\Z[\chi]/p\simeq \bigoplus_{\substack{0\le a\le p-2\\\ker \omega^a=\ker\chi}} (\Z/p)_{\omega^a}\simeq \bigoplus_{\substack{0\le a\le p-2\\\ker \omega^a=\ker\chi}} (\Z/p)^{\otimes a}.
				\end{equation*} 
				Then $1-\chi(g)g^k$ is invertible in $\Z[\chi]/p$ iff $1\equiv g^a\cdot g^k\mod p$ for some $a$ satisfying $0\le a\le p-2$ and $\ker \chi=\ker \omega^a$. Since $g$ is a primitive $(p-1)$-st root of unity mod $p$, this condition is further equivalent to saying $(p-1)\mid (a+k)$ for such an $a$. From this we conclude $\ker\chi=\ker \omega^{-k}$.
				\item When $\mathfrak{p}\neq (1)$, the congruence \eqref{eqn:bkchi_p} $pB_{k,\chi}\equiv p-1 \mod \mathfrak{p}^{v_p(k)+1}$ implies $\Z[\chi]/\mathcal{D}_{k,\chi}\simeq \Z/p^{v_p(k)+1}$. 
				
				It suffices to check this formula holds $p$-adically and $2$-adically since the denominator ideal of $\frac{B_{k,\chi}}{k}$ is $p$-primary by \Cref{thm:GBN}. As $2\mid (p-1)$, $\mathcal{D}_{k,\chi}$ has no $2$-primary factors by \eqref{eqn:bkchi_p}. Completed at $p$, the ideal $\mathfrak{p}$ is the same as $(p)$ when it is not $(1)$. Now \eqref{eqn:bkchi_p} becomes 
				\begin{equation*}
					pB_{k,\omega^{a}}\equiv p-1\mod p^{v_p(k)+1}\implies \frac{B_{k,\omega^a}}{2k}-\frac{p-1}{2pk}\in \Z_{(p)}[\omega^a],
				\end{equation*}
				where $a$ satisfies $\ker \omega^a=\ker \chi $ and $(p-1)\mid (k+a)$. This implies \begin{equation*}\Z[\chi]/\mathcal{D}_{k,\chi^{-1}}\simeq \Z/p^{v_p(k)+1}\simeq \pi_{2k-1}\left(J(p)^{h\chi}\left[\frac{1}{\ell(\chi)}\right]\right).\end{equation*}
			\end{itemize} 
			\item $N=p^v$, $v>1$ and $p>2$. By \Cref{lem:gt_chi_g_odd}, $\mathfrak{p}=(p,1-\chi(g)g^k)\neq (1)$ when $\ker\chi|_{\zpx}=\ker \omega^{-k}$. In that case, $\mathfrak{p}=(1-\zeta_{p^{v-1}},p)=(1-\zeta_{p^{v-1}})$. On the other hand, since $1+p$ is a generator of the subgroup $\Z/p^{v-1}\subseteq \zx{p^v}$ and $\chi$ is primitive, $\chi(1+p)$ is also a primitive $p^{v-1}$-th root of unity. As a result, \eqref{eqn:bkchi_p} translates into
			\begin{equation*}
				(1-\chi(p+1))\frac{B_{k,\chi}}{k}\equiv 1\mod \mathfrak{p}\implies \frac{B_{k,\chi}}{k}\equiv \frac{1}{1-\zeta_{p^{v-1}}}\mod \Zp[\zeta_{p^{v-1}}].
			\end{equation*}
			Thus $\mathcal{D}_{k,\chi}$ is either $(1-\zeta_{p^{v-1}})$ or $(2(1-\zeta_{p^{v-1}}))$. Whereas by \Cref{thm:pi_dirichlet_j}, $\pi_{2k-1}\left(J(p^v)^{h\chi}\right)\simeq \Z/p\simeq \Z[\chi]/(1-\zeta^{p_{v-1}})$.
			\item $N=4$. In this case $\chi=\chi^{-1}$ since $\zx{4}\simeq C_2$. By \eqref{eqn:bkchi_4}, we have when $k$ is odd:
			\begin{equation*}
			\frac{B_{k,\chi}}{k}-\frac{1}{2}\in \Z[\chi]=\Z \implies \frac{B_{k,\chi}}{2k}-\pm\frac{1}{4}\in \Z[\chi]=\Z.
			\end{equation*}
			Thus $\mathcal{D}_{k,\chi}=\mathcal{D}_{k,\chi^{-1}}$ is equal to the ideal $(4)$ of $\Z[\chi]\simeq \Z$. This matches the computation in \eqref{eqn:Dirichlet_k1_4} that $\pi_{2k-1}\left(S_{K(1)}^0(4)^{h\omega}\right)\simeq \Z/4$ when $k$ is odd.
			\item $N=2^v>4$. \Cref{thm:GBN} says $\frac{B_{k,\chi}}{k}$ is an algebraic integer. As a result, $\mathcal{D}_{k,\chi}$ the denominator ideal of $\frac{B_{k,\chi}}{2k}$ contains $(2)$ as a sub-ideal. By \Cref{thm:pi_dirichlet_j}, $\pi_{2k-1}\left(J(2^v)^{h\chi}\right)\simeq \Z/2\simeq \Z[\chi]/(1-\zeta_{2^{v-2}})$. As both $\mathcal{D}_{k,\chi}$ and $\mathcal{I}_{k,\chi}$ contain the ideal $(2)$ in $\Z[\chi]$, their  multiplicative difference also contains the ideal $(2)$.
		\end{enumerate} 
	\end{proof}
	\section{Comparisons of $J$-spectra and $L$-functions}
	In this section, we compare various $J$-spectra we constructed with $L$-functions. We first relate the spectrum $J(N)$ and Dedekind $\zeta$-functions in \Cref{subsec:jn_Dedekind}. In addition, we find the Brown-Comenetz duals of the Dirichlet $J$-spectra and $K(1)$-spheres, as well as $J(N)$ in \Cref{subsec:BC_dual}. This duality phenomenon is similar to the functional equations of the corresponding $L$-functions. 
	\subsection{$J$-spectra, Dedekind $\zeta$-functions, and algebraic $K$-theory}\label{subsec:jn_Dedekind}
	In this subsection, we compare the spectrum $J(N)$ with Dedekind $\zeta$-function of the field $\Q(\zeta_N)$. We will focus on the case when $N=p^v$ for some prime $p$. By \Cref{thm:Dedekind_zeta_Dirichlet_L}, the Dedekind $\zeta$-function of $\Q(\zeta_N)$ and Dirichlet $L$-functions are related by:
	\begin{equation}\label{eqn:Dedekind_zeta_prod}
	\zeta_{\Q(\zeta_N)}(s)=\prod_{\chi\colon \znx\to\Cx} L(s,\chi).
	\end{equation}
	When $N=p^v$, this yields:
	\begin{equation*}
	\frac{\zeta_{\Q(\zeta_{p^v})}(s)}{\zeta_{\Q(\zeta_{p^{v-1}})}(s)}= \prod_{\substack{\chi\colon \zx{p^v}\to\Cx\\\text{primitive}}} L(s,\chi).
	\end{equation*}
	The formula above reminds us of \Cref{prop:S_k1_pv_cofib} and \Cref{prop:S_k1_2v_cofib}.
	\begin{prop}\label{prop:J_pv_cofib}
		Let $p>2$ be a prime and $\chi\colon \zx{p^v}\to \Cpx$ be any primitive $p$-adic Dirichlet character of conductor $p^v$. 
		\begin{equation*}
		\mathrm{Cofib}(J(p^{v-1})\to J(p^v))\simeq\left\{\begin{array}{cl}
		\bigvee_{a=1}^{p-2} S_{K(1)}^{0}(p)^{h\omega^a},& v=1;\\
		\bigvee_{a=0}^{p-2} S_{K(1)}^{0}(p^v)^{h\chi_a},&v>1,
		\end{array}\right.
		\end{equation*}
		where $\chi_a=\omega^a\cdot (\chi|_{\Z/p^{v-1}})$.
	\end{prop}
	\begin{proof}
		By \Cref{cor:jn_structure}, $\mathrm{Cofib}(J(p^{v-1})\to J(p^v))$ is equivalent to its $p$-completion, since the map is an equivalence rationally, and when completed at a prime other than $p$. At prime $p$, we have
		\begin{equation*}
		\mathrm{Cofib}(J(p^{v-1})\to J(p^v))\simeq \mathrm{Cofib}(J(p^{v-1})\to J(p^v))^\wedge_p\simeq \mathrm{Cofib}\left(S_{K(1)}^0(p^{v-1})\to S_{K(1)}^0(p^v)\right).
		\end{equation*}
		When $v=1$, taking this cofiber removes the $a=0$ summand in the Adams splitting of $S_{K(1)}^0(p)$. When $v>1$, the claim follows from \Cref{prop:comparsion_d_k1_pv} and the Adams splittings.
	\end{proof}
	\begin{cor}\label{cor:J_pv_cofib}
		Notations as above. When $p>2$ and $v>1$, we have
		\begin{equation*}
		\mathrm{Cofib}(J(p^{v-1})\to J(p^v))\simeq \bigvee_{a\in [0,p-2]/\sim} J(p^v)^{h\chi_a},
		\end{equation*}
		where $a\sim b$ if $\ker \omega^a=\ker \omega^b$.
	\end{cor}
	\begin{proof}
		This follows from \Cref{prop:J_pv_cofib} and Case (2) in \Cref{thm:pi_dirichlet_j}.
	\end{proof}
	One might now wonder if there is a connection between special values of $\zeta_{\Q(\zeta_N)}$ and homotopy groups of $J(N)$ as in \Cref{thm:dirichlet_j_gbn}. Notice in \eqref{eqn:Dedekind_zeta_prod}, both even and odd characters show up on the right hand side. Also recall $L(1-k,\chi)=0$ unless $(-1)^k=\chi(-1)$. This means $\zeta_{\Q(\zeta_N)}(1-k)=0$ for all positive integers $k$. As a result, a direct analogy of \Cref{thm:dirichlet_j_gbn} does not exist in this case.
	
	There are two ways one might try to fix this. The first one is to exclude odd characters in the product formula \eqref{eqn:Dedekind_zeta_prod}. Let $\Kbb$ be a \textit{totally real} finite abelian extension  of $\Q$, and $N$ be the smallest integer such that $\Kbb\subseteq \Q(\zeta_N)$. The number field $\Kbb$ being totally real is equivalent to $\gal(\Q(\zeta_N)/\Kbb)$ containing complex conjugation, which is identified with $-1\in\znx$ via \Cref{lem:gal_Q_zeta}. This means in the product formula \Cref{thm:Dedekind_zeta_Dirichlet_L}
	\begin{equation*}
	\zeta_{\Kbb}(s)=\prod_{\substack{\chi\colon \znx\to\Cx\\ \gal(\Q(\zeta_N)/\Kbb)\subseteq\ker \chi}} L(s,\chi),
	\end{equation*}
	only even characters show up on the right hand side and $\zeta_{\Kbb}$ has non-zero special values when $s=1-2k$.
	\begin{thm}\label{thm:J_K_Dedekind}
		Let $\Kbb/\Q$ be a totally real finite abelian extension. Suppose the smallest integer $N$ such that $\Kbb\subseteq \Q(\zeta_{N})$ is a prime power $p^v$. Denote the Galois group $\gal(\Q(\zeta_{p^v})/\Kbb)$ by $G$. Then 
		\begin{equation*}
		\pi_{4t-1}\left( J(p^v)^{hG}\left[\frac{1}{|G|}\right]\right)=\left.\Z\left[\frac{1}{|G|}\right]\right/D_{\Kbb,2t},
		\end{equation*}
		where $D_{\Kbb,2t}\in \Z_{>0}$ is the denominator of $\zeta_{\Kbb}(1-2t)$. 
	\end{thm}
	\begin{proof}
		It suffices to compare the two sides at primes not dividing $|G|$. We first show $G\subseteq \zpx\subseteq \zx{p^v}\simeq \gal(\Q(\zeta_{p^v})/\Q)$. This is true when $v=1$ since $\gal(\Q(\zeta_{p})/\Q)\simeq\zpx$. Now assume $v>1$ and suppose $G$ contains an element of order $p$. Then we have $\gal(\Q(\zeta_{p^v})/\Q(\zeta_{p^{v-1}}))\simeq C_p\subseteq G$. By Galois correspondence between subfields of $\Q(\zeta_{p^v})$ and subgroups of $\zx{p^v}$, this would imply $\Kbb\subseteq \Q(\zeta_{p^{v-1}})$, contradicting our assumption.
		
		It follows that $p\nmid |G|$. Let $\ell$ be a prime such that $\ell\nmid |G|$ and $\ell\neq p$. By \Cref{cor:jn_structure}, we have
		\begin{equation*}
		\left(J(p^v)^{hG}\right)^\wedge_\ell\simeq S^0_{KU/\ell}
		\end{equation*}
		This shows $\pi_{4t-1}\left((J(p^v)^{hG})^\wedge_\ell\right)=\pi_{4t-1}\left(S^0_{KU/\ell}\right)=\Z_\ell/D_{2t}$. Using the product formula \Cref{thm:Dedekind_zeta_Dirichlet_L} and Carlitz's \Cref{thm:GBN}, we get $\pi_{4t-1}\left((J(p^v)^{hG})^\wedge_\ell\right)\simeq \Z_{\ell}/D_{2t}=\Z_{\ell}/D_{\Kbb,2t}$.
		
		Completed at the prime $p$, we have by \Cref{cor:jn_structure} and \Cref{thm:pi_dirichlet_j}
		\begin{equation*}
		\left(J(p^v)^{hG}\right)^\wedge_p\simeq S^0_{K(1)}(p^v)^{hG}\simeq \bigvee_{\substack{0\le a\le p-2
				\\G\subseteq\ker\omega^a}} S^0_{K(1)}(p^v)^{h\omega^a}\simeq \bigvee_{\substack{\chi\in\hom(\zpx,\Cx )/\sim
				\\G\subseteq\ker\chi}}\left(J(p^v)^{h\chi}\right)^\wedge_p
		\end{equation*}
		where $\chi_1\sim \chi_2$ iff they have the same images (thus differ by an element in the Galois group). Now we can prove the claim by comparing \Cref{thm:Dedekind_zeta_Dirichlet_L} and \Cref{prop:S_k1_pv_cofib} via \Cref{thm:dirichlet_j_gbn}. 
	\end{proof}
	A second approach to relate $J(N)$ to $\zeta_{\Q(\zeta_N)}$ is to consider the Taylor expansion of the Dedekind $\zeta$-functions at zero special values. 
	\begin{defn}
		For any number field $\Kbb$, we denote by $\zeta_{\Kbb}^*(1-k)$ the coefficient of the first non-zero term in the Taylor expansion of $\zeta_{\Kbb}(s)$ at $s=1-k$.
	\end{defn}
	From the definition, $\zeta_{\Kbb}^*(1-k)=\zeta_{\Kbb}(1-k)$ when the latter is not zero. This special value $\zeta_{\Kbb}^*(1-k)$ is closely related to the \textbf{algebraic $K$-theory} of $\mathcal{O}_\Kbb$, the ring of integers of $\Kbb$.
	\begin{thm}[Quillen-Lichtenbaum Conjecture, Voevodsky-Rost, {\cite[199 -- 200]{Kolster_k_thy_arithmetic}}]
		For all $k\ge 2$:
		\begin{equation*}
			\zeta^*_{\Kbb}(1-k)=\pm\frac{|K_{2k-2}(\mathcal{O}_\Kbb)|}{|K_{2k-1}(\mathcal{O}_\Kbb)_{\mathrm{tors}}|}\cdot R^B_k(\Kbb),
		\end{equation*}
		up to powers of $2$, where $R^B_k(\Kbb)$ is the $k$-th Borel regulator of $\Kbb$. 
	\end{thm}
	Let $\Kbb/\Q$ be a finite abelian extension. When $\Kbb$ is as in \Cref{thm:J_K_Dedekind}, we have seen both $\pi_{2k-1}(J(p^v)^{hG})$ and $K_{2k-1}(\mathcal{O}_{\Kbb})$ capture the denominator of $\zeta^*_{\Kbb}(1-k)$. We now want to compare the $J$-spectra and algebraic $K$-theory spectra of ring of integers of number fields. Recall in \Cref{prop:JN_k_loc}, we showed $J(N)$ is a $KU$-local $\einf$-ring spectrum, with a natural $\znx$-action by $\einf$-maps. While $K(\mathcal{O}_{\Kbb})$ is not a $KU$-local spectrum (since it is connective and is not a wedge sum of Eilenberg-MacLane spectra), it is very close to one in the following sense:
	\begin{thm}[Waldhausen, {\cite[Conjecture 11.5]{Mitchell_LQC}}]\label{thm:Waldhausen}
		Fix a prime $\ell$. Assuming Lichtenbaum-Quillen Conjecture holds at $\ell$, then the $E(1)$-localization map $K(\mathcal{O}_{\Kbb})\to L_1K(\mathcal{O}_{\Kbb})$ is an $\ell$-local isomorphism on $\pi_n$ for all $n\ge 1$.
	\end{thm}
	This shows the $KU$-localization map $K(\mathcal{O}_{\Kbb})\to L_{KU}K(\mathcal{O}_{\Kbb})$ is an equivalence on $1$-connective covers of the two spectra.The spectrum $K(\mathcal{O}_{\Kbb})$ is an $\einf$-ring spectrum by a result of May \cite[\!2.8]{Mitchell_LQC}. Since Quillen's construction of algebraic $K$-theory spectra is functorial, there is a natural $\gal(\Kbb/\Q)$-action on $K(\mathcal{O}_{\Kbb})$ by $\einf$-maps. So just like $J(N)$, $L_{KU}K(\Z[\zeta_N])$ is a $KU$-local $\einf$-ring spectrum, with a natural $\znx$-action by $\einf$-maps. When $N=1$, the two spectra $J(1)=J$ and $L_{KU}(K(\Z))$ are connected by the $KU$-local Hurewicz map:
	\begin{equation*}
	h_{KU}\colon  J=S^0_{KU}\longrightarrow L_{KU}K(\Z).
	\end{equation*}
	A natural question to ask here is if we can extend $h_{KU}$ to a natural $\znx$-equivariant map of $KU$-local $\einf$-ring spectra $h(N)\colon  J(N)\to L_{KU}K(\Z[\zeta_{N}])$. This is essentially answered by the recent work of Bhatt-Clausen-Mathew in \cite{BCM_rmk_k1_alg_k}. Recall in \Cref{prop:otop_section}, there is a sheaf of $K(1)$-local $\einf$-ring spectra $\otop_{K(1)}$ over the classifying stack $B\Zpx$. We now identify this stack with the $p$-cyclometic tower of $\Q$ as follows. By \Cref{lem:gal_Q_zeta}, we have
	 \begin{itemize}
	 	\item the Galois group $\gal(\Q(\zeta_{p^v})/\Q)$ is isomorphic to $(\Z/p^v)^\x$,
	 	\item the ring of integers of $\Q(\zeta_{p^v})$ is $\Z[\zeta_{p^v}]$.
	 \end{itemize} 
	In addition, one can check that the inclusion map $\Z\to \Z[\zeta_{p^v}]$ becomes an \'etale extension after inverting $p$. Combining the three facts above, we have a Galois correspondence between open subgroups of $\Zpx$ and $p$-cyclotometic \'etale extensions of $\Z[1/p]$. By \cite[Theorem 1.8, Theorem 1.10]{CMNN_descent}, $K(1)$-local algebraic $K$-theory satisfies \'etale descent. This means $L_{K(1)}K(-)$ defines a sheaf of $K(1)$-local $\einf$-ring spectra over $B\Zpx$, where
	\begin{equation*}
	 	\Gamma(L_{K(1)}K(-), B\Zpx)=L_{K(1)}K(\Z[1/p]),\qquad\Gamma(L_{K(1)}K(-), B(1+p^v\Zp))=L_{K(1)}K(\Z[1/p,\zeta_{p^v}]).
	\end{equation*}
	\begin{thm}[Bhatt-Clausen-Mathew, \mbox{\cite[Theorem 3.9]{BCM_rmk_k1_alg_k}}]\label{thm:BCM}
		There is a map of sheaves of $K(1)$-local $\einf$-ring spectra $\underline{h}\colon  \otop_{K(1)}\to L_{K(1)}K(-)$ over $B\Zpx$, such that its evaluation on global sections is the $K(1)$-local Hurewicz map $h_{K(1)}\colon  S^0_{K(1)}\to L_{K(1)}K(\Z[1/p])$.
	\end{thm} 
	\begin{prop}\label{cor:J_K_comparison}
	There is a $\znx$-equivariant map of $KU$-local $\einf$-ring spectra $h(N)\colon  J(N)\to L_{KU}K(\Z[\zeta_{N}])$ extending the $KU$-local Hurewicz map. Moreover, the maps $h(N_1)$ and $h(N_2)$ are compatible when $N_1$ divides $N_2$:
	\begin{equation*}
	\begin{tikzcd}
	J\dar["h_{KU}"]\rar& J(N_1)\dar["h(N_1)"]\rar&J(N_2)\dar["h(N_2)"]\\
	L_{KU}K(\Z)\rar&L_{KU}K(\Z[\zeta_{N_1}])\rar& L_{KU}K(\Z[\zeta_{N_2}])
	\end{tikzcd}
	\end{equation*}
	\end{prop} 
	\begin{proof}
		We construct $h(N)$ as a map between arithmetic fracture squares:
		\begin{equation}\label{eqn:pb_cube}
			\begin{tikzcd}[row sep=scriptsize,column sep=0 ex]
				J(N)\ar[rr]\ar[dd,dashed,"h(N)"]\ar[dr]\arrow[drrr, phantom, "\revangle", very near start]&&\prod_{p}S^0_{KU/p}\left(p^{v_p(N)}\right)\ar[dr,"L_\Q"]\ar[dd,"\prod h(N)_p",near end]&\\
				&S^0_\Q\ar[rr,crossing over]&&\left(\prod_{p}S^0_{KU/p}\left(p^{v_p(N)}\right)\right)_\Q\ar[dd,"(\prod h(N)_p)_\Q"]\\
				L_{KU}K(\Z[\zeta_N])\ar[rr]\ar[dr]\arrow[drrr, phantom, "\revangle", very near start]&&\prod_{p}L_{KU/p}K(\Z[\zeta_N])\ar[dr,"L_\Q"]&\\
				&L_{\Q}K(\Z[\zeta_N])\ar[from=uu,crossing over,"h_\Q", near end]\ar[rr]&&\left(\prod_{p}L_{KU/p}K(\Z[\zeta_N])\right)_\Q
			\end{tikzcd}
		\end{equation}
		Here, the top fracture square is from \eqref{eqn:jn} and bottom one is the arithmetic fracture square for $KU$-localization in \cite[Proposition 2.9]{Bousfield_localization}. The vertical map $S^0_\Q\to L_{\Q}K(\Z[\zeta_N])$ is the rational Hurewicz map $h_\Q$. Fix a prime $p$ and write $N=p^vN'$ where $v=v_p(N)$. We will now construct a $\znx$-equivariant map $h(N)_p\colon  S^0_{KU/p}(p^v)\to L_{KU/p}K(\Z[\zeta_N])$ of $KU/p$-local $\einf$-ring spectra in three steps. 
		\begin{enumerate}
			\item Evaluating the map $\underline{h}\colon \otop_{KU/p}\to L_{KU/p}K(-)$ in \Cref{thm:BCM} at the finite \'etale cover $B(1+p^v\Zp)\to B\Zpx$, we get a $(\Z/p^v)^\x$-equivariant map of $KU/p$-local $\einf$-ring spectra \[\underline{h}(B(1+p^v\Zp))\colon  S^0_{KU/p}(p^v)\to L_{KU/p}K(\Z[1/p,\zeta_{p^v}]).\] Notice $\znx\simeq (\Z/p^v)^\x\x(\Z/N')^\x$, the map $\underline{h}(B(1+p^v\Zp))$ then can be viewed a $\znx$-equivariant map where the summand $(\Z/N')^\x$ acts trivially on both the source and the target.
			\item The extension of rings 
			\begin{equation*}
				\Z[1/p,\zeta_{p^v}]\longrightarrow \Z[1/p,\zeta_{p^v}]\otimes_{\Z[1/p]}\Z[1/p,\zeta_{N'}]\simeq \Z[1/p,\zeta_N]
			\end{equation*}
			 is $\znx$-equvariant where the summand $\zx{N'}$ acts trivially on the source. This induces a $\znx$-equvariant map of $\einf$-rings spectra on their algebraic $K$-theory spectra:
			\begin{equation*}
				K(\Z[1/p,\zeta_{p^v}])\to K(\Z[1/p,\zeta_{p^v}])\wedge_{K(\Z[1/p])} K(\Z[1/p,\zeta_{N'}])\to K(\Z[1/p,\zeta_{N}])
			\end{equation*}
			\item Finally by \cite[Theorem 1.1]{BCM_rmk_k1_alg_k}, the natural map $L_{KU/p}K(\Z[\zeta_{N}])\to L_{KU/p}K(\Z[1/p,\zeta_{N}])$ is an equivalence of $\einf$-ring spectra. Again, this map is $\znx$-equivariant since it is induced by a $\znx$-equivariant map of rings $\Z[\zeta_N]\to \Z[1/p,\zeta_N]$.
		\end{enumerate}  Combining the maps above, we define $h(N)_p$ to be the composite
		\begin{equation}\label{eqn:h(N)p}
			\begin{tikzcd}[row sep=small]
				S_{KU/p}^0(p^v)\ar[rr,"{\underline{h}(B(1+p^v\Zp))}"]\ar[drrr,dashed,"h(N)_p"',end anchor=west]&& L_{KU/p}K(\Z[1/p,\zeta_{p^v}])\rar&L_{KU/p}K(\Z[1/p,\zeta_{N}])\\&&& L_{KU/p}K(\Z[\zeta_N])\uar["\sim"{sloped,above}].
			\end{tikzcd}
		\end{equation}
		Having constructed $h(N)_p$, we claim the rationalization of their product $\left(\prod h(N)_p\right)_\Q$ makes the diagrams on the right and front faces on the cube \eqref{eqn:pb_cube} commute. The right face commutes since $L_\Q$ is a natural transformation. The front face commutes since all four edges are maps of rational $\einf$-ring spectra and $S^0_\Q$ is initial among such spectra. Consequently, we get $h(N)\colon  J(N)\to L_{KU}K(\Z[\zeta_N])$ as a map between pullbacks. As all three maps $h_\Q, \left(\prod h(N)_p\right)_\Q,$ and $\prod h(N)_p$ are $\znx$-equivariant maps of $\einf$-ring spectra, the map $h(N)$ is also a $\znx$-equivariant map of $\einf$-ring spectra.
		
		When $N=1$, $h(1)_p$ is the $KU/p$-local Hurewicz map at all primes. This implies $h(1)\colon  J=S^0_{KU}\to L_{KU}K(\Z)$ is the $KU$-local Hurewicz map. The compatibility between $h(N_1)$ and $h(N_2)$ when $N_1\mid N_2$ follows from the compatibility of their $p$-completions for all $p$. To show this, we use the fact that the map $\underline{h}$ in \Cref{thm:BCM} is a map of sheaves of spectra and that algebraic $K$-theory is functorial. More precisely, let $v_1$ and $v_2$ be the $p$-adic valuations of $N_1$ and $N_2$, respectively. The assumption $N_1\mid N_2$ implies $v_1\le v_2$. We have a commutative diagram:
		\begin{equation*}
			\begin{tikzcd}
				S^0_{KU/p}(p^{v_1})\rar\dar &L_{KU/p}(K(\Z[1/p,\zeta_{p^{v_1}}]))\rar\dar&L_{KU/p}(K(\Z[1/p,\zeta_{N_1}]))\dar&L_{KU/p}(K(\Z[\zeta_{N_1}]))\lar["\sim"']\dar\\
				S^0_{KU/p}(p^{v_2})\rar &L_{KU/p}(K(\Z[1/p,\zeta_{p^{v_2}}]))\rar&L_{KU/p}(K(\Z[1/p,\zeta_{N_2}]))&L_{KU/p}(K(\Z[\zeta_{N_2}]))\lar["\sim"']
			\end{tikzcd}
		\end{equation*}
 	\end{proof}
 	Recall from \Cref{lem:Sp_k1_alg_p} that when $p>2$, a $K(1)$-local spectrum $X$ is determined (up to weak equivalence) by its $\Kp$-homology groups together the Adams operations ($\Zpx$-actions) on it. When $p=2$, the statement is true if we replace $KU^\wedge_{2}$ by $KO^\wedge_{2}$ (\Cref{lem:Sp_k1_alg_2}). We now conclude this subsection by studying the induced map of $h(N)$ on $\Kp$-homology groups. In even degrees, we need the following computations:
 	\begin{thm}[Dwyer-Mitchell, {\cite[Theorem 1.7]{Dwyer-Mitchell_alg_int}}]\label{thm:KpKalg}
 			Let $\Kbb$ be a number field. Denote by $\Kbb(\mu_{p^\infty})$ the cyclotomic extension of $\Kbb$ obtained by adjoining all $p$-power roots of unity and let $G_p(\Kbb)$ be the Galois group $\gal(\Kbb(\mu_{p^\infty})/\Kbb)$. The $\Kp$-homology groups of the algebraic $K$-theory spectrum of the ring of $p$-integers $\Ocal_\Kbb[1/p]$ in $\Kbb$ are given by:
 		\begin{equation*}
 			\left(\Kp\right)_{2t}(K(\Ocal_\Kbb[1/p]))\simeq \map(\Zpx/G_p(\Kbb), (\Kp)_{2t}).
 		\end{equation*}
 	\end{thm}
 	\begin{prop}\label{prop:hN_even}
 		The map $h(N)\colon  J(N)\to L_{KU}K(\Z[\zeta_N])$ induces isomorphisms of $\Zpx$-modules on even-degree $\Kp$-homology groups for all primes $p$.
 	\end{prop}
 	\begin{proof}
 		Notice for any spectrum $X$, $(\Kp)_{2t}(X)\simeq (\Kp)_0(X)\otimes(\Kp)_{2t}$ is an isomorphism of  $\Zpx$-modules. In addition, $\Kp$-homology groups of a spectrum only depends on its $KU/p$-localization. So from the construction of $h(N)$, it suffices to check $h(N)_p$ induces isomorphisms of $\Zpx$-modules in the zeroth $\Kp$-homology group for each prime $p$. In fact, we will show that this is true for each step in the construction of $h(N)_p$ in \eqref{eqn:h(N)p}. 
 		\begin{enumerate}
 			\item Recall from \Cref{defn:s_k1_pv}, $S^0_{KU/p}(p^v)=(\Kp)^{h(1+p^v\Zp)}$. So its zeroth $\Kp$-homology is
 			\begin{equation*}
 				\left(\Kp\right)_0(S^0_{KU/p}(p^v))\simeq \map(\Zpx/(1+p^v\Zp),\Zp),
 			\end{equation*}
 			where $\Zpx$ acts trivially on $\Zp$. When $\Kbb=\Q(\zeta_{p^v})$, $\Kbb(\mu_{p^\infty})=\Q(\zeta_{p^\infty})$ and $G_p(\Q(\zeta_{p^v}))=1+p^v\Zp$. By \Cref{thm:KpKalg}, we have 
 			\begin{equation*}
 				\left(\Kp\right)_0(K(\Z[1/p,\zeta_{p^v}]))\simeq \map(\Zpx/(1+p^v\Zp),\Zp).
 			\end{equation*}
 			The above computation shows two $\Kp$-homology groups are (abstractly) isomorphic as $\Zpx$-modules. It remains to check that the induced map of $\underline{h}(B(1+p^v\Zp))$ realizes this isomorphism. Notice $\underline{h}(B(1+p^v\Zp))$ is an $\zx{p^v}$-equivariant map of ring spectra, its induced map on the zeroth $\Kp$-homology groups is also a $\zx{p^v}$-map of rings. Being a map of algebras, it must preserve idempotents in $\map(\Zpx/(1+p^v\Zp),\Zp)$ and send $1$ to $1$. Notice idempotents in this ring are characteristic functions of subsets of $\Zpx/(1+p^v\Zp)$ and the identity element is the sum of all characteristic functions $\chi_g$ of the element $[g]\in \Zpx/(1+p^v\Zp)$. This means the induced map of $\underline{h}(B(1+p^v\Zp)$ sends a characteristic function of an element to another one. From the $\zx{p^v}$-equivariance of the induced map, we conclude it is given by sending a function $F\in \map(\Zpx/(1+p^v\Zp),\Zp)$ to $F(g\cdot -)$ for some $g\in \Zpx$, which is an isomorphism of $\Zpx$-modules. 
 			\item 	
 			 When $\Kbb=\Q(\zeta_{N})$, we have $\Kbb(\mu_{p^\infty})=\Q(\zeta_{p^\infty},\zeta_{N'})$, where $N=p^vN'$ and $(N',p)=1$. This implies $G_p(\Q(\zeta_{N}))=\gal(\Q(\zeta_{p^\infty},\zeta_{N'})/\Q(\zeta_{p^v},\zeta_{N'}))=1+p^v\Zp$. As a result, the field extension $\Q(\zeta_{p^v})\to \Q(\zeta_N)$ induces an isomorphism of Galois groups $1+p^v\Zp=G_p(\Q(\zeta_{p^v}))\simto G_p(\Q(\zeta_{N}))$. In light of \Cref{thm:KpKalg}, we have shown the extension of rings $\Z[1/p,\zeta_{p^v}]\to \Z[1/p,\zeta_N]$ induces an isomorphism of $\Zpx$-modules:
 			\begin{equation*}
 				\left(\Kp\right)_{0}(K(\Z[1/p,\zeta_{p^v}]))\simto \left(\Kp\right)_{0}(K(\Z[1/p,\zeta_N])).
 			\end{equation*}
 			\item Finally, since the map $L_{KU/p}K(\Z[\zeta_N])\to L_{KU/p}K(\Z[1/p,\zeta_N])$ is an equivalence, its induced map on $\Kp$-homology groups must also be an isomorphism of $\Zpx$-modules. 
 		\end{enumerate}	
 	\end{proof}
	The $\Kp$-homology groups of $J(N)$ are zero in odd degrees. So $h(N)\colon  J(N)\to L_{KU}K(\Z[\zeta_N])$ induces zero maps on odd-degree $\Kp$-homology groups. Recall that a $KU$-local spectrum can be assembled from its rationalization and $p$-completions at all primes. By \Cref{lem:Sp_k1_alg_p} and \Cref{lem:Sp_k1_alg_2}, the latter is in turn determined by the Adams operations on its $\Kp$-homology groups ($KO^\wedge_{2}$-homology groups when $p=2$). In this sense, the map $h(N)$ exhibits $J(N)$ as the "even half" of the $KU$-local algebraic $K$-theory spectrum $L_{KU}K(\Z[\zeta_N])$. In terms of homotopy groups, we have:
	\begin{prop}
		When $k\ge 1$, the induced of $h(N)$ on the $k$-th homotopy groups is an injection when $2$ is inverted.
	\end{prop}
	\begin{proof}
		When $p>2$, the HFPSS 
		\begin{equation*}
			E_2^{s,2t}=H^s_c(\Zpx;(\Kp)_{t}(K(\Z[\zeta_N])))\Longrightarrow \pi_{2t-s}L_{KU/p}(K(\Z[\zeta_N]))
		\end{equation*}
		collapses on $E_2$-page since it is concentrated in the $0$- and $1$-lines. One  extension problem on the $E_\infty$-page is 
		\begin{align*}
			0\to &H_c^1(\Zpx; (\Kp)_{2t}(K(\Z[\zeta_N])))\to \pi_{2t-1}(L_{KU/p}(K(\Z[\zeta_N])))\to H_c^0(\Zpx; (\Kp)_{2t-1}(K(\Z[\zeta_N])))\to 0
		\end{align*} 
		By \Cref{prop:hN_even}, $h(N)_*\colon  (\Kp)_{2t}(J(N))\simto (\Kp)_{2t}(K(\Z[\zeta_N]))$ is an isomorphism of $\Zpx$-modules. So we have an injection:
		\begin{align*}
			h(N)_*\colon \pi_{2t-1}(J(N)^{\wedge}_{p})=H_c^1(\Zpx; (\Kp)_{2t}(J(N)))\simto& H_c^1(\Zpx; (\Kp)_{2t}(K(\Z[\zeta_N])))\\\hookrightarrow& \pi_{2t-1}(L_{KU/p}(K(\Z[\zeta_N]))).
		\end{align*}
		By \Cref{prop:pi_jn}, $\pi_{2t}(J(N)^\wedge_{p})=0$ when $p>2$ and $t>0$. So the above argument shows $h(N)_*:\pi_{k}(J(N)^{\wedge}_{p})\to \pi_{k}(L_{KU/p}K(\Z[\zeta_N]))$ is an injection when $p>2$ and $k>0$. Since homotopy groups of $J(N)$ are all torsion groups in positive degrees, $h(N)_*$ is also an injection rationally. Consequently, we have shown $h(N)_*\colon  \pi_k(J(N))\to \pi_{k}(L_{KU}K(\Z[\zeta_N]))$ is an injection when $2$ is inverted.
	\end{proof}
	When $n\ge 1$, the image of $h(N)_*$ in $\pi_{n}(L_{KU}K(\Z[\zeta_N]))\simeq K_{n}(\Z[\zeta_N])$ (the two are isomorphic by \Cref{thm:Waldhausen}) is called the \textbf{Harris-Segal summand}, introduced in \cite{Harris-Segal_Ki}. Spectral liftings of the Harris-Segal summands at each prime have been studied by Kahn in \cite{Kahn_Bott_elements}. In this sense, our construction of the $J(N)$ spectrum is a global version of Kahn's $J(\ell,\Delta)$ spectrum. We end this subsection with a construction of a $J$-spectrum $J(\Kbb)$ attached to an abelian extension $\Kbb/\Q$. Similar to \Cref{con:jn}, we define $J(\Kbb)$ to be the homotopy pullback of the following arithmetic fracture square:
	\begin{equation*}
		\begin{tikzcd}
			J(\Kbb)\rar\dar\arrow[dr, phantom, "\lrcorner", very near start]&\prod_{p}\left(\Kp\right)^{hG_p(\Kbb)}\dar["L_\Q"]\\
			S^0_\Q\rar["h_\Q"']&\left(\prod_{p}\left(\Kp\right)^{hG_p(\Kbb)}\right)_\Q
		\end{tikzcd},
	\end{equation*}
	where $G_p(\Kbb)$ is as in \Cref{thm:KpKalg}. When $\Kbb=\Q(\zeta_N)$, we recover $J(N)$. 
	\begin{thm}\label{thm:h_K} The spectrum $J(\Kbb)$ is a $KU$-local $\einf$-ring spectrum with a $\gal(\Kbb/\Q)$-action by $\einf$-ring maps. Moreover, there is a $\gal(\Kbb/\Q)$-equivariant map $h(\Kbb)\colon  J(\Kbb)\to L_{KU}K(\Ocal_\Kbb)$ of $KU$-local $\einf$-ring spectrum such that
		\begin{enumerate}
			\item When $\Kbb=\Q(\zeta_N)$, $h(\Kbb)=h(N)$. In particular, $h(\Q)=h(1)$ is the $KU$-local Hurewicz map $h_{KU}:S^0_{KU}\to L_{KU}K(\Z)$.
			\item The map $h(\Kbb)$ is natural with respect to field extensions.
			\item The induced maps of $h(\Kbb)$ on even degree $\Kp$-homology groups are isomorphisms. (The induced maps on odd-degree $\Kp$-homology groups are zero since $(\Kp)_*(J(\Kbb))$ is zero in odd degrees.)
			\item When $t>1$, the image of $\pi_{2t-1}(h(\Kbb))$ is the Harris-Segal summand in $\pi_{2t-1}L_{KU}K(\Ocal_\Kbb)\simeq K_{2t-1}(\Ocal_\Kbb)$.
		\end{enumerate}
	\end{thm}
	\begin{rem}
		Let $N$ be the smallest integer such that $\Kbb\subseteq \Q(\zeta_N)$. Then the two spectra $J(\Kbb)$ and $J(N)^{h\gal(\Q(\zeta_N)/\Kbb)}$ are not equivalent in general. For example, take $\Kbb=\Q(\sqrt{5})$. Then $N=5$ and $\gal(\Q(\zeta_5)/\Kbb)=C_2$, generated by complex conjugation. One can check $J(\Kbb)^\wedge_2\simeq S^0_{KU/2}$ since $G_2(\Kbb)=\Z_2^\x$. However, $(J(5)^{hC_2})^\wedge_{2}\simeq \left(S^0_{KU/2}\right)^{hC_2}$, where $C_2$ acts trivially. Similar to the proof of \Cref{prop:jn_not_gal}, this homotopy fixed point spectrum is equivalent to $S^0_{KU/2}\vee S^0_{KU/2}$. 
	\end{rem}
	\subsection{Brown-Comenetz duality}\label{subsec:BC_dual}
	From \Cref{rem:jn_duality} and the computations in \Cref{Sec:pi_Dirichlet_j} and \Cref{thm:pi_dirichlet_j}, we observe the following duality phenomena hold in many (but not all) cases:
	\begin{equation*}
		\pi_t(J(N))\simeq \pi_{-2-t}(J(N)),\quad \pi_t\left(J(N)^{h\chi}\right)\simeq \pi_{-2-t}\left(J(N)^{h\chi^{-1}}\right).
	\end{equation*}
	This duality resembles the functional equations of the Dedekind $\zeta$-functions and Dirichlet $L$-functions. Let $\chi\colon \znx\to \Cx$ be a primitive Dirichlet character of conductor $N$ and $k$ is a positive integer such that $(-1)^k=\chi(-1)$. Then we have the following functional equation of $L(k;\chi)$:
	\begin{equation*}
	L(k;\chi)=\frac{\tau(\chi)}{2(k-1)!}\cdot\left(\frac{2\pi i}{N}\right)^k\cdot L(1-k;\chi^{-1}), \text{ where }\tau(\chi)=\sum_{a=1}^N\chi(a)e^{\frac{2\pi i a}{N}}.
	\end{equation*}
	The duality in homotopy groups we observed is a result of \textbf{Brown-Comenetz duality} of the spectra. In this subsection, we find the Brown-Comenetz duals for the Dirichlet $J$-spectra, $K(1)$-local spheres, and the spectra $J(N)$. Let's first review the setup following \cite{rigid}.	
	\begin{defns}
		Let $I$ be the spectrum that represents the cohomology theory $X\mapsto \hom_{\Z}(\pi_0(X),\Q/\Z)$. The Brown-Comenetz dual of a finite type spectrum $X$ is defined to be
		\begin{equation*}IX= \map(X,I)\simeq I\wedge DX,\end{equation*} 
		where $DX=\map(X,S^0)$ is the \textbf{Spanier-Whitehead dual} of $X$. In particular, $I=IS^0$ is the Brown-Comenetz dual of $S^0$.
	\end{defns}
		It follows from the definition that $\pi_t(IX)\simeq \hom_{\Z}(\pi_{-t}(X),\Q/\Z)$.
	\begin{defn}
		Let $I_1=\map(L_1S^0,I)$ be the Brown-Comenetz dual of the $E(1)$-local sphere $L_1S^0$. The $K(1)$-local Brown-Comenetz dual of $S^0$ is defined by $I_{K(1)}=L_{K(1)}I_1$. Define the $E(1)$-local and $K(1)$-local duals of a spectrum $X$ to be
		\begin{equation*}
			I_1X=\map(X,I_1)\simeq I_1\wedge D_{E(1)}X,\qquad I_{K(1)}X=\map(X,I_{K(1)})\simeq I_{K(1)}\wedge_{K(1)} D_{K(1)}(X).
		\end{equation*}
	\end{defn}
	\begin{prop}
		Homotopy groups of $I_1X$ and $I_{K(1)}X$ are computed by:
		\begin{align}
			\pi_t(I_1X)\simeq&\hom_{\Z_{(p)}}\left(\pi_{-t}(L_1X),\Q/\Z_{(p)}\right),\label{eqn:BC_dual_E1}\\
			\pi_t(I_{K(1)}X)\simeq&\hom_{\Zp}(\pi_{-t}(M_1X),\Qp/\Zp),\label{eqn:BC_dual_K1}
		\end{align}
		where $M_1X=\mathrm{hoFib}(L_1X\to L_0 X)$.
	\end{prop}
	\begin{proof}
		\eqref{eqn:BC_dual_E1} follows from the definition. \eqref{eqn:BC_dual_K1} is in \cite[84]{rigid}.
	\end{proof}
	\begin{thm}[Hopkins, {\cite[Remark 1.5]{Devinatz_KTGH}}, {\cite[Corollary 9.6]{MR_BC_dual_ASS}}]\textup{}\label{thm:E1_BC_dual}
		When $p>2$, $I_1\simeq \Sigma^2 L_1\left(S^0_p\right)$, where $S^0_p$ is the $p$-complete sphere. When $p=2$, $I_1\simeq \Sigma^2 L_1\left(\Ecal^\wedge_{2}\right)$, where $\Ecal$ is the finite CW-spectrum defined by
		\begin{equation*}
		\Ecal=\Sigma^{-2}(S^{-1}\cup_2 e^0\cup_{\eta}e^2).
		\end{equation*}
	\end{thm}
	\begin{rem}
		Bousfield localization at $E(1)$ does NOT commute with $p$-completion. On one hand, we have $L_1\left(S^0_p\right)\simeq M(\Zp)\wedge L_1S^0$,  since $L_1$ is smashing by \cite[Theorem 7.5.6]{orange}. One the other hand, $\left(L_1S^0\right)^\wedge_{p}\simeq S^0_{K(1)}$. We can then show $M(\Zp)\wedge L_1S^0\not\simeq S^0_{K(1)}$ by comparing their homotopy groups.
	\end{rem}
	Localized at an odd prime $p$, the spectrum $\Ecal$ is equivalent to the sphere spectrum. As a result, the formula $I_1\simeq \Sigma^2L_1(\Ecal^\wedge_{p})$ holds for all primes $p$.  This suggests:
	\begin{cor}\label{cor:BC_dual_K}
		$I_{KU}=\map(L_{KU} S^0,I)\simeq \Sigma^2L_{KU}(\Ecal^\wedge)$, where $\Ecal^\wedge$ is the profinite completion of $\Ecal$.
	\end{cor}
	\begin{thm}[Gross-Hopkins]\label{thm:K1_BC_dual}
		When $p>2$, $I_{K(1)}\simeq S^2_{K(1)}$. When $p=2$, $I_{K(1)}\simeq \Sigma^2\Ecal_{K(1)}$ .
	 \end{thm}
 	The Dirichlet $J$-spectra and $K(1)$-local spheres constructed in this paper are $KU$-local and $K(1)$-local, respectively. The analysis above shows:
 	\begin{align*}
 		I_{KU}\left(J(N)^{h\chi} \right)\simeq& \Sigma^2 D_{KU}\left(J(N)^{h\chi}\right)\wedge L_{KU}(\Ecal^\wedge)\\
 		I_{K(1)}\left(S^0_{K(1)}(p^v)^{h\chi}\right)\simeq&\left\{\begin{array}{cl}
 		\Sigma^2D_{K(1)}\left(S^0_{K(1)}(p^v)^{h\chi}\right),&p>2;\\
 		\Sigma^2 D_{K(1)}\left(S^0_{K(1)}(p^v)^{h\chi}\right)\wedge_{K(1)} \Ecal_{K(1)},&p=2.
 		\end{array}\right.
 	\end{align*} 	
 	To find Brown-Comenetz duals of these spectra, it now remains to identify their Spanier-Whitehead duals in $\Sp_{KU}$ and $\Sp_{K(1)}$, respectively. We start with the $K(1)$-local cases, which the $KU$-local cases depend on.
 	\begin{prop}\label{prop:D_k1_SW_dual}
 		Let $\chi$ be a $p$-adic Dirichlet character of conductor $p^v$. The Spanier-Whitehead dual of the Dirichlet $K(1)$-local sphere attached to $\chi$ is
 		\begin{equation*}
 			D_{K(1)}\left(S^0_{K(1)}(p^v)^{h\chi}\right)\simeq\left\{\begin{array}{cl}
 			S^0_{K(1)}(2^v)^{h\chi^{-1}}\wedge_{K(1)}\Ecal_{K(1)},& p=2 \textup{ and }\chi(-1)=-1;\\
 			S^0_{K(1)}(p^v)^{h\chi^{-1}},&\textup{otherwise.}
 			\end{array}\right. 
 		\end{equation*}
 	\end{prop}
 	\begin{proof}
 		When $p>2$,  by \Cref{lem:Sp_k1_alg_p}, it suffices to check the $\Zpx$-action on the $\Kp$-homology of the Dirichlet $K(1)$-local spheres. In \Cref{cor:Kp_Dk1}, we computed:  
 		\begin{equation*}
 			\left(\Kp\right)_*\left(S^0_{K(1)}(p^v)^{h\chi}\right)\simeq \hom_{\Zp}\left(\Zp[\chi],\left(\Kp\right)_*\right),
 		\end{equation*}
 		where $\Zpx$ acts on $\Zp[\chi]$ through the character $\chi$, and on $\left(\Kp\right)_*$ by the Adams operations. The dual $\Zpx$-representation of $\Zp[\chi]$ is $\Zp[\chi^{-1}]$. Also notice $\left(\Kp\right)_*$ is self dual as graded $\Zpx$-representations. We have:
 		\begin{equation*}
 			\hom_{\left(\Kp\right)_*}\left(\left(\Kp\right)_*\left(S^0_{K(1)}(p^v)^{h\chi}\right),\left(\Kp\right)_*\right)\simeq \hom_{\Zp}\left(\Zp[\chi^{-1}],\left(\Kp\right)_*\right)\simeq\left(\Kp\right)_*\left(S^0_{K(1)}(p^v)^{h\chi^{-1}}\right).
 		\end{equation*}
 		This implies $D_{K(1)}\left(S^0_{K(1)}(p^v)^{h\chi}\right)\simeq S^0_{K(1)}(p^v)^{h\chi^{-1}}$ when $p>2$.
 		
 		When $p=2$, by \Cref{lem:Sp_k1_alg_2}, we need to check the $(1+4\Z_2)$-action on the $KO^\wedge_2$-homology of the Dirichlet $K(1)$-local spheres. By \Cref{cor:KO_Dk1}, we have
 		\begin{equation*}
 		\left(KO^\wedge_{2}\right)_*\left(S^0_{K(1)}(2^v)^{h\chi}\right)\simeq\left\{\begin{array}{cl}
 		\hom\left(\Z_2[\chi],\left(KO^\wedge_{2}\right)_*\right), & \chi(-1)=1;\\
 		\hom\left(\Z_2[\chi],\left(KU^\wedge_{2}\right)^{h\omega}_*\right), & \chi(-1)=-1.\\
 		\end{array}\right.
 		\end{equation*} 
 		Here $1+4\Z_2$ acts by the character $\chi|_{\Z/2^{v-2}}$ on $\Z_2[\chi]$ and by Adams operations on $\left(KO^\wedge_{2}\right)_*$ and $\left(KU^\wedge_{2}\right)^{h\omega}_*$. Notice that $\Z_2[\chi]$ is dual to $\Z_2[\chi^{-1}]$, $\left(KO^\wedge_{2}\right)_*$ is self-dual, and $\left(KU^\wedge_{2}\right)^{h\omega}_*$ is dual to $\left(KU^\wedge_{2}\right)^{h'\omega}_*$, where $\left(KU^\wedge_{2}\right)^{h'\omega}=\map\left(S^{1-\sigma},\KR^\wedge_{2}\right)^{h\{\pm 1\}}$. From this, we get
 		\begin{align*}
 			\hom_{\left(KO^\wedge_{2}\right)_*}\left(\left(KO^\wedge_{2}\right)_*\left(S^0_{K(1)}(2^v)^{h\chi}\right),\left(KO^\wedge_{2}\right)_*\right)
 			\simeq& \left\{\begin{array}{cl}
 			\hom\left(\Z_2[\chi^{-1}],\left(KO^\wedge_{2}\right)_*\right), & \chi(-1)=1;\\
 			\hom\left(\Z_2[\chi^{-1}],\left(KU^\wedge_{2}\right)^{h'\omega}_*\right), & \chi(-1)=-1.\\
 			\end{array}\right.\\
 			\simeq&\left\{\begin{array}{cl}
 			\left(KO^\wedge_{2}\right)_*\left(S^0_{K(1)}(2^v)^{h\chi^{-1}}\right),& \chi(-1)=1;\\
 			\left(KO^\wedge_{2}\right)_*\left(S^0_{K(1)}(2^v)^{h\chi^{-1}}\wedge_{K(1)}\Ecal_{K(1)}\right),& \chi(-1)=-1.\\
 			\end{array}\right.
 		\end{align*}
 		This implies the claim at $p=2$.
 	\end{proof}
 	\begin{thm}\label{thm:D_k1_duality_pv}
 		When the conductor of $\chi$ is a power of $p$, the Brown-Comenetz dual of the Dirichlet $K(1)$-local sphere attached to $\chi$ is:
 		\begin{equation*}
 		I_{K(1)}\left(S^0_{K(1)}(p^v)^{h\chi}\right)\simeq\left\{\begin{array}{cl}
 		\Sigma^2 S^0_{K(1)}(2^v)^{h\chi^{-1}}\wedge_{K(1)}\Ecal_{K(1)},& p=2 \textup{ and }\chi(-1)=1;\\
 		\Sigma^2 S^0_{K(1)}(p^v)^{h\chi^{-1}},&\textup{otherwise.}
 		\end{array}\right. 
 		\end{equation*}
 	\end{thm}
 	\begin{proof}
 		We used the fact $\Ecal_{K(1)}\wedge\Ecal_{K(1)}\simeq S^0_{K(1)}$ in the case when $p=2$ and $\chi(-1)=-1$.
 	\end{proof}
 	From computations in \Cref{Sec:pi_Dirichlet_j}, we know $L_0\left(S^0_{K(1)}(p^v)^{h\chi}\right)\simeq *$ whenever the conductor $N$ of $\chi$ is a power of $p$.  This implies $M_1S^0_{K(1)}(p^v)^{h\chi}=S^0_{K(1)}(p^v)^{h\chi}$. Also, as the homotopy groups of the Dirichlet $K(1)$-local spheres are finite $p$-groups, they are (non-canonically) isomorphic to their $p$-adic Pontryagin duals. Now plugging \Cref{thm:D_k1_duality_pv} into \eqref{eqn:BC_dual_K1}, we get
 	\begin{align}
 		\pi_{t}\left(S^0_{K(1)}(p^v)^{h\chi}\right)\simeq&
 		\hom_{\Zp}(\pi_{t}\left(S^0_{K(1)}(p^v)^{h\chi}\right),\Qp/\Zp)\qquad\textup{(non-canonically)}\nonumber\\
 		\simeq&\hom_{\Zp}(\pi_{t}\left(M_1S^0_{K(1)}(p^v)^{h\chi}\right),\Qp/\Zp)\nonumber\\
 		\simeq&\left\{\begin{array}{cl}
 		\pi_{-2-t}\left( S^0_{K(1)}(2^v)^{h\chi^{-1}}\wedge_{K(1)}\Ecal_{K(1)}\right),& p=2 \textup{ and }\chi(-1)=1;\\
 		\pi_{-2-t}\left( S^0_{K(1)}(p^v)^{h\chi^{-1}}\right),&\textup{otherwise.}
 		\end{array}\right. \nonumber
 	\end{align}
 	
 	When the conductor $N$ of the $p$-adic Dirichlet character $\chi$ is not a $p$-power, the Brown-Comenetz dual of the Dirichlet $K(1)$-local is slightly different.
 	\begin{cor}\label{cor:D_k1_duality_not_pv}
 	Let $\chi$ be a $p$-adic Dirichlet character of conductor $N=p^v\cdot N'$ . Write $\chi=\chi_p\cdot\chi'$ as before. If $|\imag \chi'|>1$ is not a $p$-power, then
 	\begin{equation*}
 		I_{K(1)}\left(S^0_{K(1)}(p^v)^{h\chi}\right)\simeq I_{K(1)}(*)\simeq *.
 	\end{equation*}
 	If $|\imag \chi'|>1$ is a $p$-power, then
 	\begin{equation*}
 	I_{K(1)}\left(S^0_{K(1)}(p^v)^{h\chi}\right)\simeq\left\{\begin{array}{cl}
 	S^0_{K(1)}(2^v)^{h\chi^{-1}}\wedge_{K(1)}\Ecal_{K(1)},& p=2 \textup{ and }\chi|_{\zx{4}}\textup{ is trivial};\\
 	S^0_{K(1)}(p^v)^{h\chi^{-1}},&\textup{otherwise.}
 	\end{array}\right. 
 	\end{equation*}
 	\end{cor}
 	\begin{proof}
 	 	When $|\imag \chi'|>1$ is not a $p$-power, the claim follows from \Cref{prop:dirichlet_k1_contractible}. The other cases follow from \Cref{thm:Dk1_ho} and \Cref{thm:D_k1_duality_pv}.
 	\end{proof}
	Now we identify the $KU$-local Spanier-Whitehead dual of the Dirichlet $J$-spectra $J(N)^{h\chi}$. Similar to \eqref{eqn:BC_dual_E1}, we have in $\Sp_{KU}$:
	\begin{equation}\label{eqn:BC_dual_K}
	\pi_t(I_{KU} X)\simeq \hom_{\Z}\left(\pi_{-t}(L_{KU} X), \Q/\Z\right).
	\end{equation}
	In this case, our strategy is to assemble duality formulas via the arithmetic fracture squares for $KU$-local spectra. By \Cref{prop:JN_hchi_Q_contractible}, $J(N)^{h\chi}_\Q$ is contractible unless $\chi$ is trivial. In \Cref{prop:jnchi_p_decomposition}, we described how $J(N)^{h\chi}$ decomposes upon $p$-completion. Now from \Cref{thm:D_k1_duality_pv} and \Cref{cor:D_k1_duality_not_pv}, we have:
	\begin{thm}\label{thm:BC_dual_D_J}
		Let $\chi\colon \znx\to \Cx$ be a primitive Dirichlet character of conductor $N$.
		\begin{enumerate}
			\item Suppose $N=p^v$ and $p$ is odd. If $|\imag\chi|>1$ is not a power of another prime (in particular whenever $v>1$), then
			\begin{equation*}
			I_{KU}\left(J(p^v)^{h\chi}\left[\frac{1}{\ell(\chi)}\right]\right)\simeq \Sigma^2J(p^v)^{h\chi^{-1}}\left[\frac{1}{\ell(\chi)}\right]\quad\implies \quad \pi_t\left(J(p^v)^{h\chi}\left[\frac{1}{\ell(\chi)}\right]\right)\simeq \pi_{-2-t}\left(J(p^v)^{h\chi^{-1}}\left[\frac{1}{\ell(\chi)}\right]\right),
			\end{equation*}
			where $\ell(\chi)$ is as in \Cref{thm:dirichlet_j_gbn}:
			\begin{equation*}
				\ell(\chi)=\left\{\begin{array}{cl}
				\ell,&\textup{if }|\imag(\chi)|\text{ is a power of a prime }\ell\neq p;\\
				1,&\textup{otherwise.}
				\end{array}\right.
			\end{equation*}
			\item If $N=2^v\ge 4$, then
			\begin{align*}
				&I_{KU}\left(J(2^v)^{h\chi}\right)\simeq\left\{\begin{array}{cl}
				\Sigma^2J(2^v)^{h\chi^{-1}}\wedge \Ecal_{KU},& \chi(-1)=1;\\
				\Sigma^2J(2^v)^{h\chi^{-1}},& \chi(-1)=-1.
				\end{array}\right.\\
				\implies&\pi_t\left(J(2^v)^{h\chi}\right)\simeq\left\{\begin{array}{cl}
				\pi_{-2-t}\left(J(2^v)^{h\chi^{-1}}\wedge \Ecal_{KU}\right),& \chi(-1)=1;\\
				\pi_{-2-t}\left(J(2^v)^{h\chi^{-1}}\right),& \chi(-1)=-1.
				\end{array}\right.
			\end{align*}
		\end{enumerate} 
	\end{thm}
	\begin{rem}
		In (1) above, it is necessary to invert $\ell(\chi)$. This is because the degrees of suspensions are different in \Cref{thm:D_k1_duality_pv} and \Cref{cor:D_k1_duality_not_pv}.
	\end{rem}
	We now identify the Brown-Comenetz dual of $J(N)$. It is $KU$-local by \Cref{prop:JN_k_loc}.  This means $I_{KU}(J(N))\simeq \Sigma^2L_{KU}\left(\Ecal^\wedge\right)\wedge D_{KU}(J(N))$ by \Cref{cor:BC_dual_K}.
	\begin{prop}\label{prop:jn_dual}
		$J(N)$ is Spanier-Whitehead self-dual in $\Sp_{KU}$. It follows that 
		\begin{equation*}
			I_{KU}(J(N))\simeq \Sigma^2L_{KU}\left(\Ecal^\wedge\right)\wedge J(N)\simeq\Sigma^2\Ecal_{KU}\wedge J(N)\wedge M(\wh{\Z}).
		\end{equation*}
	\end{prop}
	\begin{proof}
		This is because $J(N)_\Q\simeq S^0_\Q$ and $J(N)^\wedge_{p}\simeq S^0_{K(1)}\left(p^{v_p(N)}\right)$ are both Spanier-Whitehead self-dual in $\Sp_\Q$ and $\Sp_{K(1)}$, respectively.
	\end{proof}
	\begin{lem}
		$J(4N)\wedge\Ecal_{KU}\simeq J(4N)$. 
	\end{lem}
	\begin{proof}
		Since $\Ecal\simeq S^0$ when $2$ is inverted, the equivalence holds rationally, and when completed at an odd prime. At prime $2$, $J(4N)^\wedge_{2}\simeq S^0_{K(1)}(2^{v+2})$ by \Cref{cor:jn_structure}, where $v=v_2(N)$. The claim now follows from the fact that $S^0_{K(1)}(4)\wedge_{K(1)} \Ecal_{K(1)}\simeq S^0_{K(1)}(4)$. 
	\end{proof}
	\begin{cor}\label{cor:J_4N_BC_dual}
		There is an equivalence of spectra: $I_{KU}(J(4N))\simeq \Sigma^2J(4N)\wedge M(\wh{\Z})$. It follows from \mbox{\eqref{eqn:BC_dual_K}} and the universal coefficient theorem that we have an isomorphism of homotopy groups:
		\begin{equation*}
		\pi_t(J(4N))^\wedge\simeq \pi_t(J(4N)\wedge M(\wh{\Z}))
		\simeq \pi_{t+2}(I_{KU}(J(4N)))
		\simeq\hom_{\Z}(\pi_{-2-t}(J(4N)),\Q/\Z).
		\end{equation*}
		This was observed in \Cref{rem:jn_duality}.
	\end{cor}
	
	\appendix
	\section{Cyclotomic representations of cyclic groups}\label{appen:cyclo_rep}
	In the appendix, we study the integral and $p$-adic cyclotomic representations of the cyclic group $C_n$.
	\subsection{Integral cyclotomic representations}
	Let $\Phi_n(t)$ be the $n$-th cyclotomic polynomial, i.e. the minimal polynomial of a primitive $n$-th root of unity $\zeta_n$ over $\Q$. The integral cyclotomic representation of $C_n$ has underlying abelian group $\Z[\zeta_n]\simeq \Z[t]/\Phi(t)$ and $g\in C_n$ acts by multiplication by a primitive $n$-th root of unity (or $t\in  \Z[t]/\Phi(t)$). The rank of $\Z[\zeta_n]$ as a free abelian group is equal to $\deg \Phi_n(t)=\phi(n)$.
	\begin{exmps} We consider the following examples:
		\begin{enumerate}
			\item When $n=5$, $\Z[\zeta_5]$ is a free $\Z$-module of rank $4$ as $\phi(5)=4$.  $\{1,\zeta_5,\zeta_5^2,\zeta_5^3\}$ form a basis of $\Z[\zeta_5]$. The minimal polynomial of $\zeta_5$ is $\Phi_5(t)=t^4+t^3+t^2+t+1$. Let $g\in C_5$ be a generator that acts on $\Z[\zeta_5]$ by multiplication by $\zeta_n$. Then the matrix representation of $g\in C_5$ with respect the basis $\{1,\zeta_5,\zeta_5^2,\zeta_5^3\}$ of $\Z[\zeta_5]$ is 
			\begin{equation*}
			g=\begin{pmatrix}
			&&&-1\\
			1&&&-1\\
			&1&&-1\\
			&&1&-1
			\end{pmatrix}.
			\end{equation*}
			\item When $n=6$, $\Z[\zeta_6]$ is a free $\Z$-module of rank $2$ as $\phi(6)=2$.  $\{1,\zeta_6\}$ form a basis of $\Z[\zeta_6]$. The minimal polynomial of $\zeta_6$ is $\Phi_6(t)=t^2-t+1$. Let $g\in C_6$ be a generator that acts on $\Z[\zeta_6]$ by multiplication by $\zeta_n$. Then the matrix representation of $g\in C_6$ with respect the basis $\{1,\zeta_6\}$ of $\Z[\zeta_6]$ is 
			\begin{equation*}
			g=\begin{pmatrix}
			0&1\\1&-1
			\end{pmatrix}.
			\end{equation*}
		\end{enumerate}
	\end{exmps}
	\begin{lem}\label{lem:cyclo_decomp}
		The cyclotomic representation of $C_n$ is equivalent to the external tensor product of the cyclotomic representations of $C_{p^{v_{p}(n)}}$, i.e. there is an equivalence of $C_n$-representations:
		\begin{equation*}
		\Z[\zeta_n]\simeq \bigotimes_{p\mid n} \Z\left[\zeta_{p^{v_{p}(n)}}\right]
		\end{equation*}
	\end{lem}
	\begin{lem}\label{lem:cyclo_cyclic_res}
		There is a short exact sequence of $C_{p^v}$-representations:
		\begin{equation}\label{eqn:Cp_cyclo}
		\begin{tikzcd}
		0\rar&\Z[\zeta_{p^v}]\rar&\Z[C_{p^v}]\rar&\Z[C_{p^{v-1}}]\rar&0
		\end{tikzcd}
		\end{equation}
		where $C_{p^v}$ acts on $\Z[C_{p^{v-1}}]$ via the quotient map $C_{p^v}\twoheadrightarrow C_{p^{v-1}}$. 
	\end{lem}
	\begin{proof}
		From the discussion above, we have $\Z[\zeta_{n}]=\Z[t]/\Phi_n(t)$. When $n=p^v$ is a power of prime, $\Phi_{p^v}(t)=\frac{t^{p^v}-1}{t^{p^{v-1}}-1}$. This means we have a short exact sequence of $C_{p^v}$-representations:
		\begin{equation*}
			0\longrightarrow \Z[\zeta_{p^v}]\longrightarrow \Z[t]/(t^{p^v}-1)\longrightarrow \Z[t]/(t^{p^{v-1}}-1)\longrightarrow 0
		\end{equation*}		
		The claim now follows once we identify $\Z[t]/(t^n-1)$ with $\Z[C_n]$ as a $C_n$-representation for any $n$.
	\end{proof}	
	\subsection{$p$-adic cyclotomic representations}
	From now on, let $\chi\colon \znx\to \Cpx$ be a $p$-adic Dirichlet character of conductor $N$ and $\Zp[\chi]$ be the $\Zp$-subalgebra of $\Cp$ generated by the image of $\chi$. Again, $\Zp[\chi]=\Zp[\zeta_n]$ for some $n$. Write $n=p^v\cdot n'$ with $p\nmid n'$, we have $\Zp[\zeta_n]\simeq \Zp[\zeta_{p^v}]\otimes_{\Zp}\Zp[\zeta_{n'}]$. Now it suffices to analyze $C_n$-actions on $\Zp[\zeta_n]$ in the cases when $n=p^v$ or $p\nmid n$. Let's first recall some basic facts of cyclotomic extensions of $\Q$:
	\begin{lem}[{\cite[Theorem 2.5, 2.6]{Washington_cyclotomic}}] \label{lem:gal_Q_zeta}
		We recall the following basic facts of the cyclotomic extension $\Q(\zeta_{n})/\Q$.
		\begin{enumerate}
			\item $\Q(\zeta_{n})/\Q$ is a Galois extension of degree $\phi(n)$ and $\gal(\Q(\zeta_{n})/\Q)\simeq \zx{n}$, with $a\in \zx{n}$ acting by $\zeta_{n}\mapsto \zeta_{n}^a$.
			\item The ring of integers of $\Q(\zeta_n)$ is $\Z[\zeta_n]$. Consequently, for any $\sigma\in \gal(\Q(\zeta_n)/\Q)$, $\sigma(\Z[\zeta_n])=\Z[\zeta_n]$.
		\end{enumerate}		
	\end{lem}
	As a result of this lemma, we can extract the action of $\znx$ on $\Z[\zeta_n]$ from that on $\Q(\zeta_n)$.
	\begin{prop}
		For any $\sigma\in \gal(\Q(\chi)/\Q)$, the $\znx$-representation induced by the Dirichlet character $\sigma\circ \chi$ is isomorphic to that induced by $\chi$. 
	\end{prop}
	\begin{proof}
		Let $\Z[\chi]=\Z[\zeta_n]$, where $\zeta_n$ is a primitive $n$-th root of unity. For any $\sigma\in \gal(\Q(\chi)/\Q)$, $\sigma(\zeta_n)$ is also a primitive $n$-th root of unity. As a result, the minimal polynomials of $\zeta_n$ and $\sigma(\zeta_n)$ are both $\Phi_n(t)$. It follows that the matrix representations of $\chi$ and $\sigma\circ \chi$ are differed by a change of basis induced by $\sigma$. Thus, the integral representations induced by $\chi$ and $\sigma\circ \chi$ are isomorphic. 
	\end{proof}
	\begin{prop} \label{prop:gal_Qp_zeta}
		Write $n=p^v\cdot n'$, where $p\nmid n'$ and let $m$ be the multiplicative order of $p$ mod $n'$, i.e. 
		\begin{equation*}
		m=\min\{k>0 \mid p^k\equiv 1\mod n'\}.
		\end{equation*}
		Then $\Qp(\zeta_n)/\Qp$ is a Galois extension of local fields of residue index $m$ and ramification index $\phi(p^v)$. Moreover, 
		\begin{equation*}
		\gal(\Qp(\zeta_n)/\Qp)\simeq \gal(\Qp(\zeta_{n'})/\Qp)\x\gal(\Qp(\zeta_{p^v})/\Qp)\simeq (\Z/m)\x \zx{p^v},
		\end{equation*}
		where a generator $\varphi\in \Z/m$ acts on $\Qp(\zeta_{n'})$ by the lift of the Frobenius ($p$-th power map) from $\Zp[\zeta_{n'}]/(p)\simeq \mathbb{F}_{p^m}$ to $\Qp(\zeta_{n'})\simeq \W(\mathbb{F}_{p^m})$. In particular, $\varphi(\zeta_{n'})=\zeta^p_{n'}$.
	\end{prop}
	\subsection{$p$-completions of integral cyclotomic representations}
	We conclude this appendix with a discussion on how $\Z[\chi]$ decomposes upon $p$-completion.	The simplest case is
	\begin{cor}\label{cor:Qp_pv}
		$\Zp[\zeta_{p^v}]\simeq \Z[\zeta_{p^v}]\otimes_\Z \Zp\simeq \left(\Z[\zeta_{p^v}]\right)^\wedge_{p}.$
	\end{cor}
	\begin{proof}
		By \Cref{prop:gal_Qp_zeta}, $\Qp(\zeta_{p^v})/\Qp$ is a totally ramified extension of local fields of rank $\phi(p^v)$. This means $\Zp[\zeta_{p^{v}}]$ is a free $\Zp$-module of rank $\phi(p^v)$, which is equal to the rank of $\Z[\zeta_{p^v}]$ as a free $\Z$-module. This implies $\Z[\zeta_{p^v}]$ does not split upon $p$-completion.
	\end{proof}
	Comparing \Cref{lem:gal_Q_zeta} and \Cref{prop:gal_Qp_zeta}, we have shown:
	\begin{prop}\label{prop:gal_p_adic_cyclo}
		Fix an embedding $\iota\colon \Q[\zeta_n]\hookrightarrow \Cp$. For any $\sigma\in \gal(\Qp(\zeta_n)/\Qp)$, $\sigma\circ \iota(\Q(\zeta_{n}))=\iota(\Q(\zeta_n))$. In addition, the restriction map on the Galois group induced by $\iota$
		\begin{equation}\label{eqn:iota_star}
		\iota^*\colon  \gal(\Qp(\zeta_n)/\Qp)\longrightarrow \gal(\Q(\zeta_n)/\Q)
		\end{equation}
		is injective. More precisely, rewrite $\Q(\zeta_n)=\Q(\zeta_{p^v})\otimes_\Q \Q(\zeta_{n'})$ and $\iota=\iota_p\otimes \iota_{n'}$, where
		\begin{equation*}
		\iota_p \colon \Q(\zeta_{p^v})\hookrightarrow \Cp,\quad \iota_{n'}\colon \Q(\zeta_{n'})\hookrightarrow \Cp.
		\end{equation*}
		Then we have
		\begin{itemize}
			\item $\iota_p^*\colon  \gal(\Qp(\zeta_{p^v})/\Qp)\simto \gal(\Q(\zeta_{p^v})/\Q)$ is an isomorphism.
			\item $\iota_{n'}^*\colon  \gal(\Qp(\zeta_{n'})/\Qp)\hookrightarrow \gal(\Q(\zeta_{n'})/\Q)$ is the inclusion of the subgroup of $\zx{n'}$ generated by $p\in \zx{n'}$.
		\end{itemize}
	\end{prop}
	\begin{prop}\label{prop:Phi_n_fac}
		Pick a representative $\sigma\in \gal(\Q(\zeta_n)/\Q)$ for each coset in 
		\begin{equation*}
		\cok \iota^*=\gal(\Q(\zeta_n)/\Q)/\gal(\Qp(\zeta_n)/\Qp).
		\end{equation*}
		 $\Z[\zeta_n]\otimes \Zp$ decomposes as a $\Zp$-algebra by
		\begin{equation*}
		\Z[\zeta_n]\otimes_\Z\Zp\underset{\sim}{\xrightarrow{\prod (\iota\circ \sigma)\otimes 1}} \prod_{[\sigma]\in \cok \iota^*}\Zp[\zeta_n]\simeq \bigoplus_{[\sigma]\in \cok \iota^*}\Zp[\zeta_n].
		\end{equation*}
	\end{prop}
	\begin{proof}
		The minimal polynomial of $\zeta_n$ over $\Z$ is 
		\begin{equation*}
		\Phi_n(t)=\prod_{\sigma\in  \gal(\Q(\zeta_n)/\Q)}(t-\sigma(\zeta_n)).
		\end{equation*}
		We have an isomorphism $\Z[\zeta_n]\otimes_\Z\Zp\simeq \Zp[t]/(\Phi_n(t))$. Over $\Zp$, $\Phi_n(t)$ factors as 
		\begin{equation*}
		\Phi_n(t)=\prod_{[\sigma]\in \cok\iota^*}\Phi_{n,\sigma}(t), \quad \text{where }\Phi_{n,\sigma}(t)=\prod_{\tau\in \gal(\Qp(\zeta_n)/\Qp)}(t-\tau\circ\iota\circ\sigma(\zeta_n)).
		\end{equation*}
		For each $\sigma\in \gal(\Q(\zeta_n)/\Q)$, $\Phi_{n,\sigma}(t)$ is the minimal polynomial of $\iota\circ \sigma(\zeta_n)$ over $\Zp$. As $\Phi_{n,\sigma}(t)$ are coprime to each other for different cosets $[\sigma]\in \cok\iota^*$ and $\Zp[t]/(\Phi_{n,\sigma}(t))\simeq \Zp[\zeta_n]$ for all $\sigma$, the claim now follows from the Chinese Reminder Theorem.
	\end{proof}
	\begin{cor}\label{cor:cyclo_rep_p_decomp}
		Let $\chi\colon \znx\to \Cx$ be a Dirichlet character with $\Z[\chi]=\Z[\zeta_n]$. Then there is a decomposition of the $p$-adic $(\Z/N)^\times$-representation:
		\begin{equation*}
		\Z[\chi]\otimes_\Z \Zp\simeq \bigoplus_{[\sigma]\in \cok \iota^*} \Zp[\iota\circ \sigma\circ \chi],
		\end{equation*}
		where $\iota\circ \sigma\circ \chi$ is the $p$-adic Dirichlet character defined by
		\begin{equation*}
		\begin{tikzcd}
		\znx\rar["\chi"]&\left(\Z{[}\chi{]}\right)^{\x}\rar["\sigma"]&\left(\Z{[}\chi{]}\right)^{\x}\rar[hook,"\iota"]&\Cpx.
		\end{tikzcd}
		\end{equation*} 
	\end{cor}
	\begin{proof}
		This is done by forcing the isomorphism in \Cref{prop:Phi_n_fac} to be $\znx$-equivariant.
	\end{proof}
	\begin{cor}\label{cor:cyclo_rep_p_decomp_pv}
		When $\chi\colon \znx\to \Cx$ is a primitive Dirichlet character of conductor $N=p^v$ and $p>2$, there is an equivalence of  $\zx{p^v}$-representations:
		\begin{equation*}
			\Z[\chi]^\wedge_{p}\simeq \bigoplus_{\substack{0\le a\le p-2\\\ker \omega^a=\ker\chi|_{\zpx}}} \Zp[\chi_a],
		\end{equation*}
		where $\chi_a=\omega^a\cdot (\iota\circ\chi|_{\Z/p^{v-1}})$ and $\omega\colon \zpx\to\Zpx$ is the \Teichmuller character.
	\end{cor}
	\begin{proof}
		By \Cref{cor:cyclo_rep_p_decomp}, we need show the following two sets of characters are the same:
		\begin{equation}\label{eqn:two_sets_char}
			\{\iota\circ \sigma \circ \chi\mid [\sigma]\in \cok \iota^*\}=\{\omega^a\cdot (\iota\circ\chi|_{\Z/p^{v-1}})\mid 0\le a\le p-2, \ker \omega^a=\ker\chi|_{\zpx}\}.
		\end{equation}
		We first prove the $v=1$ case. A $p$-adic character of conductor $p$ is necessarily of the form $\omega^a$ for some $a$, since $\Zp$ contains all $(p-1)$-st roots of unity. As $\iota$ and $\sigma$ are injections, $\ker \iota\circ \sigma \circ \chi=\ker \chi$. Now it suffices to check the two sets have the same size. Since $\Zp[\iota\circ \chi]=\Zp$, we have $|\cok \iota^*|=|\gal(\Q(\chi)/\Q)|=\mathrm{rank}_{\Z}(\Z[\chi])$. The character $\chi$ factors as $\zpx\twoheadrightarrow C_{n'}\hookrightarrow (\Z[\zeta_{n'}])^\x$ for some $n'|(p-1)$. Then $\Z[\chi]$ has rank $\phi(n')$. Let $g\in\zpx$ be a generator, then $\ker \chi$ is the subgroup of $\zpx$ generated by $g^{n'}$. We have 
		\begin{equation*}
			\{a\mid 0\le a\le p-2, \ker \omega^a=\ker\chi=\langle g^{n'}\rangle\subseteq \zpx \}=\{a \mid 0\le a\le p-2, \text{ the order of $a\in\zpx$ is }(p-1)/n' \}.
		\end{equation*}
		The size of this set is $\phi(n')$, which is equal to $|\cok \iota^*|$, from which we conclude the two sets of characters in \eqref{eqn:two_sets_char} are the same when $v=1$.
		
		When $v>1$, write $\Z[\chi]=\Z[\chi|_{\zpx}]\otimes \Z[\chi|_{\Z/p^{v-1}}]$. The character $\chi$ being primitive implies $\chi|_{\Z/p^{v-1}}$ is injective and $\Z[\chi|_{\Z/p^{v-1}}]=\Z[\zeta_{p^{v-1}}]$. By \Cref{cor:Qp_pv}, $\Z[\chi|_{\Z/p^{v-1}}]^\wedge_{p}=\Zp[\iota\circ \chi|_{\Z/p^{v-1}}]$. On the other hand , write $\iota=\iota_{n'}\cdot\iota_p$ as in \Cref{prop:gal_p_adic_cyclo}, where $\iota_p\colon \Q(\zeta_{p^{v-1}})\hookrightarrow \Cp$ is a field extension. \Cref{prop:gal_p_adic_cyclo} says $\iota_p^*$ is an isomorphism, which implies $\cok \iota^*=\cok \iota_{n'}^*$. The  analysis above shows:
		\begin{align*}
			\Z[\chi]^\wedge_{p}\simeq& \Z[\chi|_{\zpx}]^\wedge_{p}\bigotimes_{\Zp} \Zp[\iota_p\circ \chi|_{\Z/p^{v-1}}]\\
			\bigoplus_{[\sigma]\in \cok \iota^*} \Zp[\iota\circ \sigma\circ \chi]\simeq&  \left(\bigoplus_{[\sigma]\in \cok \iota_{n'}^*} \Zp[\iota_{n'}\circ \sigma\circ \chi|_{\zpx}]\right)\bigotimes_{\Zp} \Zp[\iota_p\circ \chi|_{\Z/p^{v-1}}]
		\end{align*}
		Now we have reduced this case to the $v=1$ situation for the character $\chi|_{\zpx}$.
	\end{proof}
	\printbibliography
\end{document}